\documentclass[reqno]{amsart}
\usepackage[english]{babel}
\usepackage[utf8]{inputenc}
\usepackage[T1]{fontenc}
\usepackage{amsmath,amssymb,amsthm,mathtools,enumitem, dynkin-diagrams}
\usetikzlibrary{trees}
\usetikzlibrary{matrix}
\usepackage{mathrsfs}
\usepackage{lmodern,subcaption,bbm}
\usepackage[all,cmtip]{xy}
\usepackage[linktocpage=true, pdfusetitle,pagebackref]{hyperref}
\renewcommand*{\backref}[1]{}
\usepackage{array}
\usepackage{csquotes}
\usepackage{multirow}
\usepackage{todonotes}
\usepackage{verbatim,ifthen}

 \usepackage{url,tikz-cd}
  \usetikzlibrary{arrows.meta}
  \usetikzlibrary{calc}
  \usetikzlibrary{matrix}

\calclayout
\newcolumntype{C}{>{$}c<{$}}

\usepackage{xcolor}
\hypersetup{
    colorlinks,
    linkcolor={red!80!black},
    citecolor={green!50!black},
    urlcolor={blue!80!black}
}
\makeatletter
\def\th@plain{
  \thm@notefont{}
  \itshape
}
\def\th@definition{
  \thm@notefont{}
  \normalfont
}
\makeatother
\newtheorem{innercustomprop}{Proposition}
\newenvironment{customprop}[1]
  {\renewcommand\theinnercustomprop{#1}\innercustomprop}
  {\endinnercustomprop}
\newtheorem{innercustomthm}{Theorem}
\newenvironment{customthm}[1]
  {\renewcommand\theinnercustomthm{#1}\innercustomthm}
  {\endinnercustomthm}
\newtheorem{theorem}{Theorem}[section]

\newtheorem{corollary}[theorem]{Corollary}

\newtheorem{proposition}[theorem]{Proposition}
\newtheorem{lemma}[theorem]{Lemma}

\theoremstyle{definition}
\newtheorem{definition}[theorem]{Definition}

\newtheorem{remark}[theorem]{Remark}

\DeclarePairedDelimiter{\abs}\lvert\rvert

        {%
            \ifthenelse{\isundefined{\showpairingformula}}%
                    {\expandafter\comment}%
                    {}%
                    }%
         {%
            \ifthenelse{\isundefined{\showpairingformula}}%
                    {\expandafter\endcomment}%
                    {}%
          }

\renewcommand{\d}[1][t]{\ensuremath{\left.\frac{d}{d#1}\right|_{#1=0}}}
\newcommand\restr[2]{{
  \left.\kern-\nulldelimiterspace
  #1
  \vphantom{\big|}
  \right|_{#2}
  }}

\makeatletter

\newcommand*\bigcdot{\mathpalette\bigcdot@{.5}}
\newcommand*\bigcdott{\mathpalette\bigcdot@{.7}}
\newcommand*\bigcdot@[2]{\mathbin{\vcenter{\hbox{\scalebox{#2}{$\m@th#1\bullet$}}}}}
\g@addto@macro\bfseries{\boldmath}
\makeatother

\newcommand{\el}{\ensuremath{\Delta^{\mathrm{e}}}}

\newcommand{\epo}[1]{\ensuremath{\mc P_{#1}^{\mathrm{e},o}}}
\newcommand{\eop}[1]{\ensuremath{\vec{e}^{\,\mathrm{op}}}}
\newcommand{\ep}[1]{\ensuremath{\mc P_{#1}^{\mathrm{e}}}}

\newcommand{\pke}[1]{\ensuremath{{p}^{\mathrm{e}}_{#1}}}
\newcommand{\eb}[1]{\ensuremath{\beta_{#1}^{\mathrm{e}}}}

\newcommand{\mc}[1]{\ensuremath{\mathcal{#1}}}
\newcommand{\mf}[1]{\ensuremath{\mathfrak{#1}}}
\newcommand{\on}[1]{\ensuremath{\operatorname{#1}}}

\newcommand{\Hom}[3]{\ensuremath{\on{Hom}_{#1}(#2,#3)}}

\newcommand{\N}{\ensuremath{\mathbb{N}}}
\newcommand{\Z}{\ensuremath{\mathbb{Z}}}
\newcommand{\C}{\ensuremath{\mathbb{C}}}

\newcommand{\K}{\ensuremath{\mathbb{K}}}

\newcommand{\PGL}[1]{\ensuremath{\mathrm{PGL}(#1)}}

\newcommand{\intd}{\ensuremath{\,\mathrm{d}}}

\date{\today}

\begin{document}

\pagenumbering{arabic}

\title[Edge Laplacians and Edge Poisson Transforms for Graphs]{Edge Laplacians and Edge Poisson Transforms for Graphs}

\author[Arends]{Christian Arends}
\address{Department of Mathematics, Aarhus University, Ny Munkegade 118,
        8000 Aarhus C, Denmark}
        \email{arends@math.au.dk}
\author[Frahm]{Jan Frahm}
\address{Department of Mathematics, Aarhus University, Ny Munkegade 118,
        8000 Aarhus C, Denmark}
        \email{frahm@math.au.dk}
\author[Hilgert]{Joachim Hilgert}
\address{Institut f\"ur Mathematik, Universit\"at Paderborn, Warburger Str. 100,
        33098 Paderborn, Germany}
        \email{hilgert@math.upb.de}

\begin{abstract}
For a finite graph, we establish natural isomorphisms between eigenspaces of a Laplace operator acting on functions on the edges and eigenspaces of a transfer operator acting on functions on one-sided infinite non-backtracking paths. Interpreting the transfer operator as a classical dynamical system and the Laplace operator as its quantization, this result can be viewed as a \emph{quantum-classical correspondence}. In contrast to previously established quantum-classical correspondences for the vertex Laplacian which exclude certain exceptional spectral parameters, our correspondence is valid for all parameters. This allows us to relate certain spectral quantities to topological properties of the graph such as the cyclomatic number and the $2$-colorability.\\
The quantum-classical correspondence for the edge Laplacian is induced by an edge Poisson transform on the universal covering of the graph which is a tree of bounded degree. In the special case of regular trees, we relate both the vertex and the edge Poisson transform to the representation theory of the automorphism group of the tree and study associated operator valued Hecke algebras.
\end{abstract}

\maketitle


\section*{Introduction}

The study of harmonic analysis, in particular combinatorial Laplacians and Poisson transforms, on regular trees is by now quite classical, see e.g.\@ \cite{FTN91} for a detailed exposition. Recently, there has been renewed interest in this subject motivated by the striking similarities between the dynamical properties of regular graphs \cite{LP16,An17,BHW21,BHW23} and rank one locally symmetric spaces.

The motivating example for this paper is the correspondence between spectral data associated with classical and quantum free motions on compact rank one locally symmetric spaces as described in
\cite{DFG,GHWa,GHWb,AH21}. The key link between the two dynamical systems is the parameterized system of (scalar) Poisson transforms relating hyperfunctions on the boundary of the corresponding symmetric space and eigenfunctions of the Laplace-Beltrami operator, see \cite{DFG,GHWb}. The spectral parameter on which these Poisson transforms depend gives the eigenvalues. For generic spectral parameters the Poisson transforms are bijections between hyperfunctions and the corresponding eigenspaces of the Laplace-Beltrami operator. For the exceptional parameters leading to Poisson transforms which are neither injective nor surjective for the aforementioned spaces, the spectral quantum-classical correspondence takes a slightly different form. It is given by vector valued Poisson transforms relating hyperfunctions on the boundary to sections of vector bundles over the symmetric space, see \cite{GHWa,AH21}. It turned out that for the compact quotients the quantum side of the spectral correspondence carries topological information of the locally symmetric space.

In \cite{BHW21,BHW23} an analog of the spectral quantum-classical correspondence was established for finite graphs. Also in that case one has to make the distinction between regular and exceptional spectral values. One way to prove the correspondence for regular spectral values is to go to the universal cover of the graph, which is a tree. Then one applies a Poisson transform, which for regular spectral parameters is a bijection between the space of finitely additive measures on the boundary of the tree and eigenfunctions of the vertex Laplacian.

More precisely, on the level of finite graphs the spectral properties of the classical dynamics are phrased in terms of transfer operators. These operators act on spaces of functions of infinite chains of edges which are the graph analogs of geodesic rays. The spectral invariants entering the spectral correspondence are then eigenspaces of the dual transfer operators. For regular parameters they can be naturally identified with $\Gamma$-invariant finitely additive measures on the boundary of the covering tree, where $\Gamma$ is the group of deck transformations acting via a kind of principal series representation depending on the given spectral parameter. Then the corresponding Poisson transform yields a $\Gamma$-equivariant bijection to a Laplace eigenspace for the tree, which in turn descends to a Laplace eigenspace on  the original finite graph. See \cite[Thm.~11.5]{BHW23} for all this.

In this paper we show that edge Laplacians and edge Poisson transforms can be used to establish an analog of the spectral quantum-classical correspondence also for exceptional spectral parameters. Moreover, we give interpretations of the quantum side in terms of the topology of the finite graph.

It should be noted that compared to the Riemannian manifold case we ask for less symmetry in the case of graphs. This is in line with the general philosophy that graph analogs are analytically easier and allow to explore generalizations which are at the moment out of reach for Riemannian manifolds. However, in the case of regular graphs, we also explain how to construct Poisson transforms using the representation theory of the automorphism group of the universal cover which is a regular tree. A special case of such a tree is the Bruhat--Tits tree associated to $\PGL{2}$ over a non-archimedean local field. Therefore, our theory can be viewed as a non-archimedean version of \cite{AH21}.

Let us describe our results in some more detail. For the precise version see the main body of the paper.

\subsection*{Edge Laplacian and edge Poisson transform}

We consider graphs $\mf G=(\mf X,\mf E)$ consisting of a set $\mf X$ of vertices and a set $\mf E\subseteq \mf X^2$ of directed edges. For each directed edge $\vec{e}=(a,b)$ we call $a=\iota(\vec{e})$ the initial and $b=\tau(\vec{e})$ the terminal point of $\vec{e}$. Moreover, let $\eop{e}=(b,a)$ denote the opposite edge of $\vec{e}$. We assume that the graph is symmetric, i.e.\@ the set $\mf E\subseteq \mf X^2$ is invariant under the switch of coordinates, does not contain loops, i.e.\@ $(x,x)\not\in\mf E$ for all $x\in\mf X$, and has no dead ends (see Proposition~\ref{prop:dead_ends} for the treatment of dead ends).

The edge Laplacian $\el$ acting on functions $f\in\mathrm{Maps}({\mf E},\C)$ is defined by
\begin{gather*}
\el f(\vec{e})\coloneqq\sum_{\substack{\eop{e}\neq\vec{e}\,'\in{\mf E}\\\iota(\vec{e}\,')=\tau(\vec{e})}}f(\vec{e}\,').
\end{gather*}
For $z\in\mathbb{C}$ and each subspace $U\subseteq\mathrm{Maps}({\mf E},\C)$ stable under $\el$, we consider the eigenspace
\begin{gather*}
    \mc E_z(\el;U)\coloneqq\{f\in U\mid \el f=zf\}.
\end{gather*}

If $\mf G$ is a tree and $\Omega$ its boundary we denote by $\partial_+\vec{e}\subseteq \Omega$ the set of boundary points which can be reached by a non-backtracking path starting with $\vec{e}$. For $0\neq z\in\C$ the edge Poisson kernel is defined by
\begin{gather*}
\pke{z}\colon{\mf E}\times\Omega\to\C,\quad (\vec{e},\omega)\mapsto z^{\langle\iota(\vec{e}),\omega\rangle}\mathbbm{1}_{\partial_+\vec{e}}(\omega),
\end{gather*}
where $\mathbbm{1}_{\partial_+\vec{e}}$ denotes the indicator function of $\partial_+\vec{e}$ on $\Omega$ and $\langle\bigcdot,\bigcdot\rangle\colon\mf X\times\Omega\to\Z$
the horocycle bracket (see Section~\ref{subsec:edge_Poisson}). The corresponding edge Poisson transform is given by
\begin{gather*}\label{eq:IntroEdgePoisson}
\ep{z}\colon\mc D'(\Omega)\to\mathrm{Maps}({\mf E},\C),\quad\mu\mapsto\int_\Omega\pke{z}(\bigcdot,\omega)\intd\mu(\omega),
\end{gather*}
where $\mc D'(\Omega)=C^{\mathrm{lc}}(\Omega)'$ denotes the algebraic dual of the space $C^{\mathrm{lc}}(\Omega)$ of locally constant functions on $\Omega$ which is canonically identified with the space of finitely additive measures on $\Omega$ (see \cite[Prop.\@~3.9]{BHW21}). We emphasize that no vector space topology is involved in taking the dual of $C^{\mathrm{lc}}(\Omega)$. Then we have the following result, which should be compared to \cite[Prop.~7.1]{BM96}.

\begin{customprop}{A}[Proposition~\ref{thm:PT_iso}]\label{prop:IntroA}
Let $\mf G$ be a tree of bounded degree and $0\neq z\in\C$. Then $\ep{z}$ is an isomorphism onto
\begin{gather*}
\mc E_z(\el;\mathrm{Maps}({\mf E},\C))=\ker(\el-z)\subseteq\mathrm{Maps}({\mf E},\C).
\end{gather*}
\end{customprop}

The space $\mc E_z(\el;\mathrm{Maps}({\mf E},\C))$ is the graph analog of the quantum side in the quantum-classical correspondence for compact rank one locally symmetric spaces. Let us describe the graph analog of the classical side of this correspondence, consisting of certain spectral data derived from the geodesic flow. For this we consider the space $\mf P$ of non-backtracking paths (the analog of geodesic rays), equipped with a natural topology, and the $1$-shift $\sigma_\mathbb{A}$ of paths (the analog of the geodesic flow), see Definition~\ref{def:shift_dyn}. The aforementioned spectral data is then given by the $z$-eigenspace $\mathcal E_z(\mathcal L; C^\mathrm{lc}(\mf P))$ of the Ruelle transfer operator
\begin{equation*}
\mathcal L: \mathrm{Maps}(\mf P, \C)\to \mathrm{Maps}(\mf P, \C), \quad (\mathcal L F)(\vec{\mathbf e})\ \coloneqq\ \sum_{\sigma_\mathbb{A}(\vec{\mathbf e}\,')=\vec{\mathbf e}}  F(\vec{\mathbf e}\,'),
\end{equation*}
on the space $C^\mathrm{lc}(\mf P)$ of locally constant functions on $\mf P$.

To establish the quantum-classical correspondence for a finite graph we pass to its universal cover $\widetilde{\mf G}$, which is a tree, and consider the corresponding group $\Gamma$ of deck transformations. Then the horocycle bracket allows to associate a cocycle with each spectral parameter $z\not=0$ which gives a representation $\pi_z$ of $\Gamma$ on $\mc D'(\widetilde\Omega)$ such that the space $\mc D'(\widetilde\Omega)^{\Gamma,z}$ of fixed points under this representation can be identified with both the classical and the quantum side (see Section~\ref{subsubsec:intertwining} for details).

Let  $\mc D'(\mf P)$ denote the algebraic dual so that the transpose  $\mathcal L'$ of  the transfer operator $\mc L$ acts on $\mc D'(\mf P)$. Then the edge Laplacian version of the quantum-classical correspondence for finite graphs reads as follows.

\begin{customthm}{B}[Theorem~\ref{thm:excep qcc}]\label{thm:IntroB}
Let $\mathfrak G$ be a finite graph and $z\in \C\setminus \{0\}$. Then the edge Poisson transform induces an isomorphism $\mc E_z(\el;\mathrm{Maps}({\mf E},\C)) \cong  \mc D'(\widetilde\Omega)^{\Gamma,z}$. Moreover, the pullback to limit points of paths yields an isomorphism $\mc D'(\widetilde\Omega)^{\Gamma,z} \cong
 \mathcal E_z(\mathcal L'; \mc D'(\mf P))$. Together with the isomorphism
 $\mathcal E_z(\mathcal L'; \mc D'(\mf P))\cong \mathcal E_z(\mathcal L; C^\mathrm{lc}(\mf P))$ provided by \cite[Cor.~9.4]{BHW23} this gives the quantum-classical correspondence
 \[
 \mc E_z(\el;\mathrm{Maps}({\mf E},\C))  \cong
 \mathcal E_z(\mathcal L; C^\mathrm{lc}(\mf P)).\]
\end{customthm}

Let us compare this correspondence to the quantum-classical correspondence derived in \cite{BHW23} using the vertex Poisson transforms
\begin{gather}\label{eq:IntroVertexPoisson}
\mc P_{z}\colon\mc D'(\widetilde{\Omega})\to\mathrm{Maps}(\widetilde{\mf X},\C),\quad\mu\mapsto\int_{\widetilde{\Omega}} p_{z}(\bigcdot,\omega)\intd\mu(\omega)
\end{gather}
with Poisson kernel $p_z\colon\widetilde{\mf X}\times\widetilde{\Omega}\to\C,\,(x,\omega)\mapsto z^{\langle x,\omega\rangle}$. For a vertex $x\in \mf X$ let $q_x+1$ be the number of neighbors of $x$. In \cite[Thm.\@~11.5]{BHW23} it is shown that $\mc P_z$ induces isomorphisms
$$ \mathcal{E}_{\chi(z)}(\Delta;\mathrm{Maps}({\mf X},\mathbb C)) \cong \mathcal D'(\widetilde{\Omega})^{\Gamma,z} \cong \mathcal E_z(\mathcal L;C^{\mathrm{lc}}(\mf P)) $$
for all $z\in\mathbb C\setminus\{0,\pm1\}$, where $\mathcal{E}_{\chi(z)}(\Delta;\mathrm{Maps}({\mf X},\mathbb C))$ is the equalizer of the vertex Laplacian $\Delta$ and the multiplication operator given by the function $\chi(z):\mathfrak X\to \C,\ x\mapsto \frac{z+q_x z^{-1}}{q_x+1}$. The only possible exceptional parameters are therefore $z=\pm1$, and for these parameters Theorem~\ref{thm:IntroB} establishes a quantum-classical correspondence. We refer to Remark~\ref{rem:ExceptionalParameters} for the fact that $z=\pm1$ are indeed exceptional parameters in the sense that the vertex Poisson transform is neither injective nor surjective. This is also reflected by Theorems~\ref{thm:IntroC} and \ref{thm:IntroD} below.

\subsection*{Application to the topology of finite graphs}

The topological data of a finite graph $\mf G$ that can be read off the quantum-classical correspondence for exceptional parameters are the cyclomatic number $\mf c(\mf G)$, i.e.\@  the minimal number of edges that must be removed from $\mf G$ to break all its cycles, and the $2$-colorability of $\mf G$. These topological invariants are related to the lack of bijectivity of the vertex Poisson transform $\mc P_z:\mc D'(\widetilde{\Omega})^{\Gamma,z}\to\mc E_{\chi(z)}(\Delta;\mathrm{Maps}(\mf X,\C))$.

For the spectral parameter $z=1$ the result is:

\begin{customthm}{C}[Theorems~\ref{thm:im_P1_Gamma} \& \ref{thm:dim_z1}]\label{thm:IntroC}
Let $\mf G$ be a finite graph. The spectral parameter $z=1$ is exceptional. More precisely, we have
 \begin{gather*}
\dim_\C(\mathrm{im}(\mc P_1\vert_{\mc D'(\widetilde\Omega)^{\Gamma,1}}))=
\begin{cases}
       0&\colon\mf c(\mf G)\neq1,\\
       1&\colon\mf c(\mf G)=1
\end{cases}
\end{gather*}
and
\begin{gather*}
\dim_\C(\mc D'(\widetilde\Omega)^{\Gamma,1})=\dim_\C(\mc E_1(\el;\mathrm{Maps}({\mf E},\C)))=
\begin{cases}
       \mf c(\mf G)&\colon\mf c(\mf G)\neq1,\\
       2&\colon\mf c(\mf G)=1.
\end{cases}
\end{gather*}

\end{customthm}

For the spectral parameter $z=-1$ we find

\begin{customthm}{D}[Theorems~\ref{thm:im_P-1_Gamma} \& \ref{thm:dim_z-1}]\label{thm:IntroD}
Let $\mf G$ be a finite graph. The spectral parameter $z=-1$ is exceptional. More precisely, we have the following. If $\mf c(\mf G)\neq1$ or $\mf G$ is not $2$-colorable, then
$\dim_\C(\mathrm{im}(\mc P_{-1}\vert_{\mc D'(\widetilde\Omega)^{\Gamma,-1}}))=0$
and
\begin{gather*}
\dim_\C(\mc D'(\widetilde\Omega)^{\Gamma,-1})=\dim_\C(\mc E_{-1}(\el;\mathrm{Maps}({\mf E},\C)))=
\begin{cases}
       \mf c(\mf G)&\colon\mf G\text{ is $2$-colorable},\\
       \mf c(\mf G)-1&\colon\mf G\text{ is not $2$-colorable}.
\end{cases}
\end{gather*}
If $\mf c(\mf G)=1$ and $\mf G$ is $2$-colorable, then $\dim_\C(\mathrm{im}(\mc P_{-1}\vert_{\mc D'(\widetilde\Omega)^{\Gamma,-1}}))=1$ and
\begin{gather*}
\dim_\C(\mc D'(\widetilde\Omega)^{\Gamma,-1})=\dim_\C(\mc E_{-1}(\el;\mathrm{Maps}(\mf E,\C)))=2.
\end{gather*}
\end{customthm}

\subsection*{Representation-theoretic Poisson transforms and Hecke algebra}

Now let us assume that $\mf G$ is a $(q+1)$-regular tree. The automorphism group $G$ of $\mf G$ acts transitively on $\mf X$, so we can identify $G/K\cong\mf X$, where $K$ denotes the stabilizer of a fixed base point $o\in\mf X$. This gives rise to a homogeneous vector bundle $G\times_K V$ over $G/K\cong\mf X$ for every finite-dimensional representation $V$ of $K$.

On the other hand, there is a family of representations $H_z=(\pi_z,\mathcal D'(\Omega))$ of $G$ parameterized by $z\in\mathbb C\setminus\{0\}$, called the \emph{unramified principal series}. These representations are irreducible for $z\not\in\{\pm1,\pm q\}$ and otherwise have a composition series of length two (see Lemma~\ref{la:exc_pts}). Their restriction to $K$ decomposes into a multiplicity-free direct sum of irreducible representations $V_i$, $i\in\mathbb N_0$ (see Proposition~\ref{prop:rep_of_K}). In particular, $V_0$ is the trivial representation and $V_0\oplus V_1$ is the representation of $K$ on $\mathrm{Maps}(\iota^{-1}(o),\mathbb C)$, the space of functions on the edges starting in the base point $o$ (the constant functions constituting a copy of the trivial representation).

Following the construction of Olbrich~\cite{Ol} in the archimedean case, we define for every finite-dimensional representation $(\pi,V)$ of $K$ a Poisson transform (see Definition~\ref{def:vv_PT})
$$\mc P_z^\pi:\operatorname{Hom}_K^{\mathrm{cont}}(H_z,V)\otimes H_z\to C(G\times_K V). $$
Here, $\operatorname{Hom}_K^{\mathrm{cont}}(H_z,V)$ denotes the space of $K$-intertwining operators from $H_z$ to $V$ which are continuous with respect to the weak convergence of measures, and $C(G\times_K V)$ denotes the space of sections of the homogeneous vector bundle $G\times_K V$ over $G/K\cong\mf X$.

If $V$ is irreducible, the space $\operatorname{Hom}_K^{\mathrm{cont}}(H_z,V)$ is one-dimensional since $H_z|_K$ is multiplicity-free. In this case, we obtain a unique (up to scalar multiples) Poisson transform $\mathcal D'(\Omega)\to C(G\times_K V)$. For $V=V_0$ the trivial representation, we have $C(G\times_K V)\cong\mathrm{Maps}(\mf X,\mathbb C)$ and $\mc P_z^\pi$ reduces to the vertex Poisson transform $\mathcal P_z$ defined in \eqref{eq:IntroVertexPoisson}. To also recover the edge Poisson transform, we choose $V=V_0\oplus V_1$ and note that $C(G\times_KV)$ canonically identifies with $\mathrm{Maps}(\mf E,\mathbb C)$. In this case, the space $\operatorname{Hom}_K^{\mathrm{cont}}(H_z,V)$ is two-dimensional so we obtain two independent Poisson transforms $\mathcal D'(\Omega)\to\mathrm{Maps}(\mf E,\mathbb C)$ which we identify as the edge Poisson transform $\mu\mapsto\mathcal P_z^{\mathrm{e}}\mu$ and the map $\mu\mapsto\mathcal P_z\mu\circ\iota$ which essentially is the vertex Poisson transform (see Proposition~\ref{def:vv_PT}).

Finally, we also construct a non-archimedean analogue of the algebra of invariant differential operators on $C(G\times_K V)$, the \emph{operator valued Hecke algebra} $\mathcal H(G,K;V)$ (see Definition~\ref{def:OperatorValuedHeckeAlgebra}). It acts on $C(G\times_K V)$ by convolution and we also define an action on $\operatorname{Hom}_K^{\mathrm{cont}}(H_z,V)$ that makes the Poisson transform $\mc P_z^\pi$ intertwining for $\mathcal H(G,K;V)\times G$ (see Proposition~\ref{prop:PT_equivariance}). For $V=V_0$ the trivial representation, this algebra is commutative and generated by the vertex Laplacian $\Delta$ acting on $\mathrm{Maps}(\mf X,\mathbb C)$, thus providing a conceptual proof of the fact that the vertex Poisson transform $\mathcal P_z$ maps into a space of eigenfunctions of $\Delta$. For $V=V_0\oplus V_1$ the Hecke algebra is no longer commutative and we describe it in terms of generators and relations:

\begin{customprop}{E}[Proposition~\ref{prop:HeckeAlgebraGenRel}]\label{intro:PropE}
	For $V=V_0\oplus V_1$ the \emph{operator valued Hecke algebra} $\mathcal H(G,K;V)$ acting on $\mathrm{Maps}(\mf E,\mathbb C)$ is generated by the edge Laplacian $\Delta^{\mathrm{e}}$ and the inversion operator $\Xi$ given by $\Xi f(\vec{e})\coloneqq f(\eop{e})$ ($\vec{e}\in\mf E$) subject to the relations
	$$ \Xi^2 = \mathrm{Id}, \qquad \Delta^{\mathrm{e}}\Xi\Delta^{\mathrm{e}}=q\Xi+(q-1)\Delta^{\mathrm{e}}. $$
\end{customprop}

In this case, $\operatorname{Hom}_K^{\mathrm{cont}}(H_z,V)$ is two-dimensional and the Hecke algebra does not act by scalar multiples of the identity, so one cannot directly conclude from Proposition~\ref{intro:PropE} that the edge Poisson transform $\mathcal P_z^{\mathrm{e}}$ maps into an eigenspace of the edge Laplacian $\Delta^{\mathrm{e}}$. However, the action of $\mathcal H(G,K;V)$ on the two-dimensional space $\operatorname{Hom}_K^{\mathrm{cont}}(H_z,V)$ gives rise to various relations between the operators $\Delta^{\mathrm{e}}$ and $\Xi$ applied to $\mathcal P_z^{\mathrm{e}}\mu$ and $\mathcal P_z\mu\circ\iota$ (see Proposition~\ref{prop:Hecke_action_on_PT}). Some of these relations have been observed earlier, so these results can be viewed as a conceptual explanation of them.

\subsection*{Structure of the paper}

In Section~\ref{sec:graphs} we introduce the edge Laplacian and the transfer operator associated with a graph. Section~\ref{sec:trees} is concerned with the edge Poisson transform and its mapping properties and gives a proof of Proposition~\ref{prop:IntroA}. Applications to the topology of finite graphs are given in Section~\ref{sec:appl_topology}, where Theorems~\ref{thm:IntroC} and \ref{thm:IntroD} are proven. In Section~\ref{sec:QCcorrespondence} we show the quantum-classical correspondence for finite graphs as stated in Theorem~\ref{thm:IntroB}. Finally, Section~\ref{sec:HomTreeGrpTheory} explains the connection to the representation theory of the automorphism group in the case of a regular tree, discusses operator-valued Hecke algebras and proves Proposition~\ref{intro:PropE}.

\subsection*{Acknowledgements}

The first named author was partially supported by a research grant from the Aarhus University Research Foundation (Grant No. AUFF-E-2022-9-34). The second named author was partially supported by a research grant from the Villum Foundation (Grant No. 00025373). The third named author was supported by the Deutsche Forschungsgemeinschaft (DFG, German Research Foundation) via the grant SFB-TRR 358/1 2023 -- 491392403.

\subsection*{Notation}

For a set $X$ we write $\mathrm{Maps}(X,\C)$ for the vector space of maps $f:X\to\C$ with the pointwise operations. If $X$ carries a topology, then $\mathrm{Maps}_c(X,\C)$ denotes the subspace of maps $f$ with compact support, i.e.\@ the closure of the set $\{x\in X\mid f(x)\neq0\}$ is compact.

\section{Edge Laplacians and transfer operators for graphs}\label{sec:graphs}

We consider \emph{graphs} $\mf G\coloneqq(\mf X,\mf E)$ consisting of a set $\mf X$ of \emph{vertices} and a set $\mf E\subseteq \mf X^2$ of \emph{directed edges}. Given a vertex $x\in\mf X$ of $\mf G$, we write $\on{deg}(x)$ for the \emph{degree at $x$}, i.e.\@ the number of \emph{neighbors} of $x$ (vertices connected to $x$ by an edge) and set ${q_x\coloneqq\on{deg}(x)-1}$. We will always assume that each vertex has finite degree. The graph $\mf G$ is said to have \emph{bounded degree} if $\sup_{x\in\mf X}q_x<\infty$.

We will always assume that $\mf E$ is symmetric under the switch of vertices, that $\mf G$ is connected, and that it has \emph{no loops}, i.e.\@ $\mf E\cap \{(x,x)\mid x\in \mf X\}=\emptyset$. For each directed edge $\vec{e}\coloneqq(a,b)$ we call $a=:\iota(\vec{e})$ the \emph{initial} and $b=:\tau(\vec{e})$ the \emph{terminal point} of $\vec{e}$. Moreover, let $\eop{e}\coloneqq(b,a)$ denote the \emph{opposite edge of $\vec{e}$}.

We equip $\mf X$ with the metric $d$  which assigns the minimal length of concatenated sequences  $\vec e_1,\ldots,\vec e_\ell$ of directed edges connecting a pair of  vertices. In particular, $\mf X$ carries the discrete topology. When we speak of the topology of the graph $\mf G$ we mean the topology of the simplicial complex we obtain when we identify two opposite directed edges as one undirected edge.

\subsection{Vertex and edge Laplacians}\label{rem:Laplacian}

We adopt from \cite[\S~3]{BHW23} the \emph{(vertex) Laplacian} $\Delta$ on $\mathfrak G$ operating on the space $\mathrm{Maps}(\mf X,\C)$ of functions $f:\mathfrak X\to\C$ by averaging over neighbors. I.e., it is given by
\begin{equation}\label{eq:vertex Laplacian}
(\Delta f)(x)\coloneqq\frac{1}{q_x+1}\sum_{\vec e\in{\mathfrak E},\iota(\vec e)=x} f\big(\tau (\vec e)\big)
= \frac{1}{q_x+1}\sum_{d(x,y)=1} f(y).
\end{equation}
For $0\neq z\in\mathbb{C}$ we consider the function
\[\chi(z):\mathfrak X\to \C,\quad x\mapsto \frac{z+q_x z^{-1}}{q_x+1}\]
and its associated ``eigenspaces''
\begin{gather*}
    \mc E_{\chi(z)}(\Delta;\mathrm{Maps}({\mf X},\C))\coloneqq\{f\colon\mathfrak{X}\to \mathbb{C}\mid \forall\, x\in\mathfrak{X}\colon \Delta f(x)=\chi(z)(x)f(x)\}
\end{gather*}
from \cite[Rem.~4.2(iii)]{BHW21}.

\begin{remark}\label{rem:Laplacians on graphs}
 Note that on graphs a variety of different Laplacians are considered. The most commonly used Laplace operator is $L=D-D\Delta$, where $D\coloneqq (d_{x,y})_{x,y\in\mathfrak{X}}$ denotes the \emph{degree matrix} defined by $d_{x,x}\coloneqq\mathrm{deg}(x)$ and $d_{x,y}\coloneqq 0$ for $x\neq y$. In the case of $(q+1)$-regular graphs we have $D=(1+q) \mathrm{Id}$, thus $\Delta$ and $L$ have the same eigenfunctions (with eigenvalues $\mu_i$ and $(1+q)(1-\mu_i)$, respectively). On a general graph of bounded degree the operator $\Delta$ is called the \emph{random walk Laplacian}. In fact, if $(p_x)_{x\in \mathfrak X}$ describes a probability distribution, then $p'_y\coloneqq\sum_{x\in\mathfrak X} p_x\Delta_{x,y}$ describes the probability distribution of a stochastic process where the probability at any vertex equidistributes on the neighboring vertices. For simplicity we will simply call $\Delta$ the (vertex) Laplacian.
\end{remark}

\begin{definition}\label{def:edge_Laplacian}
We define the \emph{edge Laplacian} $\el$ acting on functions $f\in\mathrm{Maps}({\mf E},\C)$ by
\begin{gather*}
\el f(\vec{e})\coloneqq\sum_{\substack{\eop{e}\neq\vec{e}\,'\in{\mf E}\\\iota(\vec{e}\,')=\tau(\vec{e})}}f(\vec{e}\,').
\end{gather*}
Moreover, we consider the operators $\Sigma$ and $\Xi$ given by
\begin{gather*}
    \Sigma f(\vec{e})\coloneqq\sum_{\iota(\vec{e}\,')=\iota(\vec{e})}f(\vec{e}\,')\quad\text{ and }\quad \Xi f(\vec{e})\coloneqq f(\eop{e}).
\end{gather*}
For $z\in\mathbb{C}$, each $T\in\{\el,\Sigma,\Xi\}$ and each subspace $U\subseteq\mathrm{Maps}({\mf E},\C)$ stable under $T$, we consider the eigenspace
\begin{gather*}
    \mc E_z(T;U)\coloneqq\{f\in U\mid Tf=zf\}.
\end{gather*}
\end{definition}

We have the following general result on the structure of eigenspaces of $\el$.

\begin{lemma}\label{la:eigenspace_characterization1}
Let $f\in\mc E_z(\el;\mathrm{Maps}({\mf E},\C))$ for some $z\in\mathbb{C}$. Then, for each concatenated edge pair $(\vec{e}_1,\vec{e}_2)$, i.e.\@ $\tau(\vec e_1) = \iota(\vec e_{2})$, we have
\begin{gather}\label{eq:eigenspace_pair}
    zf(\vec{e}_1)+f(\eop{e}_1)=\sum_{\iota(\vec{e})=\tau(\vec{e}_1)}f(\vec{e})=zf(\eop{e}_2)+f(\vec{e}_2).
\end{gather}
Moreover, if \eqref{eq:eigenspace_pair} holds for some $f\in\mathrm{Maps}({\mf E},\C)$ then $f\in\mc E_z(\el;\mathrm{Maps}({\mf E},\C))$.
\end{lemma}

\begin{proof}
Let $(\vec{e}_1,\vec{e}_2)$ satisfy $\tau(\vec e_1) = \iota(\vec e_{2})$. By definition of $\el$ we then obtain
\begin{gather*}
    zf(\vec{e}_1)=\sum_{\substack{\eop{e}_1\neq\vec{e}\,'\in{\mf E}\\\iota(\vec{e}\,')=\tau(\vec{e}_1)}}f(\vec{e}\,')=\sum_{\substack{\vec{e}_2\neq\vec{e}\,'\in{\mf E}\\\iota(\vec{e}\,')=\tau(\vec{e}_1)}}f(\vec{e}\,')-f(\eop{e}_1)+f(\vec{e}_2)=zf(\eop{e}_2)-f(\eop{e}_1)+f(\vec{e}_2).
\end{gather*}
On the other hand,
\begin{gather*}
    zf(\vec{e}_1)+f(\eop{e}_1)=\sum_{\substack{\eop{e}_1\neq\vec{e}\in{\mf E}\\\iota(\vec{e})=\tau(\vec{e}_1)}}f(\vec{e})+f(\eop{e}_1)=\sum_{\iota(\vec{e})=\tau(\vec{e}_1)}f(\vec{e}).
\end{gather*}
This equation also proves the other direction.
\end{proof}

Lemma \ref{la:eigenspace_characterization1} allows us to prove a first connection between the spectrum of the edge Laplacian and the topology of the graph.

\begin{proposition}\label{prop:eig_0}
For $z=0$ the dimension of the $0$-eigenspace $\mc E_0(\el;\mathrm{Maps}({\mf E},\C))$ of $\el$ is given by the number of leaves in $\mathfrak{G}$, i.e.\@ the number of vertices with degree one.
\end{proposition}

\begin{proof}
     By Lemma \ref{la:eigenspace_characterization1} we have
\begin{gather*}
    f(\eop{e}_1)=\sum_{\iota(\vec{e})=\tau(\vec{e}_1)}f(\vec{e})=f(\vec{e}_2).
\end{gather*}
for each pair $(\vec{e}_1,\vec{e}_2)$ of concatenated edges and $f\in \mc E_0(\el;\mathrm{Maps}({\mf E},\C))$. Thus, $f(\vec{e})$ only depends on $\iota(\vec{e})$, i.e.\@ there exists a function $\varphi\in \mathrm{Maps}(\mathfrak{X},\mathbb{C})$ such that $f(\vec{e})=\varphi(\iota(\vec{e}))$ for each $\vec{e}\in\mathfrak{E}$. Moreover, we have
\begin{gather*}
    0=\el f(\vec{e})=\sum_{\substack{\eop{e}\neq\vec{e}\,'\in{\mf E}\\\iota(\vec{e}\,')=\tau(\vec{e})}}f(\vec{e}\,')=\sum_{\substack{\eop{e}\neq\vec{e}\,'\in{\mf E}\\\iota(\vec{e}\,')=\tau(\vec{e})}}\varphi(\tau(\vec{e}))=q_{\tau(\vec{e})}\varphi(\tau(\vec{e})),
\end{gather*}
so that $q_{\tau(\vec{e})}\neq 0$ implies $\varphi(\tau(\vec{e}))=0$. On the other hand, any function $\varphi:\mf X\to \C$ satisfying $\varphi(x)=0$ for each $x\in\mathfrak{X}$ with $q_x\neq0$ gives rise to an eigenfunction $f\in\mc E_0(\el;\mathrm{Maps}({\mf E},\C))$.
\end{proof}

For $z\neq0$, the eigenspace $\mc E_z(\el;\mathrm{Maps}({\mf E},\C))$ does not depend on leaves. In fact, let $\mathfrak{G}'=(\mathfrak{X}',\mathfrak{E}')$ denote the largest connected subgraph of $\mathfrak{G}$ such that for each edge $\vec{e}_0\in\mathfrak{E}'$ there exists an infinite \emph{chain of edges} starting with $\vec{e}_0$, i.e.\@ an infinite sequence $\vec{e}_0,\vec{e}_1,\ldots$ of directed edges that are \emph{concatenated}, i.e.\@ $\tau(\vec{e}_j)=\iota(\vec{e}_{j+1})$ for all $j\in\mathbb{N}_0$, and \emph{non-backtracking}, i.e.\@ $\tau(\vec{e}_{j+1})\neq\iota(\vec{e}_j)$ for all $j\in\mathbb{N}_0$.

Note that this removes the set $\mathfrak{D}(\mathfrak{E})$ of all \emph{dead ends} $\vec{e}_0\in\mathfrak{E}$, i.e.\@ edges such that there are only finitely many concatenated, non-backtracking sequences $(\vec{e}_0,\ldots,\vec{e}_n)$.

\begin{proposition}\label{prop:dead_ends}
    For $z\neq0$ the $z$-eigenspace $\mc E_z(\el;\mathrm{Maps}({\mf E},\C))$ only depends on the values on $\mathfrak{E}\setminus\mathfrak{D}(\mathfrak{E})$, i.e.\@ every eigenfunction in $\mc E_z(\el;\mathrm{Maps}({\mf E}\setminus \mathfrak{D}(\mathfrak{E}),\C))$ uniquely extends to an eigenfunction in $\mc E_z(\el;\mathrm{Maps}({\mf E},\C))$.
\end{proposition}

\begin{proof}
Let us first consider a leaf $x\in\mathfrak{X}$ and the unique edge $\vec{e}_0\coloneqq(y,x)\in\mathfrak{E}$. Then
\begin{gather*}
   zf(\vec{e}_0)=(\el f)(\vec{e}_0)=0
\end{gather*}
for each $f\in\mc E_z(\el;\mathrm{Maps}({\mf E},\C))$ so that $f(\vec{e}_0)=0$. We claim that $f(\vec{e}_0)=0$ for each $\vec{e}_0\in\mathfrak{D}(\mathfrak{E})$ and proceed by induction over the maximal distance of $\iota(\vec{e}_0)$ to a leaf in the direction of $\vec{e}_0$. Let this distance be $n\geq2$. Then
\begin{gather*}
    zf(\vec{e}_0)=\sum_{\substack{\eop{e}_0\neq\vec{e}\,'\in{\mf E}\\\iota(\vec{e}\,')=\tau(\vec{e}_0)}}f(\vec{e}\,')=0
\end{gather*}
by the induction hypothesis.

Now let $\vec{e}_0$ be a maximal dead end in the sense that there does not exist an edge $\vec{e}\in\mathfrak{D}(\mathfrak{E})$ such that $\tau(\vec{e})=\iota(\vec{e}_0)$. Moreover, let $(\vec{e}_n,\ldots,\vec{e}_1)$ be a sequence of edges with $\vec{e}_{j+1}\neq\eop{e}_j,\ \tau(\vec{e}_j)=\iota(\vec{e}_{j+1})$ and $\vec{e}_1=\eop{e}_0$. Then, by the first part, we have $zf(\vec{e}_j)=f(\vec{e}_{j-1})$ for each $j\geq2$ since all other directions point in the same direction as $\vec{e}_0$ and thus into a dead end. Finally, using the first part again, we infer that the value of $f(\eop{e}_0)$ is given by the values of $f$ on the edges in $\mathfrak{E}\setminus\mathfrak{D}(\mathfrak{E})$.
\end{proof}

\begin{remark}
    A similar argument shows that also the eigenfunctions in $\mc E_{\chi(z)}(\Delta;\mathrm{Maps}({\mf X},\C))$ are uniquely determined by their values on $\mathfrak{X}'$. More precisely, if $\vec{e}\in\mathfrak{D}(\mathfrak{E})$ is a maximal dead end, i.e.\@ there is no $\eop{e}\neq\vec{e}\,'\in\mathfrak{D}(\mathfrak{E})$ with $\tau(\vec{e}\,')=\iota(\vec{e})$, and $x\in\mathfrak{X}$ is a vertex in the direction of $\vec{e}$, then $f(x)=z^{-d(x,\iota(\vec{e}))}f(\iota(\vec{e}))$. Moreover, for such $\vec{e}\in\mathfrak{D}(\mathfrak{E})$ the eigenvalue property at $\iota(\vec{e})$ for $\mathfrak{G}$ corresponds to that for $\mathfrak{G}'$: Let $x_1,\ldots,x_k$ denote the neighbors of $\iota(\vec{e})$ in $\mathfrak{G}$ that lie outside of $\mathfrak{G}'$ and $y_1,\ldots,y_{\ell}$ denote the ones that lie inside $\mathfrak{G}'$. Then, by the first part, the eigenvalue property for $\mathfrak{G}$ at $\iota(\vec{e})$ reads
\begin{gather*}
    (z+q_{\iota(\vec{e})}z^{-1})f(\iota(\vec{e}))=\sum_{i=1}^{k}f(x_i)+\sum_{j=1}^{\ell}f(y_j)=kz^{-1}f(\iota(\vec{e}))+\sum_{j=1}^{\ell}f(y_j)
\end{gather*}
which is equivalent to $(z+(q_{\iota(\vec{e})}-k)z^{-1})f(\iota(\vec{e}))=\sum_{j=1}^{\ell}f(y_j)$. However, denoting $q_{\iota(\vec{e})}'\coloneqq\on{deg}'(\iota(\vec{e}))-1$, where $\on{deg}'$ is the degree in $\mathfrak{G}'$, this becomes $(z+q_{\iota(\vec{e})}'z^{-1})f(\iota(\vec{e}))=\sum_{j=1}^{\ell}f(y_j)$, which is the eigenvalue equation with respect to $\mathfrak{G}'$.
\end{remark}

\subsection{The Ruelle transfer operator and resonant states}\label{sec:RuelleTO}

In this section we define resonant states in terms of transfer operators and their dual operators.

\begin{definition}[Shift dynamics]\label{def:shift_dyn}
The allowed transitions between the edges in a chain are encoded in the \emph{non-backtracking transition matrix} $\mathbb A = (\mathbb A_{\vec e,\vec e\,'})_{\vec e, \vec e\,'\in {\mathfrak E}}$ defined by
\[
\mathbb A_{\vec e,\vec e\,'}:=\begin{cases}
                 1 & \text{for }\tau(\vec e) = \iota(\vec e\,') \text{ and } \tau(\vec e\,')\neq\iota(\vec e)\\
                 0 & \text{otherwise}
                \end{cases}
\]
The \emph{shift space} $\mf P$ is the set of chains, i.e.\@ sequences $(\vec{e}_1,\vec{e}_2,\ldots)$ of edges such that $\mathbb{A}_{\vec e_{i},\vec e_{i+1}}=1$ for all $i\geq1$. We equip $\mf P$ with the \emph{district topology} (see \cite[§ 5]{BHW23}), with basic open sets given by the sets -- the so-called \emph{districts} -- of all chains $(\vec{e}_1,\vec{e}_2,\ldots)$ such that $(\vec{e}_1,\ldots,\vec{e}_n)$ is equal to some fixed tuple of edges ($n\in\mathbb{N}$). On the shift space we can introduce the \emph{shift operator} $\sigma_{\mathbb A} : \mf P\to \mf P,\quad (\vec e_1,\vec e_2,\vec e_3,\ldots)\mapsto (\vec e_2,\vec e_3,\ldots)$.
\end{definition}

If the graph happens to be a tree, chains can be interpreted as geodesic rays, so the shift space $\mf P$ is one possible graph-analog of the sphere bundle of a Riemannian manifold. If we adopt this interpretation, the shift operator can be viewed as the \emph{geodesic flow} on the tree.

Recall, e.g.\@ from \cite[\S~3]{BHW23}, the \emph{(Ruelle) transfer operator} $\mathcal L: \mathrm{Maps}(\mf P, \C)\to \mathrm{Maps}(\mf P, \C)$
associated to the shift $\sigma_\mathbb{A}$ via
\begin{equation}\label{eq:RuelleTO}
(\mathcal L F)(\vec{\mathbf e})\ \coloneqq\ \sum_{\sigma_\mathbb{A}(\vec{\mathbf e}\,')=\vec{\mathbf e}}  F(\vec{\mathbf e}\,').
\end{equation}

\emph{Koopman operators} are dual operators of transfer operators and lead to the definition of resonances. Note first that $\mathcal{L}$ leaves the space
\begin{gather*}
    C_c^{\mathrm{lc}}(\mf P)\coloneqq\{f\in \mathrm{Maps}(\mf P,\mathbb{C})\mid f\text{ locally constant with compact support}\}
\end{gather*}
invariant. We will consider the Koopman operator
\begin{gather*}
\mathcal L'\colon \mc D'(\mf P)\to \mc D'(\mf P)
\end{gather*}
defined by duality from $\mathcal L$ where $\mc D'(\mf P)=C_c^{\mathrm{lc}}(\mf P)'$ is the algebraic dual of $C_c^{\mathrm{lc}}(\mf P)$, cf.\@ \cite[\S~9]{BHW23}. We will refer to eigenvalues of $\mathcal{L}'$ as \emph{resonances} and to the corresponding eigendistributions as \emph{resonant states}. If the graph contains dead ends, we can argue as in Proposition~\ref{prop:dead_ends} to see that each resonant state is zero on functions $f$ with support on paths starting in $\{\tau(\vec{e})\mid \vec{e}\in\mathfrak{D}(\mathfrak{E})\}$. Since resonant states and their relation to Laplace eigenfunctions are our primary concern, we remove $\mf D(\mf E)$ from our graph and assume from now on that our graphs have no dead ends.

\section{The edge Poisson transform for trees}\label{subsec:edge_Poisson}\label{sec:trees}

In this section we assume that $\mf G=(\mf X,\mf E)$ is a \emph{tree}, i.e.\@ that there are no circuits in the graph when we identify each edge with its opposite, so that there is at most one edge connecting two vertices. In view of Proposition~\ref{prop:dead_ends} we also assume that $\mf G$ has no dead ends.

We first recall some definitions from \cite{BHW21}. Two one-sided infinite chains (i.e.\@ non-backtracking paths) of vertices in a tree are called \emph{equivalent} if they share infinitely many vertices. We call the set of equivalence classes of such chains the \emph{boundary at infinity $\Omega$} of $\mf G$. For vertices $x,y\in\mf X$ and $\omega\in\Omega$ we write $[x,y]$ for the (unique) chain connecting $x$ and $y$,  and $[x,\omega[$ for the representative of $\omega$ starting at $x$. For every edge $\vec{e}\in\mf E$ we denote by $\partial_+\vec{e}\subseteq\Omega$ the set of all boundary points which can be reached from $\vec{e}$ by a chain of edges. Note here that there is a canonical identification between chains of edges and chains of vertices. By reversing the orientation, we also define $\partial_-\vec{e}\coloneqq\partial_+\eop{e}$. Finally we denote by $\mc D'(\Omega)$ the algebraic dual of the space $C^{\mathrm{lc}}(\Omega)$ of locally constant complex-valued functions on $\Omega$. It is canonically identified with the space of \emph{finitely additive measures} on $\Omega$ (see \cite[\S~3]{BHW21}).

Let $o\in\mf X$ be a fixed base point and consider a pair $(x,\omega)\in\mathfrak{X}\times \Omega$. Then there exists a unique $y\in\mf X$ such that $[o,\omega[\cap[x,\omega[=[y,\omega[$. We set $\langle x,\omega\rangle\coloneqq d(o,y)-d(x,y)$ and call
\begin{gather*}
\langle\bigcdot,\bigcdot\rangle\colon\mf X\times\Omega\to\Z
\end{gather*}
the \emph{horocycle bracket}. It is locally constant in $\omega\in\Omega$ and fulfills the \emph{horocycle identity}
\begin{gather}\label{eq:horocycle_id}
\langle gx,g\omega\rangle=\langle x,\omega\rangle+\langle go,g\omega\rangle
\end{gather}
for every automorphism $g$ of $\mf G$, see e.g.\@ \cite[(15)]{BHW23}.

We can now define the edge Poisson transform and begin by recalling the definition of the \emph{(scalar) Poisson transform} from \cite[Rem.\@ 4.2]{BHW21}. For each $z\in\C\setminus\{0\}$ the \emph{Poisson kernel} is given by
$p_z\colon\mf X\times\Omega\to\C,\quad (x,\omega)\mapsto z^{\langle x,\omega\rangle}$. This leads to the \emph{Poisson~transform}
\begin{gather}\label{def:Poisson trafo}
\mc P_{z}\colon\mc D'(\Omega)\to\mathrm{Maps}({\mf X},\C),\quad\mu\mapsto\int_\Omega p_{z}(\bigcdot,\omega)\intd\mu(\omega).
\end{gather}

\begin{definition}\label{def:Poisson_trafo_vv}
Let $0\neq z\in\C$. We define the \emph{edge Poisson kernel} by
\begin{gather*}
\pke{z}\colon{\mf E}\times\Omega\to\C,\quad (\vec{e},\omega)\mapsto z^{\langle\iota(\vec{e}),\omega\rangle}\mathbbm{1}_{\partial_+\vec{e}}(\omega),
\end{gather*}
where $\mathbbm{1}_{\partial_+\vec{e}}$ denotes the indicator function of $\partial_+\vec{e}$ on $\Omega$. Note that for every edge $\vec{e}$ we can write $\partial_{+}\vec{e}$ as a disjoint union of $\partial_{+}\vec{e}_i$ for some edges $\vec{e}_i$ pointing away from $\iota(\vec{e})$ and $o$. Since $\Omega=\partial_{-}\vec{e}\sqcup\partial_{+}\vec{e}$ and
\begin{gather*}
\forall\, \omega,\omega'\in\partial_+\vec{e_i}\colon\qquad \pke{z}(\vec{e},\omega)=\pke{z}(\vec{e},\omega')
\end{gather*}
and $\pke{z}(\vec{e},\omega)=0$ for each $\omega\in\partial_-\vec{e}$, we hence infer that $\pke{z}$ is locally constant. Therefore, $\pke{z}(\vec{e},\cdot)$ is locally constant and for each $\mu\in\mc D'(\Omega)$ the integral
\begin{gather*}
\int_\Omega \pke{z}(\vec{e},\omega)\intd\mu(\omega)
\end{gather*}
is well-defined by \cite[Prop.\@ 3.7]{BHW21}. In fact, this proposition shows also that the integral is linear in $\mu$. Thus, we may define the \emph{edge Poisson transform} by
\begin{gather*}
\ep{z}\colon\mc D'(\Omega)\to\mathrm{Maps}({\mf E},\C),\quad\mu\mapsto\int_\Omega\pke{z}(\bigcdot,\omega)\intd\mu(\omega)=\int_{\partial_+\bigcdot}p_z(\iota(\bigcdot),\omega)\intd\mu(\omega).
\end{gather*}
We also consider the restriction
\begin{gather*}
\epo{z}\colon\mc D'(\Omega)\to\mathrm{Maps}({\mf E}_o,\C),\quad\epo{z}(\mu)\coloneqq\ep{z}(\mu)\vert_{{\mf E}_o}
\end{gather*}
of $\ep{z}$ to functions on ${\mf E}_o$, the space of edges pointing away from the origin $o\in \mf X$.
\end{definition}

Note that we can also express the edge Laplacian $\el$ in terms of the sets $\partial_{+}\vec{e}$.

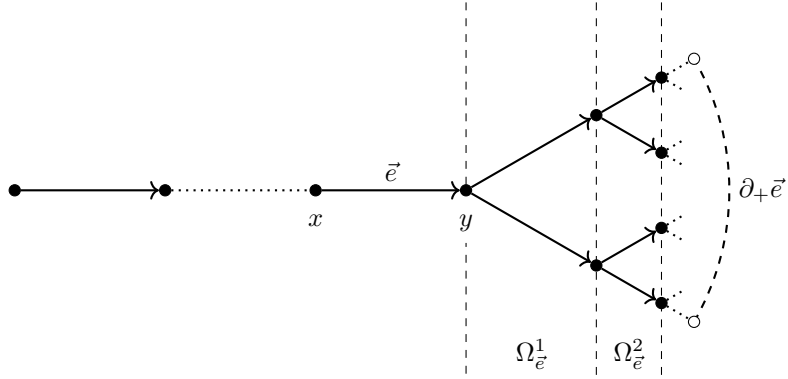
\begin{figure}
\tikzstyle{circle}=[shape=circle,draw,inner sep=1.5pt]
\begin{tikzpicture}[scale=2]
\draw (-2,0) node[circle,fill=black] (8) {};
\draw (-1,0) node[circle,fill=black] (9) {};
\draw[->,thick] (8) to (9);
\draw (-2,-0.1);
\draw (0,0) node[circle,fill=black] (0) {};
\draw[dotted,thick] (9) to (0);
\draw[dashed] (1,1.25) to (1, 0);
\draw[dashed] (1, -1.22) to (1,-0.35);
\draw (1,0) node[circle,fill=black] (1) {};
\draw[dashed] ({1+sqrt(3)/2}, 1.25) to ({1+sqrt(3)/2},-1.25);
\draw (0,-0.1) node[below] {$x$};
\draw[->,thick] (0) to (1);
\draw (1,-0.1) node[below] {$y$};
\draw (0.5,0) node[above] {$\vec{e}$};
\path (1) ++(30:1) node (2) [circle, fill=black] {};
\path (1) ++(-30:1) node (3) [circle, fill=black] {};
\draw[->,thick] (1) to (2);
\draw[->,thick] (1) to (3);
\draw ({1+sqrt(3)/4},-1.25) node[above] {$\Omega_{\vec{e}}^{1}$};
\path (2) ++(30:0.5) node (4) [circle, fill=black] {};
\path (2) ++(-30:0.5) node (6) [circle, fill=black] {};
\path (1) ++(-30:1.5) node (5) [circle, fill=black] {};
\path (3) ++(30:0.5) node (7) [circle, fill=black] {};
\draw[->,thick] (2) to (6);
\draw[->,thick] (3) to (7);
\draw[->,thick] (2) to (4);
\draw[->,thick] (3) to (5);
\draw[dashed] ({1+3*sqrt(3)/4}, 1.25) to ({1+3*sqrt(3)/4},-1.25);
\draw ({1+5*sqrt(3)/8},-1.25) node[above] {$\Omega_{\vec{e}}^{2}$};
\path (4) ++(-30:0.25) node (8) {};
\path (6) ++(30:0.25) node (9) {};
\path (6) ++(-30:0.25) node (10) {};
\path (7) ++(30:0.25) node (11) {};
\path (7) ++(-30:0.25) node (12) {};
\path (5) ++(30:0.25) node (13) {};
\draw[dotted,thick] (4) to (8);
\draw[dotted,thick] (6) to (9);
\draw[dotted,thick] (6) to (10);
\draw[dotted,thick] (7) to (11);
\draw[dotted,thick] (7) to (12);
\draw[dotted,thick] (5) to (13);
\draw[thick,dashed] ([shift=(-30:1.75)]1,0) arc (-30:30:1.75);
\path (1) ++(30:1.75) node (6) [circle,fill=white] {};
\path (1) ++(-30:1.75) node (7) [circle, fill=white] {};
\draw[dotted,thick] (4) to (6);
\draw[dotted,thick] (5) to (7);
\draw (2.75,0) node[right] {$\partial_+\vec{e}$};
\end{tikzpicture}
\caption{Definition of $\Omega_{\vec{e}}^{m}$ and $\partial_+\vec{e}$}
\label{fig:Omega_graduierung}
\end{figure}

\begin{remark}\label{rem:edge_Laplacian}
The edge Laplacian $\el$  acting on functions $f\in\mathrm{Maps}({\mf E},\C)$ from Definition~\ref{def:edge_Laplacian} can, in the case of trees, be written
\begin{gather*}
\el f(\vec{e})=\sum_{\vec{e}\,'\in\Omega_{\vec{e}}^1}f(\vec{e}\,'),
\end{gather*}
where, for $m\in\N$, $\Omega_{\vec{e}}^m\coloneqq\{\vec{e}\,'\in\vec{\mf E}\mid\partial_+\vec{e}\,'\subseteq\partial_+\vec{e},\, d(\iota(\vec{e}),\iota(\vec{e}\,'))=m\}$ (see Figure \ref{fig:Omega_graduierung}).
\end{remark}

\subsection{The endpoint map}\label{subsec:end point map}
We will realize resonant states in the space of finitely additive measures on the boundary $\Omega$. In order to establish this connection we first need to relate $\mf P$ to the boundary.

\begin{definition}
We define the \emph{end point map}
\begin{equation}
 \label{eq:B}
 B:\mf P\to \Omega,\quad (\vec{e}_1,\vec{e}_2,\ldots) \mapsto [(\vec{e}_1,\vec{e}_2,\ldots)],
\end{equation}
which maps a geodesic ray to its end point, and equip $\Omega$ with the quotient topology. Note that this topology is generated by the sets $\partial_{+}\vec{e}$ and that the canonical isomorphism $\mf P\cong \mathfrak{X}\times \Omega$ is a homeomorphism. With respect to these topologies we consider the spaces
\begin{gather*}
    C_c^{\mathrm{lc}}(\mf P)\coloneqq\{f\in \mathrm{Maps}(\mf P,\mathbb{C})\mid f \text{ locally constant with compact support}\},\\
    C^{\mathrm{lc}}(\Omega)\coloneqq\{f\in \mathrm{Maps}(\Omega,\mathbb{C})\mid f\text{ locally constant}\}.
\end{gather*}

We obtain a (surjective) pushforward by $B$ via integration of compactly supported locally constant functions against fibers:
\begin{gather*}
        B_{\star}\colon C_c^{\mathrm{lc}}(\mf P)\to C^{\mathrm{lc}}(\Omega),\quad f\mapsto\left(\omega \mapsto \sum_{x\in\mathfrak{X}}f(x,\omega)\right),
\end{gather*}
and, by duality, an (injective) pullback (cf.\@ \cite[Prop.~3.9]{BHW21})
\begin{equation}\label{eq:Bstar'}
B^\star \colon \mc D'(\Omega)  \to \mc D'(\mf P),\quad B^{\star}(\varphi)(f)\coloneqq\varphi(B_{\star}(f)).
\end{equation}
\end{definition}

We can characterize the image of $B^\star$.

\begin{lemma}\label{lem:B*-characterisation}
For $\lambda \in \mc D'(\mf P)$ there exists $T\in\mathcal{D}'(\Omega)$ with $\lambda =B^\star T$ if and only if
\begin{equation}\label{eq:B*T}
\forall f\in C^{\mathrm{lc}}(\Omega):\quad \mathfrak X\to \mathbb C,\  x\mapsto \langle\lambda,\delta_x\otimes f\rangle\text{ is constant}.
\end{equation}
Here $\delta_x$ is the indicator function of $\{x\}\subseteq \mf X$ and $\delta_x\otimes f\in C_c^{\mathrm{lc}}(\mf P)$ is defined by
\begin{gather*}
       (\delta_x\otimes f)(\vec{e}_1, \vec{e}_2,\ldots)\coloneqq\delta_x(\iota(\vec{e}_1))\cdot f(B(\vec{e}_1, \vec{e}_2,\ldots)).
\end{gather*}
\end{lemma}

\begin{proof}
Suppose that $\lambda=B^\star T$. Then by definition of $B_{\star}$ and $B^\star $ we have $\langle\lambda, \delta_x\otimes f\rangle = \langle T,f\rangle$ which is independent of $x$. Conversely, if \eqref{eq:B*T} holds and $F\in C_c^{\mathrm{lc}}(\mf P, \mathbb C)$, then we can write $F=\sum_{x\in \mathfrak X} \delta_x\otimes f_x$ as a finite sum for some $f_x\in C^{\mathrm{lc}}(\Omega, \mathbb C)$ and we have
\[
\langle \lambda, F\rangle=  \sum_{x\in \mathfrak X} \langle \lambda,\delta_x\otimes f_x\rangle=\sum_{x\in \mathfrak X} \langle \lambda,\delta_{x_0}\otimes f_x\rangle=
\langle \lambda,\delta_{x_0}\otimes \sum_{x\in \mathfrak X} f_x\rangle
\]
for some arbitrary but fixed $x_0\in\mathfrak X$. Setting $\langle T,f\rangle \coloneqq \langle \lambda, \delta_{x_0}\otimes f\rangle$ we obtain
\[
\langle \lambda, F\rangle= \langle T,\sum_{x\in \mathfrak X} f_x\rangle=
\langle T,B_{\star}  F\rangle = \langle B^\star T,F\rangle.\qedhere
\]
\end{proof}

The following remark will allow us to relate the end point map to the transfer operator.

\begin{remark}\label{rem:delta_prop}
Let $f\in C_c^{\mathrm{lc}}(\Omega)$ and assume $\mathrm{supp}(f) \subseteq \partial_+\vec{e}= B(\{(\vec e,\vec e_2,\ldots)\})$ for an edge $\vec{e}$ with $x\coloneqq\iota(\vec e)$. Then we get
\begin{gather*}
\mathcal L(\delta_{x}\otimes f)(\vec e_1,\vec e_2,\ldots)
=\sum_{\vec e_0\colon \tau(\vec e_0) = \iota (\vec e_1) }(\delta_{x}\otimes f)(\vec e_0,\vec e_1,\ldots),
\end{gather*}
but $\iota(\vec e_0) = x=\iota (\vec e)$ and $B(\vec e_0,\vec e_1,\ldots)\in \partial_+\vec{e}$ already imply $\vec e_0 = \vec e$. We thus obtain
\begin{equation}
 \label{eq:delta propagation}
 \mathcal L(\delta_{\iota(\vec e)}\otimes f) = \delta_{\tau(\vec e)}\otimes f.
\end{equation}
\end{remark}

\begin{lemma}\label{lem:B*-characterisation'}
 For a linear functional $\lambda\in \mc D'(\mf P)$ the following assertions are equivalent
 \begin{enumerate}
  \item[{\rm(1)}] There is a $T\in \mathcal{D}'(\Omega)$ such that $\lambda=B^\star T$.
  \item[{\rm(2)}] ${\mathcal L'} \lambda = \lambda$.
 \end{enumerate}
In particular, we obtain the isomorphism
\begin{gather*}
    B^{\star}\colon\mathcal{D}'(\Omega)\cong\{\lambda\in \mc D'(\mf P)\mid\mathcal{L}'\lambda=\lambda\}.
\end{gather*}
\end{lemma}

\begin{proof}
 For the implication (1)$\Rightarrow$(2)
 it suffices to test both sides against functions of the form $\delta_{\iota(\vec e)}\otimes f$ (recall the notation from Lemma~\ref{lem:B*-characterisation}) with $f\in C^{\mathrm{lc}}(\Omega)$ and
 $\mathrm{supp} f \subseteq \partial_{+}\vec e$ (see Remark \ref{rem:delta_prop}) as any element in $C^{\mathrm{lc}}_c(\mf P, \mathbb C)$ is a finite linear combination of such functions. But in view of Remark~\ref{rem:delta_prop} and \eqref{eq:B*T} we can calculate
\begin{eqnarray*}
\langle{\mathcal L}'(B^\star T),\delta_{\iota(\vec e)}\otimes f\rangle
&=&\langle B^\star T,{\mathcal L}(\delta_{\iota(\vec e)}\otimes f)\rangle\\
&=&\langle B^\star T,\delta_{\tau(\vec e)}\otimes f\rangle\\
&=&\langle B^\star T,\delta_{\iota(\vec e)}\otimes f\rangle,
\end{eqnarray*}
which proves the claim.

For the implication (2) $\Rightarrow$ (1) we have to show (according to Lemma~\ref{lem:B*-characterisation}) that for $x,y\in {\mathfrak X}$
and $f\in C^{\mathrm{lc}}(\Omega)$ we have $\langle \lambda,\delta_x\otimes f\rangle=\langle \lambda, \delta_y\otimes f\rangle$. Note that we have the disjoint decomposition
\[
\Omega= \Omega_x^y\cup \Omega_y^x \cup\bigcup_{z\in{} ]x,y[} (\Omega_x^z\cap \Omega_y^z),
\]
where $\Omega_x^y:=\{\omega\in \Omega\mid y\in [x,\omega[\}$. We write
\[
f=\Big(\mathbbm{1}_{\Omega_x^y}+\mathbbm{1}_{\Omega_y^x}+\sum_{z\in{} ]x,y[} \mathbbm{1}_{\Omega_x^z\cap \Omega_y^z}\Big)f,
\]
where $\mathbbm{1}_M$ denotes the indicator function of $M$, and treat the summands separately. Let $n$ be the distance between $x$ and $y$. Using the Equation~\eqref{eq:delta propagation} $n$ times we can calculate
\begin{eqnarray*}
\langle \lambda, \delta_x\otimes \mathbbm{1}_{\Omega_x^y}f\rangle
&=&\langle (\mathcal L')^n\lambda, \delta_x\otimes \mathbbm{1}_{\Omega_x^y}f\rangle\\
&=&\langle \lambda, \mathcal L^n (\delta_x\otimes \mathbbm{1}_{\Omega_x^y}f)\rangle\\
&=&\langle \lambda, \delta_y\otimes \mathbbm{1}_{\Omega_x^y}f\rangle
\end{eqnarray*}
since $\mathrm{supp}(\mathbbm{1}_{\Omega_x^y}f)\subseteq \Omega_x^y \subseteq\partial_+\vec{e}$ for any directed edge $\vec e$ in $[x,y]$ pointing away from $x$. Interchanging the roles of $x$ and $y$ we also obtain
$$\langle \lambda, \delta_y\otimes \mathbbm{1}_{\Omega_y^x}f\rangle
=\langle \lambda, \delta_x\otimes \mathbbm{1}_{\Omega_y^x}f\rangle.$$
For $z\in ]x,y[$ let $m$ be the distance between $x$ and $z$.
Now we have
$$\mathrm{supp}(\mathbbm{1}_{\Omega_x^z\cap \Omega_y^z}f)\subseteq \Omega_x^z\cap \Omega_y^z \subseteq \partial_+\vec e$$
for any directed edge $\vec e$ in $[x,z]$ pointing away from $x$. Thus we can calculate
\begin{eqnarray*}
\langle \lambda, \delta_x\otimes \mathbbm{1}_{\Omega_x^z\cap \Omega_y^z}f\rangle
&=&\langle (\mathcal L')^m\lambda, \delta_x\otimes \mathbbm{1}_{\Omega_x^z\cap \Omega_y^z}f\rangle\\
&=&\langle \lambda, \mathcal L^m (\delta_x\otimes \mathbbm{1}_{\Omega_x^z\cap \Omega_y^z}f)\rangle\\
&=&\langle \lambda, \delta_z\otimes \mathbbm{1}_{\Omega_x^z\cap \Omega_y^z}f\rangle.
\end{eqnarray*}
Similarly, we have
$$\mathrm{supp}(\mathbbm{1}_{\Omega_x^z\cap \Omega_y^z}f)\subseteq \Omega_x^z\cap \Omega_y^z \subseteq \partial_+\vec e$$
for any directed edge $\vec e$ in $[y,z]$ pointing away from $y$. Thus we can also calculate
\begin{eqnarray*}
\langle \lambda, \delta_y\otimes \mathbbm{1}_{\Omega_x^z\cap \Omega_y^z}f\rangle
&=&\langle (\mathcal L')^{n-m}\lambda, \delta_y\otimes \mathbbm{1}_{\Omega_x^z\cap \Omega_y^z}f\rangle\\
&=&\langle \lambda, \mathcal L^{n-m} (\delta_y\otimes \mathbbm{1}_{\Omega_x^z\cap \Omega_y^z}f)\rangle\\
&=&\langle \lambda, \delta_z\otimes \mathbbm{1}_{\Omega_x^z\cap \Omega_y^z}f\rangle.
\end{eqnarray*}
Together this gives
\[\langle \lambda, \delta_y\otimes \mathbbm{1}_{\Omega_x^z\cap \Omega_y^z}f\rangle
=
\langle \lambda, \delta_x\otimes \mathbbm{1}_{\Omega_x^z\cap \Omega_y^z}f\rangle.
\]
The final claim about the isomorphisms follows from the injectivity of the pullback $B^\star$.
\end{proof}

We are now able to see that we can realize resonant states on the boundary.
\begin{proposition}\label{prop:cpt_picture}
   For $z\neq0$ we have an isomorphism
\begin{gather*}
    p_{z}B^\star\colon\mathcal{D}'(\Omega)\cong\mathcal E_z({\mathcal L}';\mc D'(\mf P)).
\end{gather*}
\end{proposition}

\begin{proof}
Note first that for any $\vec{e}\in\mathfrak{E},\, \omega\in\partial_+\vec{e}$ and $\vec{e}\,'$ with $\mathbb{A}_{\vec{e}\,',\vec{e}}=1$ we have $\langle\iota(\vec{e}),\omega\rangle=\langle\iota(\vec{e}\,'),\omega\rangle+1$ by definition of the horocycle bracket. Using $(\pi,B): \mathfrak P\to \mathfrak X\times \Omega$ to identify $\mathfrak P$ and $\mathfrak X\times \Omega$ we find
\begin{gather*}
 \mathcal L (p_z F)(\vec e_1,\vec e_2,\ldots) = z^{-1}p_z(\vec e_1,\vec e_2,\dots) (\mathcal L F)(\vec e_1,\vec e_2,\ldots)
\end{gather*}
for any $F\in C_c^{\mathrm{lc}}(\mf P, \mathbb C)$, and by duality we have for $\lambda \in \mc D'(\mf P)$
\begin{gather*}
 \mathcal L' (p_z \lambda) = z p_z(\mathcal L' \lambda).
\end{gather*}
Thus the desired isomorphism follows from concatenating the isomorphism from Lemma~\ref{lem:B*-characterisation'} with multiplication by $p_z$.
\end{proof}

As the (edge) Poisson transform has domain $\mathcal{D}'(\Omega)$, Proposition \ref{prop:cpt_picture} raises the question what the compositions $\mathcal{P}_{z}\circ(p_{z}B^\star)^{-1}$ and $\ep{z}\circ(p_{z}B^\star)^{-1}$ look like. In order to answer this question, we will establish that the following  diagram is commutative.

\begin{gather*}
\xymatrix{
&&&\mc E_z(\el;\mathrm{Maps}({{\mf E}},\C))\\
\mc D'({\Omega})\ar@/^1.5pc/[rrru]^{\ep{z}}\ar@/_1.5pc/[rrrd]^{\mc P_z}\ar@{^{(}->>}[rr]^{\kern-20pt{p_zB^\star}}&&\mc E_z({\mc L'};\mc D'(\mf P))\ar@{^{(}->>}[ru]^{\pi^{\mathfrak{E}}_*}\ar@{^{(}->>}[rd]^{\pi_*}&\\
&&&\mc E_{\chi(z)}({\Delta};\mathrm{Maps}({{\mf X}},\C))}
\end{gather*}

We still need to define some of the maps occurring in this diagram.

\begin{definition}\label{def:projections}
Consider the projection onto the first edge, respectively vertex, given by ${\pi}^{\mathfrak{E}}\colon \mf P\to{{\mf E}},\quad (\vec{e}_1,\vec{e}_2,\ldots)\mapsto \vec{e}_1$ respectively
$\pi\coloneqq\iota\circ{\pi}^{\mathfrak{E}}\colon \mf P\to{\mf X}$
and their canonical pullback maps
\begin{align*}
\pi^*\colon\mathrm{Maps}({\mf X},\C)\to\mathrm{Maps}(\mf P,\C),&\quad f\mapsto f\circ\pi,\\
\pi^{\mathfrak{E},*}\colon\mathrm{Maps}({{\mf E}},\C)\to\mathrm{Maps}(\mf P,\C),&\quad \varphi\mapsto \varphi\circ\pi^{\mathfrak{E}}.
\end{align*}
\end{definition}

\begin{proposition}\label{prop:push-forwards_dist_proj}
  There are well-defined pushforwards
\begin{align*}
  \pi_*\coloneqq(\pi^*)'&\colon \mc D'(\mf P)\to\mathrm{Maps}_c({\mathfrak X},\C)'\\
\pi^{\mathfrak{E}}_{*}\coloneqq(\pi^{\mathfrak{E},*})'&\colon \mc D'(\mf P)\to \mathrm{Maps}_c({{\mf E}},\C)'.
\end{align*}
\end{proposition}

\begin{proof}
  Note that the sets $\pi^{-1}(x)\subseteq \mf P$ and  $(\pi^{\mf E})^{-1}(\vec{e})\subseteq \mf P$
  are compact and open for each $x\in \mf X$ and $\vec{e}\in \mf E$. Therefore, the images of $\pi^*$ and $\pi^{\mf E, *}$ are contained in $C^\mathrm{lc}(\mf P)$.  Moreover, this shows that functions with compact supports get mapped to functions with compact supports. Thus we have
  $\pi^{*}\colon \mathrm{Maps}_c({\mathfrak X},\C)\to C^\mathrm{lc}_c(\mf P)$ and
  $\pi^{\mf E,*}\colon \mathrm{Maps}_c({\mathfrak X},\C)\to C^\mathrm{lc}_c(\mf P)$. Passing to the dual maps shows the claim.
\end{proof}

\begin{remark}\label{rem:push_forward_restr}
       Since $\mathfrak{X}$ (as well as $\mathfrak{E}$) carries the discrete topology, compactly supported functions are finitely supported. Thus, $\mathrm{Maps}_c({\mathfrak X},\C)'\cong \mathrm{Maps}(\mathfrak{X},\mathbb{C})$ and the maps $\pi_*$ and $\pi^{\mathfrak{E}}_*$ are uniquely described by
\begin{align*}
\pi_*(\lambda)(\delta_x)=\langle\lambda,\delta_x\otimes\mathbbm{1}_{{\Omega}}\rangle\text{ resp.\@ }\pi^{\mathfrak{E}}_*(\lambda)(\delta_{\vec{e}})=\langle\lambda,\delta_{\iota(\vec{e})}\otimes\mathbbm{1}_{\partial_+\vec{e}}\rangle.
\end{align*}

In fact, by definition we have $\pi^{\mathfrak{E}}_*(\lambda)(\delta_{\vec{e}})=\langle\lambda,\delta_{\vec{e}}\circ\pi^{\mathfrak{E}}\rangle$ with
\begin{gather*}
    (\delta_{\vec{e}}\circ\pi^{\mathfrak{E}})(\vec{e}_1,\vec{e}_2,\ldots)=\delta_{\vec{e}}(\vec{e}_1)=\delta_{\iota(\vec{e})}(\iota(\vec{e}_1))\mathbbm{1}_{\partial_{+}\vec{e}}(B(\vec{e}_1,\vec{e}_2,\ldots))=(\delta_{\iota(\vec{e})}\otimes\mathbbm{1}_{\partial_+\vec{e}})(\vec{e}_1,\vec{e}_2,\ldots).
\end{gather*}
Summing this equation over all edges $\vec{e}$ with $\iota(\vec{e})=x$ proves the other part.

\end{remark}

\begin{proposition}\label{prop:poisson_proj}
        For each $0\not=z\in\C$ we have
\begin{gather*}
\mc P_z=\pi_*\circ p_zB^\star\quad\text{and}\quad\ep{z}=\pi^{\mathfrak{E}}_*\circ p_zB^\star.
\end{gather*}
In particular, $\pi_*$ respectively $\pi^{\mathfrak{E}}_*$ maps $\mathcal E_z({\mathcal L}';\mc D'(\mf P))$ into $\mc E_{\chi(z)}(\Delta;\mathrm{Maps}({\mf X},\C))$, respectively $\mc E_z(\el;\mathrm{Maps}({\mf E},\C))$, by \cite[Rem.~4.2]{BHW21}, respectively  Proposition~\ref{prop:cpt_picture}.
\end{proposition}

\begin{proof}
Denoting by $\vec{e}_1(x,\omega)$ for $(x,\omega)\in{\mf X}\times{\Omega}$ the first edge on the chain from $x$ to $\omega$, we calculate for $\varphi\in\mathrm{Maps}_c({{\mf E}},\C)$ and $\mu\in\mc D'({\Omega})$
\begingroup
\allowdisplaybreaks
\begin{align*}
  \langle\pi^{\mathfrak{E}}_*(p_z B^\star \mu), \varphi\rangle&=\langle p_zB^\star\mu,\pi^{\mathfrak{E},*}(\varphi)\rangle=\langle B^\star\mu,p_z{\pi}^{\mathfrak{E},*}(\varphi)\rangle\\
&=\langle \mu,\sum_{x\in{\mf X}}p_z(x,\bigcdot)\varphi(\vec{e}_1(x,\bigcdot))\rangle\\
&=\sum_{x\in{\mf X}}\int_{{\Omega}}p_z(x,\omega)\varphi(\vec{e}_1(x,\omega))\intd\mu(\omega)\\
&=\sum_{x\in{\mf X}}\sum_{\substack{y\in{\mf X}\\d(x,y)=1}}\int_{\partial_+(x,y)}p_z(x,\omega)\intd\mu(\omega)\varphi((x,y))\\
&=\sum_{\vec{e}\in{{\mf E}}}\int_{\partial_+\vec{e}}p_z(\iota(\vec{e}),\omega)\intd\mu(\omega)\varphi(\vec{e})\\
&=\sum_{\vec{e}\in{{\mf E}}}\int_{{\Omega}}\pke{z}(\vec{e},\omega)\intd\mu(\omega)\varphi(\vec{e})\\
&=\sum_{\vec{e}\in{{\mf E}}}(\ep{z}\mu)(\vec{e})\varphi(\vec{e})=\langle \ep{z}\mu,\varphi\rangle.
\end{align*}
\endgroup
For $x\in \mathfrak{X}$ and $\mu\in \mc D'(\Omega)$ this immediately implies
\begin{gather*}
  \langle\pi_*(p_z B^\star \mu), \delta_{x}\rangle=\langle p_z B^\star\mu, \delta_{x}\circ\pi\rangle=\langle p_z B^\star\mu, \sum_{\iota(\vec{e})=x}\delta_{\vec{e}}\circ\pi^{\mathfrak{E}}\rangle=\sum_{\iota(\vec{e})=x}\ep{z}(\mu)(\vec{e})=\mc P_{z}(\mu)(x).\qedhere
\end{gather*}
\end{proof}

\subsection{Mapping properties of edge Poisson transforms}

For $z\not\in \{-1,0,1\}$, by \cite[Thm.~4.7]{BHW21}, the Poisson transform $\mathcal P_z$ is an isomorphism $\mc D'(\Omega)\to \mathcal E_{\chi(z)}(\Delta; \mathrm{Maps}({\mathfrak X},\C))$. In this subsection we prove a similar result for the edge Poisson transform and the edge Laplacian. We start by proving that images of edge Poisson transforms are eigenfunctions of $\el$.

\begin{lemma}\label{la:PT_eigenfunctions}
For each $0\neq z\in\C$ we have
\begin{gather*}
\ep{z}(\mc D'(\Omega))\subseteq\mc E_z(\el;\mathrm{Maps}({\mf E},\C)).
\end{gather*}
\end{lemma}

\begin{proof}
Let $\vec{e}\in\mathfrak{E}$ and $\mu\in\mathcal{D}'(\Omega)$. Then
\begin{align*}
    z\ep{z}(\mu)(\vec{e})+\ep{z}(\mu)(\eop{e})&=z\int_{\partial_{+}\vec{e}}z^{\langle\iota(\vec{e}),\omega\rangle}\intd\mu(\omega)+\int_{\partial_{+}\eop{e}}z^{\langle\tau(\vec{e}),\omega\rangle}\intd\mu(\omega)\\
&=\int_{\partial_{+}\vec{e}}z^{\langle\tau(\vec{e}),\omega\rangle}\intd\mu(\omega)+\int_{\partial_{+}\eop{e}}z^{\langle\tau(\vec{e}),\omega\rangle}\intd\mu(\omega)\\
&=\sum_{\iota(\vec{e}')=\tau(\vec{e})}\ep{z}(\mu)(\vec{e}\,'),
\end{align*}
where we used that for each $\omega\in\partial_{+}\vec{e}$ and $\xi\in[o,\omega[\cap[\tau(\vec{e}),\omega[$
\begin{gather*}
    \langle\tau(\vec{e}),\omega\rangle=d(o,\xi)-d(\tau(\vec{e}),\xi)=d(o,\xi)-d(\iota(\vec{e}),\xi)+1=\langle\iota(\vec{e}),\omega\rangle+1.\qedhere
\end{gather*}
\end{proof}

We shall first establish the bijectivity of the restricted Poisson transform $\epo{z}$ onto the $z$-eigenspace $\mc E_z(\el;\mathrm{Maps}({\mf E}_o,\C))$ of $\el$ in $\mathrm{Maps}({\mf E}_o,\C)$. This will serve as a foundation for our subsequent proof of the bijectivity of $\ep{z}$.

\begin{lemma}\label{la:PT_measure}
For each $z\in\C\setminus\{0\},\ \mu\in\mc{D}'(\Omega)$ and $\vec{e}\in{\mf E}_o$ we have
\begin{gather*}
\epo{z}(\mu)(\vec{e})=z^{d(o,\iota(\vec{e}))}\mu(\partial_+\vec{e}).
\end{gather*}
\end{lemma}

\begin{proof}
For each $\mu\in\mc{D}'(\Omega)$ and $\vec{e}\in{\mf E}_o$ we have
\begin{align*}
\epo{z}(\mu)(\vec{e})
=\int_{\partial_+\vec{e}}z^{\langle\iota(\vec{e}),\omega\rangle}\intd\mu(\omega)
=\int_{\partial_+\vec{e}}z^{d(o,\iota(\vec{e}))}\intd\mu(\omega)=z^{d(o,\iota(\vec{e}))}\mu(\partial_+\vec{e})
\end{align*}
since $\iota(\vec{e})\in[o,\omega[$ for each $\omega\in\partial_+\vec{e}$.
\end{proof}

\begin{proposition}
Let $z\neq 0$. Then the map
\begin{gather*}
\Phi_z\colon\mc E_z(\el;\mathrm{Maps}({\mf E}_o,\C))\to\mc E_1(\el;\mathrm{Maps}({\mf E}_o,\C)),\quad \Phi_z(f)(\vec{e})\coloneqq z^{-d(o,\iota(\vec{e}))}f(\vec{e})
\end{gather*}
is a linear isomorphism with inverse
\begin{gather*}
\Psi_z\colon\mc E_1(\el;\mathrm{Maps}({\mf E}_o,\C))\to \mc E_z(\el;\mathrm{Maps}({\mf E}_o,\C)),\quad \Psi_z(f)(\vec{e})\coloneqq z^{d(o,\iota(\vec{e}))}f(\vec{e}).
\end{gather*}
\end{proposition}

\begin{proof}
We first prove that $\Phi_z$ maps into $\mc E_1(\el;\mathrm{Maps}({\mf E}_o,\C))$. For $f\in \mc E_z(\el;\mathrm{Maps}({\mf E}_o,\C))$ and $\vec{e}\in{\mf E}_o$ we calculate
\begin{align*}
\el\Phi_z(f)(\vec{e})&=\sum_{\vec{e}'\in\Omega_{\vec{e}}^1}\Phi_z(f)(\vec{e}\,')=\sum_{\vec{e}'\in\Omega_{\vec{e}}^1}z^{-d(o,\iota(\vec{e}\,'))}f(\vec{e}\,')\\
&=z^{-d(o,\iota(\vec{e}))-1}\el f(\vec{e})\\
&=z^{-d(o,\iota(\vec{e}))}f(\vec{e})=\Phi_z(f)(\vec{e}).
\end{align*}
Similarly, we have for $g\in\mc E_1(\el;\mathrm{Maps}({\mf E}_o,\C))$ and $\vec{e}\in{\mf E}_o$
\begin{align*}
\el\Psi_z(g)(\vec{e})&=\sum_{\vec{e}'\in\Omega_{\vec{e}}^1}\Psi_z(g)(\vec{e}\,')=\sum_{\vec{e}'\in\Omega_{\vec{e}}^1}z^{d(o,\iota(\vec{e}\,'))}g(\vec{e}\,')\\
&=z^{d(o,\iota(\vec{e}))+1}\el g(\vec{e})\\
&=zz^{d(o,\iota(\vec{e}))}g(\vec{e})=z\Psi_z(g)(\vec{e}).
\end{align*}
Obviously, $\Phi_z$ and $\Psi_z$ are each others inverses.
\end{proof}

\begin{proposition}\label{prop:PT_diff_param}
For $z\in \C\setminus\{0\}$  the diagram
\begin{gather*}
\xymatrix{
\mc D'(\Omega) \ar[rr]^{\epo{z}} \ar[rrd]_{\epo{1}}&&\mc E_z(\el;\mathrm{Maps}({\mf E}_o,\C)) \ar[d]^{\Phi_z}\\
&&\mc E_1(\el;\mathrm{Maps}({\mf E}_o,\C))
}
\end{gather*}
commutes.
\end{proposition}

\begin{proof}
For each $\mu\in\mc D'(\Omega)$ and $\vec{e}\in{\mf E}_o$ we have
\begin{gather*}
\Phi_z(\epo{z}(\mu))(\vec{e})=z^{-d(o,\iota(\vec{e}))}\epo{z}(\mu)(\vec{e})=\mu(\partial_+\vec{e})=\epo{1}(\mu)(\vec{e})
\end{gather*}
by Lemma \ref{la:PT_measure}.
\end{proof}

We will make use of the following generalization of Lemma \ref{la:eigenspace_characterization1}.

\begin{lemma}\label{la:eigenspace_op_edges}
Let $f\in\mc E_z(\el;\mathrm{Maps}({\mf E},\C))$ for some $z\in\mathbb{C}$ and, for $x\in\mathfrak{X}$, define
\begin{gather*}
    \mathfrak{E}_{x}\coloneqq\{\vec{e}\in\mathfrak{E}\mid\vec{e}\text{ points away from $x$}\}.
\end{gather*}
Then, for each $x\in\mathfrak{X}$ and each $\vec{e}\in\mathfrak{E}\setminus\mathfrak{E}_{x}$,
\begin{gather}\label{eq:eigenspace_op_edges}
    z^{n}f(\vec{e})=(z^2-1)\sum_{k=2}^{n}z^{n-k}f(\eop{e}_k)-z^{n-1}f(\eop{e})+\sum_{\iota(\vec{e}\,')=x}f(\vec{e}\,'),
\end{gather}
where $(\vec{e}_1=\vec{e},\vec{e}_2,\ldots,\vec{e}_{n})$ denotes the edges on the path $[\iota(\vec{e}),x]$.
\end{lemma}

\begin{proof}
We prove the result by induction on the length of the path. If $\vec{e}\in\mathfrak{E}\setminus\mathfrak{E}_{x}$ is an edge with $\tau(\vec{e})=x$, Equation \eqref{eq:eigenspace_op_edges} follows from Lemma \ref{la:eigenspace_characterization1} with $\vec{e}_1=\vec{e}$.

Now let $\vec{e}\in\mathfrak{E}\setminus\mathfrak{E}_{x}$ be such that $d(x,\iota(\vec{e}))=n+1$ for some $n\in\mathbb{N}$. Then Lemma \ref{la:eigenspace_characterization1} applied to $(\vec{e},\vec{e}_2)$ and the induction hypothesis imply
\begin{align*}
    z^{n+1}f(\vec{e})&=z^{n}f(\vec{e}_2)+z^{n+1}f(\eop{e}_2)-z^{n}f(\eop{e})\\
&=(z^2-1)\sum_{k=2}^{n}z^{n-k}f(\eop{e}_{k+1})+(z^{n+1}-z^{n-1})f(\eop{e}_2)+\sum_{\iota(\vec{e}\,')=x}f(\vec{e}\,')-z^{n}f(\eop{e})\\
&=(z^2-1)\sum_{k=2}^{n+1}z^{n+1-k}f(\eop{e}_k)-z^{n}f(\eop{e})+\sum_{\iota(\vec{e}\,')=x}f(\vec{e}\,').\qedhere
\end{align*}
\end{proof}

\begin{remark}\label{rem:L_spaces}
    For $z=1$ Lemma \ref{la:eigenspace_op_edges} implies that $\mc E_1(\el;\mathrm{Maps}({\mf E},\C))$ agrees with the space
\begin{gather*}
 L(\mathfrak{E})\coloneqq\{f\colon\mathfrak{E}\to\mathbb{C}\mid\exists\, c\in\mathbb{C}\colon \forall\, x\in\mathfrak{X},\, \vec{e}\in\mathfrak{E}\colon f(\vec{e})+f(\eop{e})=c=\sum_{\iota(\vec{e}\,')=x}f(\vec{e}\,')\}
\end{gather*}
from \cite[Def.\@ 3.12, Rem.\@ 3.13]{BHW21}. Moreover, by \cite[Thm.\@ 3.15]{BHW21}, this space is naturally isomorphic to $\mathcal{D}'(\Omega)$ via the map $\mu \mapsto (\vec{e} \mapsto \mu(\partial_{+}\vec{e}))$ and thus induces an isomorphism $\mathcal{D}'(\Omega)\cong \mc E_1(\el;\mathrm{Maps}({\mf E},\C))$.
\end{remark}

\begin{proposition}\label{prop:iso_o}
For each $z\neq 0$ let $\eb{z}$ denote the map
\begin{gather*}
\eb{z}\colon\mc E_z(\el;\mathrm{Maps}({\mf E}_o,\C))\to\mc D'(\Omega),\quad\forall\vec{e}\in{\mf E}_o\colon\eb{z}f(\partial_+\vec{e})\coloneqq z^{-d(o,\iota(\vec{e}))}f(\vec{e}).
\end{gather*}
Then $\eb{z}$ is well-defined (i.e.\@~it actually defines a finitely additive measure) and inverse to $\epo{z}$. In particular, the edge Poisson transform $\epo{z}$ is a linear isomorphism from $\mc D'(\Omega)$ to $\mc E_z(\el;\mathrm{Maps}({\mf E}_o,\C))$.
\end{proposition}

\begin{proof}
If $\eb{z}$ is well-defined, Lemma~\ref{la:PT_measure} immediately implies that $\eb{z}$ is inverse to $\epo{z}$. For the well-definedness note that the following diagram is commutative:
\begin{gather*}
\xymatrix{
\mc E_z(\el;\mathrm{Maps}({\mf E}_o,\C))\ar[r]^{\Phi_z} \ar[rd]_{\eb{z}}&\mc E_1(\el;\mathrm{Maps}({\mf E}_o,\C))\ar[d]^{\eb{1}}\\
&\mc D'(\Omega)
}
\end{gather*}
It thus suffices to prove that $\eb{1}$ is well-defined. In the case of $z=1$, however, we know from Remark \ref{rem:L_spaces} that $\ep{1}$ is bijective with $(\ep{1})^{-1}(f)(\partial_{+}\vec{e})= f(\vec{e})$. Composing $(\ep{1})^{-1}$ with the isomorphism $\mc E_1(\el;\mathrm{Maps}({\mf E}_o,\C))\cong\mc E_1(\el;\mathrm{Maps}({\mf E},\C))$ (see Lemma \ref{la:eigenspace_op_edges}) now gives $\eb{1}$ and in particular proves its well-definedness.
\end{proof}

We can now determine the image of $\ep{z}$.

\begin{proposition}\label{thm:PT_iso}
Let $0\neq z\in\C$. Then $\ep{z}$ is an isomorphism onto
\begin{gather*}
\mc E_z(\el;\mathrm{Maps}({\mf E},\C))=\ker(\el-z)\subseteq\mathrm{Maps}({\mf E},\C).
\end{gather*}
\end{proposition}

\begin{proof}
Let $f\in\mc E_z(\el;\mathrm{Maps}({\mf E},\C))$, consider $\mu\coloneqq\eb{z}(f|_{\mathfrak{E}_{o}})$ and set $\varphi\coloneqq\ep{z}(\mu)$. Then, by Proposition \ref{prop:iso_o}, we have $\varphi(\vec{e})=\epo{z}(\mu)(\vec{e})=f(\vec{e})$ for each edge $\vec{e}\in\mathfrak{E}_{o}$. But now Lemma~\ref{la:eigenspace_op_edges} and Lemma~\ref{la:PT_eigenfunctions} imply that $f$ and $\varphi$ also agree on $\mathfrak{E}\setminus\mathfrak{E}_{o}$.
\end{proof}

\begin{remark}\label{rem:ExceptionalParameters}
	Note that Proposition~\ref{thm:PT_iso} applies to $z=\pm 1$, i.e.\@ the parameters for which \cite[Thm.~4.7]{BHW21} fails to guarantee that the scalar Poisson transformation is bijective. In fact, for $(q+1)$-regular trees the Poisson transformation is neither injective nor surjective for $z=\pm1$. For $z=1$ the Poisson kernel $p_1(x,\omega)$ is constant, so the image of $\mc P_1$ is the one-dimensional space of constant functions. For $z=-1$, the Poisson kernel equals $p_{-1}(x,\omega)=(-1)^{d(o,x)}$, so the image of $\mc P_{-1}$ is the one-dimensional space of functions whose values at neighboring vertices sum to $0$. On the other hand, $\chi(z)$ is invariant under the transformation $z\mapsto\frac{q}{z}$, and the Poisson transforms $\mc P_{\pm q}:\mc D'(\Omega)\to\mc E_{\chi(\pm q)}(\Delta;\mathrm{Maps}(\mf X,\C))$ are bijective by \cite[Thm.~4.7]{BHW21}. In particular, $\mc E_{\chi(\pm1)}(\Delta;\mathrm{Maps}(\mf X,\C))$ is isomorphic to $\mc D'(\Omega)$ which has at least dimension $2$, so $\mc E_{\chi(\pm1)}(\Delta;\mathrm{Maps}(\mf X,\C))$ is strictly larger than the one-dimensional image of $\mc P_{\pm1}$.
\end{remark}

\subsection{Relating edge and scalar Poisson transforms}

In this subsection we  provide a commutative diagram relating scalar Poisson transforms to edge Poisson transforms. We first remark that, for each $x\in\mathfrak{X}$ and $\mu\in \mathcal{D}'(\Omega)$,
\begin{gather}\label{eq:pt_ept}
    \mathcal{P}_z(\mu)(x)=\int_\Omega p_{z}(x,\omega)\intd\mu(\omega)=\sum_{\iota(\vec{e})=x}\int_{\partial_+\vec{e}}p_z(x,\omega)\intd\mu(\omega)=\sum_{\iota(\vec{e})=x}\ep{z}(\mu)(\vec{e}).
\end{gather}
By Proposition~\ref{thm:PT_iso} this induces a map
\begin{gather*}
    \mathcal{F}\colon \mc E_z(\el;\mathrm{Maps}({\mf E},\C))\to \mc E_{\chi(z)}(\Delta;\mathrm{Maps}(\mf X,\C)),\quad \mathcal{F}(f)(x)\coloneqq\sum_{\iota(\vec{e})=x}f(\vec{e}).
\end{gather*}
The aim of this section is to describe, if it exists, an inverse of $\mathcal{F}$. Similar to Equation \eqref{eq:pt_ept} this gives rise to a relation between $\mathcal{P}_z$ and $\ep{z}$, but in this case shows how to express $\ep{z}$ in terms of $\mathcal{P}_z$.

In \cite[Observation 4.3, Rem.\@ 4.2, Lem.\@ 4.5]{BHW21} is was shown that for $z^2\not\in \{0,1\}$ and the vertex Laplacian $\Delta$ of $\mf G$ from Subsection~\ref{rem:Laplacian} one can find a \emph{boundary value map} $\beta_z:\mc E_{\chi(z)}(\Delta;\mathrm{Maps}(\mf X,\C))\to L({\mf E}_o)\coloneqq\mc E_1(\el;\mathrm{Maps}({\mf E}_o,\C))$ essentially inverting the scalar Poisson transform $\mc P_{z}$.

Combining Lemma~\ref{la:eigenspace_op_edges} with Proposition \ref{prop:PT_diff_param} yields the following proposition.

\begin{proposition}\label{prop:PT_vv_sc}
With the boundary value map $\beta_z$ for $z^2\not\in\{0,1\}$ and the Poisson transform $\mc P_z$ from \cite[Observation 4.3, Rem.\@ 4.2, Lem.\@ 4.5]{BHW21} the diagram
\begin{gather*}
\xymatrix{
\mc E_{\chi(z)}(\Delta;\mathrm{Maps}(\mf X,\C))\ar[r]^-{\overline\beta_z}&\mc E_1(\el;\mathrm{Maps}({\mf E},\C))\ar[r]^-{\overline\Psi_z}&\mc E_z(\el;\mathrm{Maps}({\mf E},\C))\\
&\mc D'(\Omega)\ar[ul]^{\mc P_z}\ar[u]^{\ep{1}}\ar[ur]_{\ep{z}}&
}
\end{gather*}
commutes, where $\overline\beta_{z}$ resp.\@ $\overline\Psi_z$ denote the extensions of $\beta_z$ resp.\@ $\Psi_z$ to $\mathfrak{E}$ in the sense of Lemma \ref{la:eigenspace_op_edges}.
Moreover, for $z\in\C\setminus\{0\}$ the scalar and the edge Poisson transforms are related by
\begin{gather}\label{eq:rel_PTs}
z\mc P_z(\mu)(\tau(\vec{e}))-\mc P_z(\mu)(\iota(\vec{e}))=(z^2-1)\ep{z}(\mu)(\vec{e})
\end{gather}
for each $\mu\in\mc D'(\Omega)$ and $\vec{e}\in{\mf E}$.
\end{proposition}

\begin{proof}
The commutativity of the left triangle restricted to $\mathfrak{E}_o$ is shown in \cite[Thm.\@~4.7]{BHW21}; on $\mathfrak{E}\setminus\mathfrak{E}_o$ it follows from Lemma \ref{la:eigenspace_op_edges} and \ref{la:PT_eigenfunctions}. The commutativity of the right triangle follows similarly from Proposition~\ref{prop:PT_diff_param}. For the relation between the Poisson transforms for $\vec{e}\in\mathfrak{E}_{o}$ we use the proof of \cite[La.\@ 4.4]{BHW21} (which only uses the assumption that $z^2\not\in\{0,1\}$ for the definedness of $\beta_z$), to write
\begin{gather*}
z\mc P_z(\mu)(\tau(\vec{e}))-\mc P_z(\mu)(\iota(\vec{e}))=(z^2-1)z^{d(o,\iota(\vec{e}))}\epo{1}(\mu)(\vec{e})=(z^2-1)(\Psi_z\circ\epo{1}(\mu))(\vec{e})
\end{gather*}
and use the commutativity of the right triangle. Now consider $\vec{e}\in\mathfrak{E}\setminus\mathfrak{E}_{o}$ and note that
\begin{gather*}
\mc P_z(\mu)(\tau(\vec{e}))=\sum_{\iota(\vec{e}\,')=\tau(\vec{e})}\ep{z}(\mu)(\vec{e}\,')=z\ep{z}(\mu)(\vec{e})+\ep{z}(\mu)(\eop{e}).
\end{gather*}
Multiplying this equation by $(z^2-1)$ and subtracting Equation \eqref{eq:rel_PTs} for $\eop{e}$ we arrive at
\begin{gather*}
    z^2\mathcal{P}_z(\mu)(\tau(\vec{e}))-z\mathcal{P}_z(\mu)(\iota(\vec{e}))=(z^2-1)z\ep{z}(\mu)(\vec{e}).\qedhere
\end{gather*}
\end{proof}

Proposition \ref{prop:PT_vv_sc} shows that for $z^2=1$ the edge Poisson transform cannot be written in terms of $\mathcal{P}_z$. The parameters $z=\pm 1$, for which the Poisson transform $\mc P_z$ is defined, but not bijective, are called \emph{exceptional parameters}. They had to be excluded from the discussion  of quantum-classical correspondences in \cite{BHW23}.

\subsection{Intertwining properties of Poisson transforms}\label{subsubsec:intertwining}

For any $0\neq z\in \C$ we can define a representation of the tree automorphism group $G\coloneqq\mathrm{Aut}({\mathfrak G})$ on $\mc D'(\Omega)$ by the following construction: For each $g\in G$ and $\omega\in\Omega$ let $N_z(g,\omega)\coloneqq z^{-\langle go,g\omega\rangle}$ and define
\begin{gather}\label{eq:pi_z_def}
    \pi_z(g)\mu\coloneqq g_{*}(N_z(g,\bigcdot)\mu)
\end{gather}
on $\mathcal{D}'(\Omega)$, where $N_z(g,\bigcdot)\mu$ denotes the multiplication of $\mu$ with the (locally constant) function $N_z(g,\bigcdot)$ on $\Omega$ and $g_*$ denotes the pushforward of the resulting linear form in $\mc D'(\Omega)$ by the automorphism $g$. Then, for $g_1,\, g_2\in G$ and $\omega\in\Omega$, the horocycle identity \eqref{eq:horocycle_id} implies $N_{z}(g_1g_2,\omega)=N_{z}(g_1,g_2\omega)N_{z}(g_2\omega)$ and in turn that $\pi_{z}$ defines a representation (see \cite[\S 11]{BHW23}).

\begin{lemma}\label{la:intertwine_cpt_pic}
The map $p_{z}B^\star$ from Proposition \ref{prop:cpt_picture} intertwines $\pi_{z}$ with the left regular representation on $\mc D'(\mf P)$.
\end{lemma}

\begin{proof}
For each $\varphi\in C^\mathrm{lc}_c(\mf P, \C)$ we calculate
\begin{align*}
    (p_{z}B^\star)(\pi_{z}(g)\mu)(\varphi)&=B^\star(\pi_{z}(g)\mu)(p_{z}\varphi)=(\pi_{z}(g)\mu)\left(\sum_{x\in{\mathfrak{X}}} p_{z}(x,\bigcdot)\varphi(x,\bigcdot)\right)\\
&=\mu\left(\sum_{x\in{\mathfrak{X}}}N_{z}(g,\bigcdot)p_{z}(x,g\bigcdot)\varphi(x,g\bigcdot)\right).
\end{align*}
Now the horocycle identity \eqref{eq:horocycle_id} implies $N_{z}(g,g\omega)p_{z}(x,g\omega)=p_z(g^{-1}x,\omega)$ so that
\begin{align*}
    (p_{z}B^\star)(g\mu)(\varphi)&=\mu\left(\sum_{x\in{\mathfrak{X}}}p_z(g^{-1}x,\omega)\varphi(x,g\bigcdot)\right)=\mu\left(\sum_{x\in{\mathfrak{X}}}p_z(x,\omega)\varphi(gx,g\bigcdot)\right)\\
&=B^{\star}(\mu)(p_{z}\cdot (g^{-1}\varphi))=(g(p_{z}B^{\star}(\mu)))(\varphi).\qedhere
\end{align*}
\end{proof}

\begin{remark}
    Note that, by Lemma \ref{la:intertwine_cpt_pic}, taking $\Gamma$-invariants in Proposition \ref{prop:cpt_picture} yields the identification of resonant states with $\Gamma$-invariant measures from \cite[Prop.\@ 11.4]{BHW23}.
\end{remark}

Lemma \ref{la:intertwine_cpt_pic} immediately implies the $G$-equivariance of $\ep{z}$.

\begin{lemma}\label{la:PT_intertwining}
The edge Poisson transform $\ep{z}$ intertwines the representation $\pi_z$ with the left regular representation on $\mathrm{Maps}({\mf E},\C)$.
\end{lemma}

\begin{proof}
By Lemma \ref{la:intertwine_cpt_pic} and Proposition \ref{prop:poisson_proj} it suffices to prove that $\pi^{\mathfrak{E}}_*$ is $G$-equivariant. However, note that for $f\in \mc D'(\mf P)$ and $\vec{e}\in\mathfrak{E}$
\begin{gather*}
    \pi^{\mathfrak{E}}_*(gf)(\delta_{\vec{e}})=f(g^{-1}(\delta_{\vec{e}}\circ \pi^{\mathfrak{E}}))=f((g^{-1}\delta_{\vec{e}})\circ \pi^{\mathfrak{E}})=\pi^{\mathfrak{E}}_*(f)(\delta_{g^{-1}\vec{e}})=g\pi^{\mathfrak{E}}_*(f)(\delta_{\vec{e}}).\qedhere
\end{gather*}
\end{proof}

\begin{remark}\label{rem:Poisson1}
As a special case we obtain the intertwining property of the Poisson transform from \cite[Obs.\@~11.2]{BHW23}. Indeed, for each $x\in\mathfrak{X}$ and $g\in G$
\begin{gather*}
    \mathcal{P}_z(\pi_{z}(g)\mu)(x)=\sum_{\vec{e}\in\mathfrak{E}\colon\iota(\vec{e})=x}\ep{z}(\pi_{z}(g)\mu)(\vec{e})=\sum_{\vec{e}\in\mathfrak{E}\colon\iota(\vec{e})=x}\ep{z}(\mu)(g^{-1}\vec{e})=\mathcal{P}_{z}(\mu)(g^{-1}x).
\end{gather*}
\end{remark}

\begin{remark}
By Proposition \ref{prop:PT_vv_sc} and \ref{prop:poisson_proj} we have
\begin{gather*}
\ep{z}\circ\mc P_z^{-1}=\overline\Psi_z\circ\overline\beta_z=\pi^{\mathfrak{E}}_*\circ\pi_*^{-1}
\end{gather*}
as maps from $\mc E_{\chi(z)}(\Delta;\mathrm{Maps}({\mf X},\C))$ to $\mc E_z(\el;\mathrm{Maps}({{\mf E}},\C))$. Then we obtain
\begin{gather*}
   \sum_{\vec{e}\in{{\mf E}}}(\overline\Psi_z\circ\overline\beta_z)(\varphi)(\vec{e})=\sum_{\vec{e}\in{{\mf E}}}\frac{z}{z^2-1}\varphi(\tau(\vec{e}))-\frac{1}{z^2-1}\varphi(\iota(\vec{e}))=\frac{1}{1+z}\sum_{x\in{\mf X}}(q_x+1)\varphi(x).
\end{gather*}
\end{remark}

\section{Applications to the topology of finite graphs}\label{sec:appl_topology}

In this section we consider a connected finite graph $\mf G=(\mf X,\mf E)$ without dead ends  and its
\emph{universal cover} $\widetilde{\mf G}=(\widetilde{\mf X},\widetilde{\mf E})$, by which we mean the graph obtained from taking the simply connected covering of the undirected version of $\mf G$ (identifying each edge with its opposite) and then replacing each unoriented edge by two opposite oriented edges. Then $\widetilde{\mf G}$ is a tree of bounded degree as in Section~\ref{sec:trees}. We denote its boundary by $\widetilde{\Omega}$, the group of deck transformations by $\Gamma\leq G\coloneqq\mathrm{Aut}(\widetilde{\mathfrak{X}})$ and
the non-backtracking transition matrix of $\widetilde{\mf G}$ by $\widetilde{\mathbb{A}}$. Further, we denote the canonical projection from $\widetilde{\mathfrak{X}}$ to $\mathfrak{X}$ by $\pi_{\mathfrak{X}}$. Note that the action of $G$ on $\widetilde{\mathfrak{X}}$ extends to actions on $\widetilde{\mathfrak{E}}$ and $\widetilde{\mf P}$ by acting on each vertex contained in the edge, respectively chain. Similarly, $\pi_{\mathfrak{X}}$ induces a projection $\pi_{\mathfrak{E}}\colon\widetilde{\mathfrak{E}}\to \mathfrak{E}$ such that $\iota\circ\pi_{\mathfrak{E}}=\pi_{\mathfrak{X}}\circ\widetilde{\iota}$ and $\tau\circ\pi_{\mathfrak{E}}=\pi_{\mathfrak{X}}\circ\widetilde{\tau}$, where $\iota,\widetilde{\iota}$ resp.\@ $\tau,\widetilde{\tau}$ denote the initial resp.\@ end point projections of $\mathfrak{G}$ and $\widetilde{\mathfrak{G}}$. Moreover,
\begin{gather*}
    \pi_{\mf P}\colon\widetilde{\mf P}\to \mf P,\qquad\pi_{\mf P}(\vec{e}_1,\vec{e}_2,\ldots)\coloneqq(\pi_{\mathfrak{E}}(\vec{e}_1),\pi_{\mathfrak{E}}(\vec{e}_2),\ldots)
\end{gather*}
defines a projection as we do not allow multiple edges in $\mathfrak{G}$. Note that $\pi_{\mf P}$ is continuous with respect to the district topology from \cite[\S~5]{BHW23} mentioned in Definition \ref{def:shift_dyn}. Finally, we let $G$ act on $\mathrm{Maps}(\widetilde{\mf P},\mathbb{C})$ by the left regular representation and on its algebraic dual $\mathrm{Maps}(\widetilde{\mf P},\mathbb{C})'$ by duality.

\subsection{\texorpdfstring{$\Gamma$}{Gamma}-invariant resonant states and boundary measures}
We now describe a connection between resonant states on the universal cover $\widetilde{\mathfrak{G}}$ and those on $\mathfrak{G}$. For this denote the transfer operator on $\widetilde{\mathfrak{G}}$ resp.\@ $\mathfrak{G}$ by $\widetilde{\mathcal{L}}$ resp.\@ $\mathcal{L}$. We adapt this notation to distinguish other objects on $\widetilde{\mathfrak{G}}$ and $\mathfrak{G}$.  Moreover, we pick a section $\sigma_{\mathfrak{X}}\colon\mathfrak{X}\to\widetilde{\mathfrak{X}}$ for $\pi_{\mathfrak{X}}$, i.e.\@ $\sigma_{\mathfrak{X}}(\mathfrak{X})\subset \widetilde{\mathfrak{X}}$ is a set of representatives for the $\Gamma$-orbits on $\widetilde{\mathfrak{X}}$. Note that $\sigma_{\mathfrak{X}}$ induces sections $\sigma_{\mf P}\colon \mf P\to\widetilde{\mf P}$ and $\sigma_{\mathfrak{E}}\colon \mathfrak{E}\to\widetilde{\mathfrak{E}}$ in the following way: If $(x_1,x_2,\ldots)$ is a chain of vertices in $\mathfrak{X}$, we set $\sigma_{\mf P}(x_1,x_2,\ldots)\coloneqq (\widetilde{x}_1,\widetilde{x}_2,\ldots)$ with $\widetilde{x}_1\coloneqq\sigma_{\mathfrak{X}}(x_1)$ and $\widetilde{x}_n$ for $n>1$ the unique vertex with $d(\widetilde{x}_{n-1},\widetilde{x}_n)=1$ such that there exists some $\gamma\in\Gamma$ with $\widetilde{x}_n=\gamma\sigma_{\mathfrak{X}}(x_n)$ (where the uniqueness follows from the fact that we do not allow multiple edges in $\mathfrak{G}$).
 Our first aim is to identify the dual space $\mc D'(\mf P)$ with the space $\mc D'(\widetilde{\mf P})^{\Gamma}$ of $\Gamma$-invariant elements in $\mc D'(\widetilde{\mf P})$.

 For functions the pullback $\pi_{\mf P}^*\colon \mathrm{Maps}(\mf P,\mathbb{C})\to\mathrm{Maps}(\widetilde{\mf P},\mathbb{C})^{\Gamma}$ is bijective and maps $C^{\mathrm{lc}}(\mf P)$ onto $(C^{\mathrm{lc}}(\widetilde{\mf P}))^{\Gamma}$ since $\pi_{\mf P}$ is continuous and preserves districts. For the dual version of this isomorphism we define (recall that $\mathfrak{G}$ is finite) the $\Gamma$-invariant pushforward
\begin{gather*}
 \pi_{\mf P,\star}\colon\left\{\begin{array}{ccc}
               C_c^{\mathrm{lc}}(\widetilde{\mf P}, \mathbb C) &\to & C^{\mathrm{lc}}(\mf P, \mathbb C)\\
               f & \mapsto & (\pi_{\mf P}^*)^{-1}\left( \sum_{\gamma\in\Gamma} \gamma\cdot f\right)\\
              \end{array}\right.
\end{gather*}
which is surjective: For $\varphi\in C^{\mathrm{lc}}(\mf P, \mathbb C)$ we have $\varphi=\pi_{\mf P,\star}(f)$ with
\begin{gather*}
  f(\mathbf{e})\coloneqq\mathbbm{1}_{\sigma_{\mf P}(\mf P)}(\mathbf{e})\pi_{\mf P}^*(\varphi)(\mathbf{e})
\end{gather*}
for $\mathbf{e}\in \mf P$.

\begin{lemma}\label{la:upstairs_downstairs_distr}
The pullback map
\begin{gather*}
 \pi_{\mf P}^\star\colon \mc D'(\mf P) \to \mc D'(\widetilde{\mf P})^\Gamma,\qquad u \mapsto u\circ \pi_{\mf P,\star}
\end{gather*}
is a well-defined isomorphism.
\end{lemma}

\begin{proof}
For the $\Gamma$-invariance of $\pi_{\mf P}^\star$ we pick $u\in \mc D'(\mf P),\ f\in C^{\mathrm{lc}}_c(\widetilde{\mf P}, \mathbb C)$ and $\gamma\in\Gamma$ to calculate
\begin{gather*}
    \langle\gamma\cdot\pi_{\mf P}^\star(u),f\rangle=\langle\pi_{\mf P}^\star(u),\gamma^{-1}\cdot f\rangle=\langle u,\pi_{\mf P,\star}(\gamma^{-1}\cdot f)\rangle=\langle u,\pi_{\mf P,\star}(f)\rangle=\langle\pi_{\mf P}^\star(u),f\rangle.
\end{gather*}

The injectivity of $\pi_{\mf P}^\star$ is immediate from the surjectivity of $\pi_{\mf P,\star}$. Now suppose that $U\in \mc D'(\widetilde{\mf P})$ is $\Gamma$-invariant and pick $f\in C^{\mathrm{lc}}_c(\widetilde{\mf P}, \mathbb C)$. Then
$$\langle U,\gamma\cdot f\rangle = \langle \gamma^{-1}\cdot U,f\rangle = \langle  U,f\rangle$$
for all $\gamma\in \Gamma$. Moreover, we may decompose $f=\sum_{j=1}^k\gamma_j\cdot f_j$ with $f_j\in C^{\mathrm{lc}}_{\sigma_{\mf P}(\mf P)}(\widetilde{\mf P}, \mathbb C)$ (i.e.\@ supported in $\sigma_{\mf P}(\mf P)$) and $\gamma_1,\ldots, \gamma_k\in \Gamma$. Now, for $F\in C^{\mathrm{lc}}(\mf P, \mathbb C)$ we set
\begin{gather}\label{eq:pullback_surj}
\langle u, F\rangle \coloneqq \langle U, (F\circ \pi_{\mf P})\mathbbm{1}_{\sigma_{\mf P}(\mf P)} \rangle
\end{gather}
and thus define a $u\in \mc D'(\mf P)$. Finally we calculate
\begin{align*}
\langle \pi_{\mf P}^\star(u),f\rangle
&=\langle u, \pi_{\mf P,\star} (f)\rangle=\langle U,  (\pi_{\mf P,\star} (\gamma_1f_1+\ldots+\gamma_kf_k)\circ \pi_{\mf P})\mathbbm{1}_{\sigma_{\mf P}(\mf P)}\rangle\\
&=\sum_{j=1}^k \langle U,  (\pi_{\mf P,\star} (\gamma_j f_j)\circ \pi_{\mf P})\mathbbm{1}_{\sigma_{\mf P}(\mf P)}\rangle\\
&=\sum_{j=1}^k \langle U,  (\pi_{\mf P,\star} (f_j)\circ \pi_{\mf P})\mathbbm{1}_{\sigma_{\mf P}(\mf P)}\rangle\\
&=\sum_{j=1}^k \langle U,  f_j\rangle=\sum_{j=1}^k \langle U, \gamma_j f_j\rangle= \langle U,  f\rangle.\qedhere
\end{align*}
\end{proof}

Next we show that $\pi_{\mf P,\star}$ and the pullback $\pi_{\mf P}^\star$ are compatible with the transfer operators so they preserve resonant states.

\begin{proposition}\label{prop:res_upstairs_downstairs}
We have
\begin{gather*}
    \pi_{\mf P}^*\circ\mathcal{L}=\widetilde{\mathcal{L}}\circ\pi_{\mf P}^*,\quad \mathcal{L}\circ\pi_{\mf P,\star}=\pi_{\mf P,\star}\circ\widetilde{\mathcal{L}}\quad\text{and}\quad\pi_{\mf P}^\star\circ\mathcal{L}'=\widetilde{\mathcal{L}}'\circ\pi_{\mf P}^\star.
\end{gather*}
In particular, $\pi_{\mf P}^\star$ induces an isomorphism $\mathcal E_z({\mathcal L}';\mc D'(\mf P))\cong\mathcal E_z({\widetilde{\mathcal L}}';\mc D'(\widetilde{\mf P})^{\Gamma})$.
\end{proposition}

\begin{proof}
Let $\varphi\in C^{\mathrm{lc}}(\mf P)$. Then
\begin{align*}
   (\pi_{\mf P}^*\circ\mathcal{L})(\varphi)(\vec{e}_1,\vec{e}_2,\ldots)&=\mathcal{L}(\varphi)(\pi_{\mathfrak{E}}(\vec{e}_1),\pi_{\mathfrak{E}}(\vec{e}_2),\ldots)\\
&=\sum_{\vec e_{\mathfrak{G}}\colon \mathbb A_{\vec e_{\mathfrak{G}}, \pi_{\mathfrak{E}}(\vec e_1)} = 1} \varphi(\vec e_{\mathfrak{G}}, \pi_{\mathfrak{E}}(\vec e_1), \pi_{\mathfrak{E}}(\vec e_2),\ldots).
\end{align*}
On the other hand,
\begin{gather*}
    (\widetilde{\mathcal{L}}\circ\pi_{\mf P}^*)(\varphi)(\vec{e}_1,\vec{e}_2,\ldots)=\sum_{\vec e_0\colon \widetilde{\mathbb A}_{\vec e_0, \vec e_1} = 1}\varphi(\pi_{\mathfrak{E}}(\vec{e}_0),\pi_{\mathfrak{E}}(\vec{e}_1),\ldots).
\end{gather*}
Since $\mathfrak{G}$ does not contain multiple edges, the projections under $\pi_{\mathfrak{E}}$ of the edges $\vec{e}_0\in\widetilde{\mathfrak{E}}$ with $\widetilde{\mathbb{A}}_{\vec{e}_0,\vec{e}_1}=1$ are exactly the edges $\vec{e}_{\mathfrak{G}}\in\mathfrak{E}$ with $\mathbb{A}_{\vec{e}_{\mathfrak{G}},\pi_{\mathfrak{E}}(\vec{e}_1)}=1$. This proves the first equality.

For the second equality note first that, for $f\in C_c^{\mathrm{lc}}(\widetilde{\mf P})$,
\begin{gather*}
    (\pi_{\mf P}^*\circ\mathcal{L}\circ(\pi_{\mf P}^*)^{-1})\left(\sum_{\gamma\in\Gamma}\gamma\cdot f\right)=(\widetilde{\mathcal{L}}\circ\pi_{\mf P}^*\circ(\pi_{\mf P}^*)^{-1})\left(\sum_{\gamma\in\Gamma}\gamma\cdot f\right)=\sum_{\gamma\in\Gamma}\gamma\cdot\widetilde{\mathcal{L}}(f).
\end{gather*}
But the left hand side equals $\pi_{\mf P}^*(\mathcal{L}(\pi_{\mf P,\star}(f)))$ so that
\begin{gather*}
    \mathcal{L}(\pi_{\mf P,\star}(f))=(\pi_{\mf P}^*)^{-1}\left(\sum_{\gamma\in\Gamma}\gamma\cdot\widetilde{\mathcal{L}}(f)\right)=\pi_{\mf P,\star}(\widetilde{\mathcal{L}}(f)),
\end{gather*}
proving the second equality.

Finally, for $u\in \mc D'(\mf P)$ and $f\in C_c^{\mathrm{lc}}(\widetilde{\mf P})$,
\begin{align*}
    (\pi_{\mf P}^\star\circ\mathcal{L}')(u)(f)&=(\mathcal{L}'(u)\circ\pi_{\mf P,\star})(f)=u((\mathcal{L}\circ\pi_{\mf P,\star})(f))=u((\pi_{\mf P,\star}\circ\widetilde{\mathcal{L}})(f))\\
&=\pi_{\mf P}^\star(u)(\widetilde{\mathcal{L}}(f))=(\widetilde{\mathcal{L}}'\circ\pi_{\mf P}^\star)(u)(f).\qedhere
\end{align*}
\end{proof}

Next we relate the (edge) Laplacians of $\mathfrak{G}$ and $\widetilde{\mathfrak{G}}$. As before we write $\pi_{\mf X}\colon\widetilde{\mf X}\to\mf X$ for the canonical projection and $\Gamma$ for the group of deck transformations. Since $\pi_{\mf X}$ is a local isomorphism we obtain a map
\begin{gather*}
\pi_{\mf X}\times\pi_{\mf X}\colon{\widetilde{\mf E}}\to{\mf E},\quad (x,y)\mapsto(\pi_{\mf X}(x),\pi_{\mf X}(y)).
\end{gather*}
Note that $\pi_{\mf X}$ induces the bijection
\begin{gather*}
(\pi_{\mf X}\times\pi_{\mf X})^*\colon\mathrm{Maps}({\mf E},\C)\to\mathrm{Maps}({\widetilde{\mf E}},\C)^{\Gamma},\quad f\mapsto f\circ(\pi_{\mf X}\times\pi_{\mf X}).
\end{gather*}

\begin{lemma}\label{la:el_local_global}
$(\pi_{\mf X}\times\pi_{\mf X})^*\el=\widetilde{\el}(\pi_{\mf X}\times\pi_{\mf X})^*$, where $\widetilde{\el}$ is the edge Laplacian for $\widetilde{\mf G}$.
\end{lemma}

\begin{proof}
Let $\vec{e}=(a,b)\in{\widetilde{\mf E}}$. Let $x_1,\ldots,x_{q_b}\in\widetilde{\mf X}$ denote the neighbors of $b$ different from~$a$. Since $\pi_{\mf X}$ is a local isomorphism the neighbors of $\pi_{\mf X}(b)$ are given by $\pi_{\mf X}(x_1),\ldots,\pi_{\mf X}(x_{q_b}),\pi_{\mf X}(a)$. Thus,
\begin{align*}
\big((\pi_{\mf X}\times\pi_{\mf X})^*\el f\big)(\vec{e})
&=(\el f)(\pi_{\mf X}(a),\pi_{\mf X}(b))
=\sum_{j=1}^{q_b}\big((\pi_{\mf X}\times\pi_{\mf X})^*f\big)(b,x_j)\\
&=(\widetilde{\el}(\pi_{\mf X}\times\pi_{\mf X})^*f)(\vec{e}).\qedhere
\end{align*}
\end{proof}

\begin{remark}\label{rem:scalar Poisson}
For any $0\not=z\in\mathbb C$ and any subgroup $\Gamma\subset \mathrm{Aut}(\widetilde{\mathfrak G})$ we introduce the space of $\Gamma$-invariant measures
\begin{equation}\label{eq:Mfa-Gamma}
\mc D'(\widetilde\Omega)^{\Gamma,z} \coloneqq \{f\in \mc D'(\widetilde\Omega)\mid \forall \gamma\in\Gamma\colon \pi_z(\gamma) f=f \}.
\end{equation}
Here, $\pi_z$ is the representation of $\mathrm{Aut}(\widetilde{\mathfrak G})$ given by \eqref{eq:pi_z_def}. Note that, as automorphisms commute with $\widetilde\Delta$ and $\widetilde{\el}$, the $\Gamma$-action leaves the eigenspaces of these operators invariant. As above, we identify $\mathrm{Maps}(\Gamma\backslash \widetilde{\mathfrak X},\mathbb{C})$ with the space $\mathrm{Maps}(\widetilde{\mathfrak X},\mathbb{C})^\Gamma$ of $\Gamma$-invariant functions on $\widetilde{\mathfrak X}$, so that the Laplacian $\widetilde \Delta$ on $\widetilde{\mathfrak G}$ induces the Laplace operator $\Delta$ on $\mathfrak G$. In view of $\widetilde \chi(z)=\chi(z)\circ\pi_{\mathfrak X}$ we see that pull-back by $\pi_{\mathfrak X}$ induces an isomorphism
\begin{gather*}
\pi_{\mathfrak X}^*\colon  \mathcal E_{\chi(z)}(\Delta;\mathrm{Maps}(\mathfrak X,\C)) \to
 \mathcal E_{\widetilde\chi(z)}(\widetilde\Delta;\mathrm{Maps}(\widetilde{\mathfrak X},\C))^\Gamma.
\end{gather*}
Since the deck transformations of $\mathfrak G$ form a subgroup of $\mathrm{Aut}(\widetilde{\mathfrak G})$, the above observation together with \cite[Thm.\@ 4.7]{BHW21} implies that $\widetilde{\mathcal P_z}$ induces an isomorphism $\mc D'(\widetilde\Omega)^{\Gamma,z}\to \mathcal E_{\chi(z)}(\Delta; \mathrm{Maps}(\mathfrak X,\C))$ for $z\in\C\setminus\{0,\pm1\}$.
\end{remark}

 A similar result holds for the edge Poisson transform.

\begin{proposition}\label{prop:vv_iso_local}
Let $0\neq z\in\C$. Then $\widetilde{\ep{z}}$ induces an isomorphism
\begin{gather*}
\mc D'(\widetilde\Omega)^{\Gamma,z}\cong\mc E_z(\el;\mathrm{Maps}({\mf E},\C)).
\end{gather*}
\end{proposition}

\begin{proof}
The intertwining property of $\widetilde{\ep{z}}$ from Lemma \ref{la:PT_intertwining} and Proposition~\ref{thm:PT_iso} implies
\begin{gather*}
\widetilde{\ep{z}}\colon\mc D'(\widetilde\Omega)^{\Gamma,z}\xrightarrow{\sim}\ker(\widetilde{\el}-z)^{\Gamma}.
\end{gather*}
Concatenating with $(\pi_{\mf X}\times\pi_{\mf X})_*^{-1}$ yields the claimed isomorphism by Lemma \ref{la:el_local_global}.
\end{proof}

\subsection{Exceptional parameter \texorpdfstring{$1$}{1} and circuit rank}

Let us now investigate the exceptional parameter $z=1$ in some more detail. We first determine the image of the scalar Poisson transform restricted to the $\Gamma$-invariant elements.

\begin{definition}
We define the \emph{circuit rank} or \emph{cyclomatic number of $\mf G$} by
\begin{gather*}
\mf c(\mf G)\coloneqq\abs{\mf E}-\abs{\mf X}+1,
\end{gather*}
where we count the pairs of opposite directed edges as one (undirected) edge. This number is equal to the number of independent cycles in $\mf G$ or, equivalently, to the minimal number of edges that must be removed from $\mf G$ to break all its cycles.
\end{definition}

\begin{theorem}\label{thm:im_P1_Gamma}
We have
\begin{gather*}
\dim_\C(\mathrm{im}(\mc P_1\vert_{\mc D'(\widetilde\Omega)^{\Gamma,1}}))=
\begin{cases}
       0&\colon\mf c(\mf G)\neq1,\\
       1&\colon\mf c(\mf G)=1.
\end{cases}
\end{gather*}
\end{theorem}

\begin{proof}
Combining Proposition~\ref{thm:PT_iso} with Remark \ref{rem:L_spaces} implies that $\mc D'(\widetilde\Omega)^{\Gamma,1}$ is isomorphic to $L({\widetilde{\mf E}})^{\Gamma}$. Note that $\pi_{\mf X}\times\pi_{\mf X}$ induces an isomorphism between $L({\widetilde{\mf E}})^{\Gamma}$ and $L({\mf E})$. Moreover, for each $\mu\in\mc D'(\widetilde\Omega)^{\Gamma,1}$ and $x\in\widetilde{\mf X}$ we have
\begin{gather*}
\mc P_1(\mu)(x)=\sum_{\iota(\vec{e}')=x}\ep{1}(\mu)(\vec{e}\,')
\end{gather*}
and, since $\ep{1}(\mu)\in L({\widetilde{\mf E}})$, for each $\vec{e}\in{\widetilde{\mf E}}$ and $x,y\in\widetilde{\mf X}$
\begin{gather}\label{eq:proof_sc_PT}
\ep{1}(\mu)(\vec{e})+\ep{1}(\mu)(\eop{e})=\mc P_1(\mu)(x)=\sum_{\iota(\vec{e}')=y}\ep{1}(\mu)(\vec{e}\,').
\end{gather}
By the $\Gamma$-invariance of $\ep{1}(\mu)$ and $\mc P_1(\mu)$ these equations are equivalent to
\begin{align}
\label{eq:proof_sc_PT1}\mc P_1(\mu)(\pi_{\mf X}^{-1}(x))&=\ep{1}(\mu)((\pi_{\mf X}\times\pi_{\mf X})^{-1}(\vec{e}))+\ep{1}(\mu)((\pi_{\mf X}\times\pi_{\mf X})^{-1}(\eop{e}))\\
\label{eq:proof_sc_PT2}\mc P_1(\mu)(\pi_{\mf X}^{-1}(x))&=\sum_{\iota(\vec{e}')=y}\ep{1}(\mu)((\pi_{\mf X}\times\pi_{\mf X})^{-1}(\vec{e}\,'))
\end{align}
for all $\vec{e}\in{\mf E}$ and $x,y\in\mf X$. Note that
\begin{gather*}
{\mf E}=\bigsqcup_{x\in \mf X}\bigsqcup_{\substack{\vec{e}\in{\mf E}\\\iota(\vec{e})=x}}\{\vec{e}\}
\end{gather*}
so that Equation \eqref{eq:proof_sc_PT2} yields
\begin{gather*}
\abs{\mf X}\mc P_1(\mu)(\pi_{\mf X}^{-1}(x))=\sum_{y\in\mf X}\sum_{\iota(\vec{e}')=y}\ep{1}(\mu)((\pi_{\mf X}\times\pi_{\mf X})^{-1}(\vec{e}\,'))=\sum_{\vec{e}\in\vec{\mf E}}\ep{1}(\mu)((\pi_{\mf X}\times\pi_{\mf X})^{-1}(\vec{e})),
\end{gather*}
independently of $x\in\mf X$. On the other hand Equation \eqref{eq:proof_sc_PT1} gives
\begin{gather*}
\abs{\mf E}\mc P_1(\mu)(\pi_{\mf X}^{-1}(x))=\sum_{\vec{e}\in{\mf E}}\ep{1}(\mu)((\pi_{\mf X}\times\pi_{\mf X})^{-1}(\vec{e})).
\end{gather*}
If $\abs{\mf X}\neq\abs{\mf E}$ we thus obtain that $\mc P_1(\mu)$ is zero.

Let us now consider the case $\mf c(\mf G)=1$ so that $\mf G$ has exactly one cycle. By Equation~\eqref{eq:proof_sc_PT} we obtain that $\mc P_1(\mu)$ is constant and thus $\dim_\C(\mc P_1\vert_{\mc D'(\widetilde\Omega)^{\Gamma,1}})\leq\dim_\C(\mathrm{im}(\mc P_1))=1$. We choose one of the two \emph{directed} cycles given by the non-backtracking chain of directed edges $(\vec{e}_1,\ldots,\vec{e}_n)$ and define vertices $x_1,\ldots,x_{n+1}$ by $\tau(\vec{e}_i)=\iota(\vec{e}_{i+1})\eqqcolon x_{i+1}$ for $1\leq i\leq n-1$ and $x_{n+1}\coloneqq\tau(\vec{e}_n)\eqqcolon x_1$. For each $C\in\C$ we now construct a function $f$ in $\mc E_1(\el;\mathrm{Maps}({\mf E},\C))$ such that $\sum_{\iota(\vec{e}')=x}f(\vec{e}\,')=C$ for all $x\in\mf X$. For $x,y\in\mf X$ we define the set $P(x,y)$ of all non-backtracking paths $(y_1,\ldots,y_m)$ of vertices $y_k$ such that $y_1=x,\ y_2=y$ and such that there exist $i$ and $j$ such that $y_i=x_j$ and $y_{i+1}=x_{j+1}$. Let $c_1\in\C$ be arbitrary. Then we define
\begin{gather*}
f((x,y))\coloneqq
\begin{cases}
        c_1&\colon(x,y)=\vec{e}_i\text{ for some }i,\\
        C-c_1&\colon(y,x)=\vec{e}_i\text{ for some }i,\\
        C&\colon\{x,y\}\nsubseteq\{x_1,\ldots,x_n\}\text{ and }P(x,y)\neq\emptyset,\\
        0&\colon\{x,y\}\nsubseteq\{x_1,\ldots,x_n\}\text{ and }P(x,y)=\emptyset.
\end{cases}
\end{gather*}
As for each edge $(x,y)\in{\mf E}$ such that $(y,z)\in{\mf E}$ implies $z=x$ we must have $f((x,y))=0$ by the defining property of $\mc E_1(\el;\mathrm{Maps}({\mf E},\C))$ we actually infer that every function in $\mc E_1(\el;\mathrm{Maps}({\mf E},\C))$ has to have this form. The dimension of $\mc E_1(\el;\mathrm{Maps}({\mf E},\C))$ is thus two in this case and, since $f$ satisfies $\sum_{\iota(\vec{e}')=x}f(\vec{e}\,')=C$ for each $x\in\mf X$, we obtain that $\mc P_1(\mu)\in\mathrm{im}(\mc P_1\vert_{\mc D'(\widetilde\Omega)^{\Gamma,1}})$ for
\begin{gather*}
\mu\coloneqq(\ep{1})^{-1}((\pi_{\mathfrak{X}}\times\pi_{\mathfrak{X}})_*(f))\in\mc D'(\widetilde\Omega)^{\Gamma,1}
\end{gather*}
with, for each $\widetilde x\in\widetilde X$,
\begin{multline*}
\mc P_1(\mu)(\widetilde x)=\sum_{\vec{e}=(\widetilde x,\widetilde y)\in{\widetilde{\mf E}}}\ep{1}(\mu)(\vec{e})=\sum_{\vec{e}=(\widetilde x,\widetilde y)\in{\widetilde{\mf E}}}(\pi_{\mf X}\times\pi_{\mf X})_*(f)(\vec{e})\\
=\sum_{\vec{e}=(\widetilde x,\widetilde y)\in{\widetilde{\mf E}}}f(\pi_{\mf X}(\widetilde x),\pi_{\mf X}(\widetilde y))=C.
\end{multline*}
In particular, $\mathrm{im}(\mc P_1\vert_{\mc D'(\widetilde\Omega)^{\Gamma,1}})\neq\{0\}$ and therefore $\dim_\C(\mathrm{im}(\mc P_1\vert_{\mc D'(\widetilde\Omega)^{\Gamma,1}}))=1$.
\end{proof}

\begin{remark}
        Let $n\in\N$ and $\mf G$ be $n$-regular. Then we have $\abs{\mf E}=\frac{n\abs{\mf X}}{2}$ and thus
\begin{gather*}
\mf c(\mf G)=\frac{n\abs{\mf X}}{2}-\abs{\mf X}+1=\frac{n-2}{2}\abs{\mf X}+1.
\end{gather*}
In particular, if $\mf G$ arises as a quotient of a Bruhat--Tits tree we obtain $\mf c(\mf G)>1$ since $n-1$ is a prime power.
\end{remark}

As a direct corollary we obtain:
\begin{corollary}\label{cor:gen_grad_van_1}
    Let $\mathfrak{c}(\mathfrak{G})\neq1$. Then $\ep{1}$ provides an isomorphism
\begin{align*}
    \ep{1}\colon \mc D'(\widetilde\Omega)^{\Gamma,1}&\cong \{f\in\mc E_1(\el;\mathrm{Maps}(\widetilde{\mf E},\C))^\Gamma\mid\forall x\in\widetilde{\mathfrak{X}}\colon\sum_{\substack{\vec{e}\in\widetilde{\mathfrak{E}}\\\iota(\vec{e})=x}}f(\vec{e})=0\}\\
&\cong\{f\in\mc E_1(\el;\mathrm{Maps}(\mf E,\C))\mid\forall x\in\mathfrak{X}\colon\sum_{\substack{\vec{e}\in\mathfrak{E}\\\iota(\vec{e})=x}}f(\vec{e})=0\}.
\end{align*}
\end{corollary}

\begin{proof}
By Proposition~\ref{thm:PT_iso} we know that $\ep{1}$ maps $\mc D'(\widetilde\Omega)$ bijectively onto $\mc E_1(\el;\mathrm{Maps}({\mf E},\C))$. Moreover, it is $G$-equivariant by
Lemma \ref{la:PT_intertwining}. Since $\mathcal{P}_1(x)=\sum_{\vec{e}\in\mathfrak{E},\,\iota(\vec{e})=x}\ep{1}(\vec{e})$ for each $x\in\mathfrak{X}$ we obtain the desired isomorphism from Theorem~\ref{thm:im_P1_Gamma}.
\end{proof}

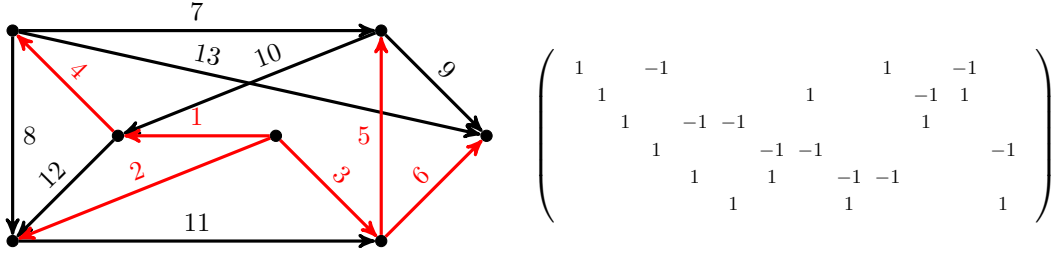
\begin{figure}
\tikzstyle{circle}=[shape=circle,draw,inner sep=1.5pt]
\begin{subfigure}[b]{0.5\linewidth}
\begin{tikzpicture}[scale=0.7,>=stealth',baseline]
\draw (0,0) node[circle,fill=black] (1) {};
\draw (7,0) node[circle,fill=black] (2) {};
\draw (2,-2) node[circle,fill=black] (3) {};
\draw (5,-2) node[circle,fill=black] (4) {};
\draw (9,-2) node[circle,fill=black] (5) {};
\draw (0,-4) node[circle,fill=black] (6) {};
\draw (7,-4) node[circle,fill=black] (7) {};
\draw[->,very thick] (1) -- (2) node[midway,sloped,above] {7};
\draw[->,very thick] (1) -- (6) node[midway,sloped,right,rotate=90] {8};
\draw[->,very thick] (2) -- (5) node[midway,sloped,above] {9};
\draw[->,very thick] (2) -- (3) node[pos=0.4,sloped,above] {10};
\draw[->,very thick] (6) -- (7) node[midway,sloped,above] {11};
\draw[->,very thick] (3) -- (6) node[midway,sloped,above] {12};
\draw[->,very thick] (1) -- (5) node[pos=0.4,sloped,above] {13};
\draw[->,very thick,red] (4) -- (3) node[midway,sloped,above] {1};
\draw[->,very thick,red] (4) -- (6) node[midway,sloped,above] {2};
\draw[->,very thick,red] (4) -- (7) node[midway,sloped,above] {3};
\draw[->,very thick,red] (3) -- (1) node[midway,sloped,above] {4};
\draw[->,very thick,red] (7) -- (2) node[midway,sloped,left,rotate=-90] {5};
\draw[->,very thick,red] (7) -- (5) node[midway,sloped,above] {6};
\end{tikzpicture}
\end{subfigure}\hfill
\begin{subfigure}[b]{0.5 \linewidth}
\begin{tikzpicture}[scale=0.7,baseline]
\matrix (m) at (0,-2) [matrix of math nodes,nodes in empty cells,right delimiter={)},
left delimiter={(},ampersand replacement=\&,nodes={scale=0.75}]{
    1\& \& \& -1\& \& \& \& \& \& 1\& \& -1\&\\
     \& 1\& \& \& \& \& \& 1\& \& \& -1\& 1\&\\
     \& \& 1\& \& -1\& -1\& \& \& \& \& 1\& \&\\
     \& \& \& 1\& \& \& -1\& -1\& \& \& \& \&-1\\
     \& \& \& \& 1\& \& 1\& \& -1\& -1\& \& \&\\
     \& \& \& \& \& 1\& \& \& 1\& \& \& \&1\\};
\end{tikzpicture}
\end{subfigure}
\caption{Example of a graph $\mf G$ (black) with spanning tree $\mf T$ (red) and corresponding matrix $A$ ($\mf c(\mf G)=7$)}
\label{fig:finite_graph}
\end{figure}

\begin{theorem}\label{thm:dim_z1}
We have
\begin{gather*}
\dim_\C(\mc D'(\widetilde\Omega)^{\Gamma,1})=\dim_\C(\mc E_1(\el;\mathrm{Maps}({\mf E},\C)))=
\begin{cases}
       \mf c(\mf G)&\colon\mf c(\mf G)\neq1,\\
       2&\colon\mf c(\mf G)=1.
\end{cases}
\end{gather*}
\end{theorem}

\begin{proof}
If $\mf c(\mf G)=1$ the result follows as in the proof of Theorem~\ref{thm:im_P1_Gamma}. Let us now consider the case of $\mf c(\mf G)\neq 1$. As $\mc D'(\widetilde\Omega)^{\Gamma,z}\cong\mc E_z(\el;\mathrm{Maps}({\mf E},\C))$ by Proposition \ref{prop:vv_iso_local}, it suffices to determine the dimension of $\mc E_1(\el;\mathrm{Maps}({\mf E},\C))$. For this we take some $f\in\mc E_1(\el;\mathrm{Maps}({\mf E},\C))$ and set $\mu\coloneqq(\ep{1})^{-1}((\pi_{\mf X}\times\pi_{\mf X})_*(f))$. Note that $\mc P_1(\mu)=0$ by Theorem~\ref{thm:im_P1_Gamma} so that
\begin{gather}\label{eq:proof_sum0}
\forall x\in \mf X\colon\quad\sum_{\iota(\vec{e}')=x}f(\vec{e}\,')=0
\end{gather}
and, since $f\in L({\mf E})$ by Remark~\ref{rem:L_spaces},
\begin{gather}\label{eq:proof_edge0}
\forall\vec{e}\in{\mf E}\colon\quad0=f(\vec{e})+f(\eop{e}).
\end{gather}
Note that \eqref{eq:proof_sum0} and \eqref{eq:proof_edge0} are equivalent to \eqref{eq:proof_edge0} and
\begin{gather}\label{eq:proof_sum1}
\forall x\in \mf X\colon\quad \el f(\vec{e}_x)=f(\vec{e}_x),
\end{gather}
where we can choose for each $x\in \mf X$ \emph{one} fixed edge $\vec{e}_x\in{\mf E}$ with $\tau(\vec{e}_x)=x$.
We now choose these edges in a systematic way. Let $\mf T$ be a spanning tree of $\mf G$ and choose a base point $o\in\mf X$.
We number the edges of $\mf T$ pointing away from $o$ as follows. First number the edges with initial point $o$, then the edges with terminal points of distance two from $o$ and so on (in the $j$th step we number all edges with terminal points of distance $j$ from $o$). For the edges that are not included in $\mf T$, we choose a direction and number the directed edges with the remaining numbers from the set $\{1,\ldots,\abs{\mf E}\}$.
 For each $e\in\mf E$ we have now fixed a direction and we denote the resulting
directed edges by $\vec{e}_1,\ldots,\vec{e}_{\abs{\mf E}}\in{\mf E}$. By
Equation~\eqref{eq:proof_edge0} the values $f(\vec{e}_j)$ already determine
$f$. For each point $x\in\mf X\setminus\{o\}$ we have exactly one edge
$\vec{e}_x\in\{\vec{e}_1,\ldots,\vec{e}_{\abs{\mf E}}\}$ which is contained in
$\mf T$ and has $\tau(\vec{e}_x)=x$. For each of these points we encode, using
Equation \eqref{eq:proof_edge0}, the condition from Equation
\eqref{eq:proof_sum1} in a matrix $A$ (see Figure~\ref{fig:finite_graph}). We
claim that $A$ is in row echelon form. In fact, the variable $f(\vec{e}_j)$ cannot appear in the subsequent equations for $\vec{e}_k$ with $k>j$ since each
edge $\vec{e}_k$ is pointing away from $\vec{e}_j$ and $\mf T$ does not contain
any circles. Thus, we have $\abs{\mf E}$ variables
$f(\vec{e}_1),\ldots,f(\vec{e}_{\abs{\mf E}})$ and $\abs{\mf X}-1$ independent
conditions and therefore $\abs{\mf E}-(\abs{\mf X}-1)=\mf c(\mf G)$ degrees of
freedom. We finally prove that the given conditions imply Equation
\eqref{eq:proof_sum1} also for $x=o$. Adding $f(\eop{e}_o)$ on both sides we first
note that this is equivalent to Equation \eqref{eq:proof_sum0} for $x=o$ by
Equation \eqref{eq:proof_edge0}. Successively applying Equation
\eqref{eq:proof_sum1} for each edge in $\mf T$ we obtain
\begin{gather*}
\sum_{\iota(\vec{e}\,')=o}f(\vec{e}\,')=\sum_{j>\abs{\mf X}-1}f(\vec{e}_j)+f(\eop{e}_j)=0.\qedhere
\end{gather*}
\end{proof}

Note that $c(\mf G)\not=0$ since $\mf G$ does by assumption, not have dead ends. Thus, Theorem~\ref{thm:dim_z1} together with Theorem~\ref{thm:im_P1_Gamma}, implies that $\mc P_1$ is neither surjective nor injective. In particular, $z=1$ is an exceptional parameter.

\subsection{Exceptional parameter \texorpdfstring{$-1$}{-1} and existence of \texorpdfstring{$2$}{2}-colorings}

We investigate the exceptional parameter $z=-1$ in a similar fashion as the parameter $z=1$. The dimensions in this case depend not only on the circuit rank, but also on whether the graph is \emph{$2$-colorable} (i.e.\@ bipartite) or not.

\begin{theorem}\label{thm:im_P-1_Gamma}
If $\mf c(\mf G)\neq1$ or $\mf G$ is not $2$-colorable, then
\begin{gather*}
\dim_\C(\mathrm{im}(\mc P_{-1}\vert_{\mc D'(\widetilde\Omega)^{\Gamma,-1}}))=0.
\end{gather*}
If $\mf c(\mf G)=1$ and $\mf G$ is $2$-colorable, then $\dim_\C(\mathrm{im}(\mc P_{-1}\vert_{\mc D'(\widetilde\Omega)^{\Gamma,-1}}))=1$.
\end{theorem}

\begin{proof}
Let $\mu\in\mc D'(\widetilde\Omega)^{\Gamma,-1}$. Then $\varphi\coloneqq \mc P_{-1}(\mu)\circ\pi_{\mf X}^{-1}$ and $f\coloneqq(\pi_{\mf X}\times\pi_{\mf X})_*^{-1}(\ep{-1}(\mu))\in\mc E_{-1}(\el;\mathrm{Maps}({\mf E},\C))$ are well-defined and
\begin{gather*}
\forall x\in\mf X\colon\quad\sum_{\iota(\vec{e}')=x}f(\vec{e}\,')=\varphi(x).
\end{gather*}
Since $f\in\mc E_{-1}(\el;\mathrm{Maps}({\mf E},\C))$ we therefore obtain
\begin{gather}\label{eq:proof_-1_eq1}
\forall x\in\mf X\ \forall\vec{e}\in{\mf E},\ \tau(\vec{e})=x\colon\quad f(\eop{e})-f(\vec{e})=\varphi(x).
\end{gather}
Moreover, for each edge $\vec{e}=(x,y)\in{\mf E}$,
\begin{gather}\label{eq:proof_-1_eq2}
\varphi(y)=f(\eop{e})-f(\vec{e})=-(f(\vec{e})-f(\eop{e}))=-\varphi(x).
\end{gather}
We first assume that $\mf G$ is not $2$-colorable. Fix a vertex $o\in\mf X$. Then, for each $x\in\mf X$, we have $\varphi(x)=\pm\varphi(o)$ by Equation \eqref{eq:proof_-1_eq2}. Since $\mf G$ is not $2$-colorable there exist $x,y\in\mf X$ with $\{x,y\}\in\mf E$ such that $\varphi(x)=\varphi(y)$ and thus $\varphi(x)=0$. Since $\mf G$ is connected we have $\varphi\equiv 0$.

Let us now consider the case when $\mf G$ is $2$-colorable, i.e.\@ bipartite. Then we can write $\mf X=U_1\sqcup U_2$ such that each edge in $\mf E$ contains exactly one vertex from $U_1$ and $U_2$. Let
\begin{gather*}
{\mf E}_{i}\coloneqq\{\vec{e}\in{\mf E}\colon\tau(\vec{e})\in U_i\},\qquad i\in\{1,2\},
\end{gather*}
such that ${\mf E}={\mf E}_{1}\sqcup{\mf E}_{2}$. Note that ${\mf E}_{1}\cong{\mf E}_{2}$ as sets via $\vec{e}\mapsto\eop{e}$. By Equation \eqref{eq:proof_-1_eq2}, $\varphi$ is constant on $U_i$. We denote the constant on $U_1$ by $c$ such that $\varphi$ is constantly $-c$ on $U_2$. Then Equation \eqref{eq:proof_-1_eq1} implies
\begin{gather*}
\forall\vec{e}\in{\mf E}_{1}\colon\quad f(\eop{e})-f(\vec{e})=c\qquad\text{ and }\qquad\forall\vec{e}\in{\mf E}_{2}\colon\quad f(\eop{e})-f(\vec{e})=-c.
\end{gather*}
We calculate
\begin{align*}
2\abs{\mf E}c&=(\abs{{\mf E}_{1}}+\abs{{\mf E}_{2}})c=\sum_{\vec{e}\in{\mf E}_{1}}f(\eop{e})-f(\vec{e})+\sum_{\vec{e}\in{\mf E}_{2}}f(\vec{e})-f(\eop{e})\\
&=\sum_{\vec{e}\in{\mf E}_{1}}f(\eop{e})-f(\vec{e})+\sum_{\vec{e}\in{\mf E}_{1}}f(\eop{e})-f(\vec{e})\\
&=2\sum_{\vec{e}\in{\mf E}_{1}}f(\eop{e})-2\sum_{\vec{e}\in{\mf E}_{1}}f(\vec{e})\\
&=2\sum_{x\in U_1}\sum_{\iota(\vec{e})=x}f(\vec{e})-2\sum_{y\in U_2}\sum_{\iota(\vec{e})=y}f(\vec{e})\\
&=2\abs{U}c-2\abs{V}(-c)=2\abs{\mf X}c.
\end{align*}
In particular, $c=0$ if $\abs{\mf E}\neq\abs{\mf X}$. In the case $\abs{\mf E}=\abs{\mf X}$, using a similar construction as in the proof of Theorem~\ref{thm:im_P1_Gamma}, we obtain dimension one.
\end{proof}

\begin{corollary}
    Let $\mathfrak{c}(\mathfrak{G})\neq1$ or $\mathfrak{G}$ be not bipartite. Then $\ep{-1}$ provides an isomorphism
\begin{align*}
    \ep{-1}\colon \mc D'(\widetilde\Omega)^{\Gamma,-1}&\cong \{f\in\mc E_{-1}(\el;\mathrm{Maps}(\widetilde{\mf E},\C))^\Gamma\mid\forall x\in\widetilde{\mathfrak{X}}\colon\sum_{\substack{\vec{e}\in\widetilde{\mathfrak{E}}\\\iota(\vec{e})=x}}f(\vec{e})=0\}\\
&\cong\{f\in\mc E_{-1}(\el;\mathrm{Maps}(\mf E,\C))\mid\forall x\in\mathfrak{X}\colon\sum_{\substack{\vec{e}\in\mathfrak{E}\\\iota(\vec{e})=x}}f(\vec{e})=0\}.
\end{align*}
\end{corollary}

\begin{proof}
Analogous to the proof of Corollary \ref{cor:gen_grad_van_1}.
\end{proof}

\begin{theorem}\label{thm:dim_z-1}
If $\mf c(\mf G)=1$ we have
\begin{gather*}
\dim_\C(\mc D'(\widetilde\Omega)^{\Gamma,-1})=\dim_\C(\mc E_{-1}(\el;\mathrm{Maps}(\vec{\mf E},\C)))=
\begin{cases}
       2&\colon \mf G\text{ is $2$-colorable},\\
       0&\colon \mf G\text{ is not $2$-colorable}.
\end{cases}
\end{gather*}
For $\mf c(\mf G)\neq 1$ we obtain
\begin{gather*}
\dim_\C(\mc D'(\widetilde\Omega)^{\Gamma,-1})=\dim_\C(\mc E_{-1}(\el;\mathrm{Maps}({\mf E},\C)))=
\begin{cases}
       \mf c(\mf G)&\colon\mf G\text{ is $2$-colorable},\\
       \mf c(\mf G)-1&\colon\mf G\text{ is not $2$-colorable}.
\end{cases}
\end{gather*}
\end{theorem}

\begin{proof}
We first prove $\mc E_{-1}(\el;\mathrm{Maps}({\mf E},\C))=\{0\}$ if $\mf c(\mf G)=0$, i.e.\@ if $\mf G$ is a finite tree. Let $f\in\mc E_{-1}(\el;\mathrm{Maps}({\mf E},\C))$. Then the defining property of $\mc E_{-1}(\el;\mathrm{Maps}({\mf E},\C))$ implies that $f(\vec{e})=0$ for each edge $\vec{e}=(x,y)\in{\mf E}$ pointing to a vertex $y$ of degree one. Recursively applying the eigenvalue equation, we obtain that $f$ vanishes identically. Note that trees are $2$-colorable.

If $\mf c(\mf G)=1$ and $\mf G$ is not $2$-colorable, we infer as above that, for $f\in\mc E_{-1}(\el;\mathrm{Maps}({\mf E},\C))$, we must have $f(\vec{e})=0$ for each edge $\vec{e}$ pointing to a vertex of degree one. Since $\mf G$ is not $2$-colorable, there is a cycle that contains an odd number of vertices. Now we infer that on the directed edges of the cycle the sign has to alternate. Since the number of vertices on the cycle is odd, we see that $f$ has to vanish identically. If $\mf c(\mf G)=1$ and $\mf G$ is $2$-colorable we can construct every solution similar to the construction in the proof of Theorem~\ref{thm:im_P1_Gamma} to obtain dimension $2$.

We now consider the case $\mf c(\mf G)\neq 1$. Let $f\in\mc E_{-1}(\el;\mathrm{Maps}({\mf E},\C))$ and $\mu\coloneqq(\ep{-1})^{-1}((\pi_{\mf X}\times\pi_{\mf X})_*(f))$. By Theorem~\ref{thm:im_P-1_Gamma} we have $\mc P_{-1}(\mu)=0$ and thus
\begin{gather}\label{eq:proof_-1_sum0}
\forall x\in\mf X\colon\quad\sum_{\iota(\vec{e}')=x}f(\vec{e}\,')=0.
\end{gather}
Using the eigenvalue equation we hence infer
\begin{gather}\label{eq:proof_-1_edge0}
\forall\vec{e}\in{\mf E}\colon\quad f(\vec{e})-f(\eop{e})=0.
\end{gather}
Conversely, note that \eqref{eq:proof_-1_sum0} and \eqref{eq:proof_-1_edge0} imply $f\in\mc E_{-1}(\el;\mathrm{Maps}({\mf E},\C))$. On the other hand, \eqref{eq:proof_-1_sum0} and \eqref{eq:proof_-1_edge0} are equivalent to \eqref{eq:proof_-1_edge0} and
\begin{gather}\label{eq:proof_-1_sum1}
\forall x\in \mf X\colon\quad \el f(\vec{e}_x)=-f(\vec{e}_x),
\end{gather}
where we can choose for each $x\in \mf X$ \emph{one} fixed edge $\vec{e}_x\in{\mf E}$ with $\tau(\vec{e}_x)=x$.
As in the proof of Theorem~\ref{thm:dim_z1} we now choose a spanning tree $\mf T$ of $\mf G$, a base point $o\in\mf X$, a choice of directed edges and a (specific) numbering of these such that the first $\abs{\mf X}-1$ edges are contained in $\mf T$. For the edges $\vec{e}_1,\ldots,\vec{e}_{\abs{\mf X}-1}$ of $\mf T$ we again encode the condition from Equation \eqref{eq:proof_-1_sum1} in a matrix $A$, which is seen to be in row echelon form. At this point we have $\abs{\mf E}$ variables and $\abs{\mf X}-1$ independent conditions and thus $\mf c(\mf G)$ degrees of freedom left. However, we have to check whether the given conditions imply Equation \eqref{eq:proof_-1_sum1} also for $x=o$. We claim that this is the case if and only if $\mf G$ is $2$-colorable. Note first that Equation \eqref{eq:proof_-1_sum1} for $x=o$ is equivalent to Equation \eqref{eq:proof_-1_sum0} for $x=o$ by Equation \eqref{eq:proof_-1_edge0} for $\vec{e}_o$. Successively applying Equation \eqref{eq:proof_-1_sum1} we can write
\begin{gather}\label{eq:proof_-1_o}
\sum_{\iota(\vec{e}')=o}f(\vec{e}\,')=\sum_{j>\abs{\mf X}-1}(-1)^{d_{\mf T}(\iota(\vec{e}_j),o)} f(\vec{e}_j)+(-1)^{d_{\mf T}(\tau(\vec{e}_j),o)} f(\eop{e}_j),
\end{gather}
where $d_{\mf T}$ denotes the distance between the vertices $x,y\in\mf X$ with respect to $\mf T$. Note that the parity of $d_{\mf T}(x,o)$ defines a coloring on the vertices and each coloring must arise in this way. If $\mf G$ is $2$-colorable, $d_{\mf T}(\iota(\vec{e}),o)$ and $d_{\mf T}(\tau(\vec{e}),o)$ have different parity for each edge $\vec{e}\in{\mf E}$. Therefore, the right hand side of \eqref{eq:proof_-1_o} vanishes by Equation \eqref{eq:proof_-1_edge0}. In this case the condition at $o$ thus follows from the other conditions. If $\mf G$ is not $2$-colorable, there exists an edge $\vec{e}_j$ such that $d_{\mf T}(\iota(\vec{e}_j),o)$ and $d_{\mf T}(\tau(\vec{e}_j),o)$ have the same parity. Therefore, it does not follow from \eqref{eq:proof_-1_edge0} that the sum in \eqref{eq:proof_-1_o} vanishes. The condition at $o$ thus gives another independent condition and we have $\mf c(\mf G)-1$ degrees of freedom in this case.
\end{proof}

Theorem~\ref{thm:dim_z-1} and  Theorem~\ref{thm:im_P-1_Gamma} can be used to prove that $z=-1$ is an exceptional parameter, unless $c(\mf G)=1$ and $\mf G$ is not $2$-colorable.

\section{Quantum-classical correspondence for finite graphs}\label{sec:QCcorrespondence}

As in Section~\ref{sec:appl_topology} we consider in this section a finite graph $\mf G=(\mf X,\mf E)$ of bounded degree and its \emph{universal cover} $\widetilde{\mf G}=(\widetilde{\mf X},\widetilde{\mf E})$. We remind the reader that our graphs do not have dead ends. We apply results of the previous sections to prove isomorphisms between eigenspaces of Laplacians and eigenspaces of transfer (or Koopman) operators for finite graphs. Motivated by the analogous Archimedean results for compact locally symmetric spaces where the Laplacian can be considered as the quantization of the geodesic flow, see \cite{DFG,GHWb,AH21}, we call such isomorphisms \emph{quantum-classical correspondences}. For regular spectral parameters such a quantum-classical correspondence for finite graphs was established in \cite[Thm.~8.3]{BHW23}. We rederive this result using the scalar Poisson transformation as was done in \cite[Thm.~11.5]{BHW23}, but in a slightly different way which has an analog for the edge Poisson transform, see Theorem~\ref{thm:reg qcc}. This analog also works for the exceptional spectral parameters, resulting in Theorem~\ref{thm:excep qcc}.

We recall from \cite[\S~4]{BHW23} the following scale of finite dimensional subspaces of $\mathrm{Maps}(\mf P,\C)$: for $N\in \N$ let $\mathcal F_{N,\mathbb A}$ denote the space of functions on $\mf P$ that only depend on the first $N$ symbols of $\vec{\mathbf e}=(\vec e_1,\vec e_2,\ldots)$.

\begin{remark}\label{rem:LonF1}
\begin{itemize}
\item[(i)]
 The action of the transfer operator $\mathcal L$ can be rewritten as
\[
 (\mathcal LF)(\vec e_1,\vec e_2,\ldots) = \sum_{\vec e_0: \mathbb A_{\vec e_0, \vec e_1} = 1} F(\vec e_0, \vec e_1, \vec e_2,\ldots).
\]
Thus, if for $N>1$ and $F\in \mathcal F_{N,\mathbb A}$ the expression $F(\vec e_0, \vec e_1, \vec e_2,\ldots)$ only depends on the first $N$ symbols $\vec e_0,\vec e_1,\ldots, \vec e_{N-1}$, then $\mathcal LF(\vec e_1,\vec e_2,\ldots)$
only depends on the first $N-1$ symbols. Summarizing, we have for $N>1$, $\mathcal L: \mc F_{N,\mathbb A}\to \mc F_{{N-1},\mathbb A}$. In particular $\mathcal F_{1,\mathbb A}$ is invariant under $\mathcal L$. In order to describe the action of $\mathcal L$ on the space $\mathcal F_{1,\mathbb A}$, we note that we can interpret any element in $\mathcal F_{1,\mathbb A}$ simply as a map $f:{\mathfrak E}\to \C$ and we calculate
\[
 \mathcal (\mathcal Lf) (s) = \sum_{t: \mathbb A_{t,s} = 1} f(t)
\]
In other words, if we identify $\mc F_{1,\mathbb A}\cong \C^{|{\mathfrak E}|}\cong \mathrm{Maps}({\mathfrak E},\C)$, then $\mathcal L$ acts on $\mathcal F_{1,\mathbb A}$ simply by $\mathbb A^\top$, the transpose of the transition matrix.

\item[(ii)] Rephrasing the last statement more formally we can say 
 that
\[\mathrm{pr}_0^*:\mathcal E_z(\mathbb A^\top;\mathrm{Maps}({\mathfrak E},\C))\to
\mathcal E_z(\mathcal L;\mathrm{Maps}(\mf P,\C)), \quad f\mapsto f\circ \mathrm{pr}_0\]
with $\mathrm{pr}_0(\vec{e}_1,\vec{e}_2,\ldots)=\vec{e}_1$ is well-defined, linear and injective with image $\mathcal E_z(\mathcal L;\mathcal F_{1,\mathbb A})$.

\item[(iii)] Recall from Definition~\ref{def:edge_Laplacian} that
\[
 \mathcal (\Delta^\mathrm{e}f) (s) = \sum_{t: \mathbb A_{s,t} = 1} f(t).
\]
Thus, on $\mathcal F_{1,\mathbb A}$ the transfer operator $\mathcal L$ agrees with  $\Delta^\mathrm{e}_\mathrm{op}$, the edge Laplacian of the \emph{opposite directed graph} $(\mathfrak X,{\mathfrak{E}}_\mathrm{op})$, i.e.\@ the same graph with the maps $\iota$ and $\tau$ interchanged.
\end{itemize}
\end{remark}

\begin{proposition}\label{prop:F_1-loc-const}
Let $\mathcal F_{\infty,\mathbb A}:= \bigcup_{k\in\N} \mathcal F_{k,\mathbb A}$. Then
\begin{itemize}
\item[(i)] $\mathcal F_{\infty,\mathbb A}=C^{\mathrm lc}(\mf P)$.

\item[(ii)] For $0\not= z\in \C$ we have $\mathcal E_z(\mathcal L; C^\mathrm{lc}(\mf P))\subset \mathcal F_{1,\mathbb A}$.
\end{itemize}

\end{proposition}

\begin{proof}
For any $f\in \mathcal F_{\infty,\mathbb A}$ there exists a $k\in \N$ such that $\mathcal L^k f\in \mathcal F_{1,\mathbb A}$. For $f\in \mathcal E_z(\mathcal L;\mathcal F_{\infty,\mathbb A})$ we have $f= \frac{1}{z^{k}}\mathcal L^{k} f$, so $\mathcal E_z(\mathcal L; \mathcal F_{\infty,\mathbb A})=\mathcal E_z(\mathcal L;\mathcal F_{1,\mathbb A})$.

We note that for a locally constant function $f$ on $\widetilde{\mf P}\cong \widetilde{\mathfrak X}\times \widetilde \Omega$ for each $x\in \widetilde{\mathfrak X}$ we have a finite set  $E_x\subseteq {\widetilde{\mathfrak E}}$ such that $\widetilde\Omega$ is the disjoint union of the $\partial_+(\vec e)$ with $\vec e\in E_x$ and $f(x,\bigcdot)$ constant on each of these $\partial_+(\vec e)$.

If $f$ is $\Gamma$-invariant we can consider the union $E$ of all $E_x$ with $x$ running through a (finite!) set of $\Gamma$-representatives in $\widetilde{\mathfrak X}$, refine the coverings of $\widetilde \Omega$ and thus assume that $f(x,\bigcdot)$ is constant for arbitrary $x\in \widetilde{\mathfrak X}$ on each  $\partial_+(\vec e)$ with $\vec e\in E$. Viewing $f$ as a function on $\widetilde{\mf P}$ this means that $f$ does not depend on the parts of the chains beyond the $\vec e\in E$. Since any point in $\widetilde{\mathfrak X}$ has finite distance to $\iota(E)$ and any chain has to pass through one of the $\vec e\in E$, we see that $f\in \mathcal F_{\infty,\mathbb A}$.
\end{proof}

\begin{remark}
For $z^2\notin\{0,1\}$ we have
\[\mc D'(\widetilde\Omega)^{\Gamma,z}
\overset{\ref{rem:scalar Poisson}}{\cong}
\mathcal E_{\chi(z)}(\Delta;\mathrm{Maps}(\mathfrak X,\C))
\overset{\text{\cite[Prop.~3.1]{BHW23}}}{\cong}
\mathcal E_z(\mathbb A^\top; \mathrm{Maps}({\mathfrak E},\C))
\overset{\ref{rem:LonF1}}{\cong}
\mathcal E_z(\mathcal L;\mathcal F_{1,\mathbb A}).\]
\label{rem:Mfa-LF1}
\end{remark}

Combining Proposition~\ref{prop:F_1-loc-const} and Remark~\ref{rem:Mfa-LF1}  with \cite[Cor.~9.4]{BHW23}, we obtain the following quantum-classical correspondence for the regular spectral parameters $z\in \C\setminus\{-1,0,1\}$.

\begin{theorem}[Regular quantum-classical correspondence for finite graphs]\label{thm:reg qcc}
 Let $\mathfrak G=(\mathfrak X,\mathfrak E)$ be a finite graph of bounded degree without dead ends and $z\in \C\setminus \{-1,0,1\}$, then
 \[
  \mathcal E_{\chi(z)}(\Delta;\mathrm{Maps}(\mathfrak X,\C))
  \overset{\ref{rem:Mfa-LF1}}{\cong}
  \mathcal E_z(\mathcal L;\mathcal F_{1,\mathbb A})
  \overset{\ref{prop:F_1-loc-const}}{=}
  \mathcal E_z(\mathcal L; C^\mathrm{lc}(\mf P))
  \overset{\text{\cite[Cor.~9.4]{BHW23}}}{=}
  \mathcal E_z(\mathcal L'; \mc D'(\mf P))
   \]
 are isomorphic vector spaces.
 \label{thm:qcc_local'}
\end{theorem}

To obtain a quantum-classical correspondence also for the exceptional parameters we follow the same strategy but use the edge Poisson transform instead of the vertex Poisson transform. This means, we combine Lemma~\ref{la:intertwine_cpt_pic} and Lemma~\ref{la:upstairs_downstairs_distr} with \cite[Cor.~9.4]{BHW23} and Proposition~\ref{prop:vv_iso_local} to obtain

\begin{theorem}[Quantum-classical correspondence for finite graphs]\label{thm:excep qcc}
 Let $\mathfrak G=(\mathfrak X,\mathfrak E)$ be a finite graph of bounded degree without dead ends and $z\in \C\setminus \{0\}$, then
 \begin{eqnarray*}
 \mc E_z(\el;\mathrm{Maps}({\mf E},\C)) &\overset{\ref{prop:vv_iso_local}}{\cong}&
 \mc D'(\widetilde\Omega)^{\Gamma,z}\\
 &\overset{\ref{la:intertwine_cpt_pic}}{\cong}&
  \mathcal E_z(\widetilde{\mathcal L}'; \mc D'(\widetilde{\mf P})^\Gamma)\\
  &\overset{\ref{la:upstairs_downstairs_distr}}{=}&
  \mathcal E_z(\mathcal L'; \mc D'(\mf P))\\
  &\overset{\text{\cite[Cor.~9.4]{BHW23}}}{=}&
  \mathcal E_z(\mathcal L; C^\mathrm{lc}(\mf P))
 \end{eqnarray*}
 are isomorphic vector spaces.
\end{theorem}

\section{The case of regular trees and connections to group theory}\label{sec:HomTreeGrpTheory}

In this section we consider a $(q+1)$-regular tree $\mathfrak{G}=(\mathfrak{X},\mathfrak{E})$ and its automorphism group $G\coloneqq\mathrm{Aut}(\mathfrak{G})$. We choose a reference geodesic $(\ldots, x_{-1},x_0,x_{1},\ldots)$, write $o$ for $x_0$ and $\omega_{\pm}$ for the two boundary points of the geodesic in the positive resp.\@ negative time limit. Moreover, we let $K$ be the stabilizer of $o$ which is a maximal compact subgroup of $G$ and choose an element $\tau\in G$ that acts by $\tau x_j=x_{j+1}$ on the reference geodesic and write $A\coloneqq\langle\tau\rangle$.

\begin{remark}
	In the case where $q$ is a prime power, $\mf G$ can be viewed as the Bruhat--Tits tree of the group $\operatorname{PGL}(2,\K)$ with $\K$ a non-archimedean local field whose residue field has cardinality $q$. All results in this section still hold true if we replace $G=\mathrm{Aut}(\mathfrak{G})$ by $\operatorname{PGL}(2,\K)$ and $K$ by $\operatorname{PGL}(2,\mc O)$ with $\mc O$ the ring of integers in $\K$. In this setting, $A$ is the subgroup of diagonal matrices whose diagonal entries are powers of a fixed uniformizer of $\mc O$.
\end{remark}

\subsection{Representations of \texorpdfstring{$K$}{K}}\label{sec:rep of K}

The following proposition generalizes \cite[Thm.\@~3.3]{S70} to the case of regular trees of arbitrary degree. We assume that it is well-known to experts, but could not find a reference in the literature, so we provide a full proof.

\begin{proposition}\label{prop:rep_of_K}
Consider the inner product
\begin{gather}\label{eq:la_decomp_fa}
    C^{\mathrm{lc}}(\Omega)\times C^{\mathrm{lc}}(\Omega)\to\mathbb{C},\quad (\phi,\psi) \mapsto \int_{K}\phi(k\omega_{+})\overline{\psi(k\omega_{+})}\intd k,
\end{gather}
where $\intd k$ denotes the normalized Haar measure on $K$, and let
\begin{gather*}
    W_0\coloneqq\{\varphi\in C^{\mathrm{lc}}(\Omega)\text{ constant}\}
\end{gather*}
and, for $i>0$,
\begin{gather*}
    W_i\coloneqq\{\varphi\in C^{\mathrm{lc}}(\Omega)\mid \forall\vec{e}\in\mathfrak{E}_o\colon d(o,\tau(\vec{e}))=i\Rightarrow \varphi|_{\partial_{+}\vec{e}}\text{ constant}\}.
\end{gather*}
Then, as a $K$-representation, $C^{\mathrm{lc}}(\Omega)$ decomposes as the algebraic direct sum
\begin{gather*}
    C^{\mathrm{lc}}(\Omega)=\bigoplus_{i\in\mathbb{N}_0}V_i,
\end{gather*}
where $V_i$ is given by $W_0$ for $i=0$ and as the orthogonal complement of $W_{i-1}$ in $W_i$ for $i>0$. Each $V_i$ is irreducible, unitary and self-dual. For $q>1$ the dimension of $V_i$ is $q^{i}-q^{i-2}$ for $i>1,\ \dim V_1=q$ and $\dim V_0=1$. For $q=1$ we have $\dim V_0=\dim V_1=1$ and $\dim V_i=0$ for $i\geq 2$.
\end{proposition}

\begin{proof}
Let $\varphi\in C^{\mathrm{lc}}(\Omega)$. Since $\varphi$ is locally constant and $\Omega$ is compact, there exists some $n\in\mathbb{N}_0$ such that $\varphi\in W_n=\bigoplus_{i=0}^{n}V_i$. Moreover, since the map $K\to \Omega,\ k \mapsto k\omega_{+}$ is $K$-equivariant and the Haar measure on $K$ is $K$-invariant, the $K$-representation on $C^{\mathrm{lc}}(\Omega)$ and $W_i$ is unitary with respect to the inner product \eqref{eq:la_decomp_fa}. In particular, each $V_i$ is $K$-invariant. We will now prove that each $V_i$ is irreducible. For the proof it is more convenient to interpret $V_i$ as the quotient $W_i/W_{i-1}$. Since $V_0=\mathbb{C}$ we may consider $i\neq 0$. In this case let $0\neq f\in V_i$ and $g\in W_{i}$ be a representative of $f$, i.e.\@ $g\in f+W_{i-1}$. We denote the edges $\vec{e}\in\mathfrak{E}_{o}$ with $d(o,\tau(\vec{e}))=i-1$ by $\vec{e}^{\, 1},\ldots,\vec{e}^{\, n}$ and the edges $\vec{e}\in\mathfrak{E}_{o}$ with $d(o,\tau(\vec{e}))=i$ and $\iota(\vec{e})=\tau(\vec{e}^{\, j})$ by $\vec{e}^{\, j}_1,\ldots,\vec{e}^{\, j}_m$ (for $i=1$ the first ones do not exist and we let $\vec{e}^{\, 1}_{1},\ldots,\vec{e}^{\, 1}_{q+1}$ be the edges starting in $o$ in this case). Since $g\not\in W_{i-1}$ there exists some edge $\vec{e}^{\, j}$ such that $g$ is not constant on $\partial_{+}\vec{e}^{\, j}$ so that there exist edges, which we may assume to be $\vec{e}^{\, j}_{1}$ and $\vec{e}^{\, j}_{2}$, such that $g$ takes different values on $\partial_{+}\vec{e}^{\, j}_{1}$ and $\partial_{+}\vec{e}^{\, j}_{2}$. By adding an appropriate function $f'$ in $W_{i-1}$ to $g$ and rescaling we may assume that $g$ is zero on $\partial_{+}\vec{e}^{\, j}_{1}$ and one on $\partial_{+}\vec{e}^{\, j}_{2}$. Now let $k\in K$ be such that $k$ interchanges $\partial_{+}\vec{e}^{\, j}_{1}$ and $\partial_{+}\vec{e}^{\, j}_{2}$ and fixes every other edge $\vec{e}^{\, i}_{\ell}$. Then, by subtracting $kg$ from $g$, we have an element $h$ that is $-1$ on $\partial_{+}\vec{e}^{\, j}_{1}$, $1$ on $\partial_{+}\vec{e}^{\, j}_{2}$ and zero elsewhere. We denote this by the vector $(-1,1,0,\ldots,0)$. Adding the function $\mathbbm{1}_{\partial_{+}\vec{e}^{\, j}}\in W_{i-1}$ we also get $(0,2,1,\ldots,1)$. But now we see that, by combining the latter function with permutations of the first one, we obtain $(0,\ell,0,\ldots,0)$ for some $\ell\in\mathbb{N}$ and thus $(0,1,0,\ldots,0)$ which clearly generates $V_i$.\\
For the self-duality note that the integral over $K$ defines a non-degenerate $K$-invariant bilinear form on $V_i$ (similar to the inner product in \eqref{eq:la_decomp_fa}, but without the complex conjugation), which identifies $V_i$ with its dual space in a $K$-equivariant way. Finally, the dimensions of the $V_i$ can be computed by counting the number of edges at a fixed distance so that, for $i>1$, $\dim V_i=(q+1)q^{i-1}-(q+1)q^{i-2}$, and $\dim V_1=(q+1)-1$.
\end{proof}

\begin{remark}\label{rem_W1}
Note that $W_0$ is the subspace of constant functions. Moreover, $W_1=V_0\oplus V_1$ is the subspace of functions that are constant on all $\partial_{+}\vec{e}$ for $\vec{e}$ a directed edge starting in $o$. Note that $K$ acts transitively on these directed edges, so if we denote the stabilizer of $(o,x_1)$ in $K$ by $K_1$, it follows that $W_1$ is isomorphic to the space of functions on $K/K_1$.\\
Let us identify the subgroup $K_1$ in the case of a Bruhat--Tits tree. If $\K$ is a non-archimedean local field with ring of integers $\mc O$, the quotient $\PGL{2,\K}/\PGL{2,\mc O}$ is a $(q+1)$-regular tree, where $q$ is the cardinality of the residue field of $\mc O$. It can further be identified with the equivalence classes of $\mc O$-lattices in $\K^2$ with respect to the equivalence relation given by homothety. $\PGL{2,\K}$ acts transitively on the space of all equivalence classes of lattices and $\PGL{2,\mc O}$ is the stabilizer of the base point $o=[\mc Oe_1+\mc Oe_2]$. If $\mf p\subseteq\mc O$ denotes the unique maximal ideal in $\mc O$, then $x_1=[\mc Oe_1+\mf pe_2]$ is a neighbor of $o$ in the tree. With this choice, the stabilizer $K_1$ of $x_1$ in $K$ is given by the Iwahori subgroup consisting of matrices which are upper triangular modulo $\mathfrak{p}$.
\end{remark}

\subsection{The operator Hecke algebra}
In this subsection we describe a Hecke algebra that acts on the space $C(G\times_KV)$ of sections of a vector bundle associated to a finite dimensional (not necessarily irreducible) $K$-representation $(\pi, V)$. We first consider the space
\begin{gather*}
    C_c(G,\on{End}(V))^{\on{cj}_K}\coloneqq \{\Phi\in C_c(G,\on{End}(V))\mid \forall g\in G,\, k\in K\colon \Phi(kgk^{-1})=\pi(k)\circ \Phi(g)\circ \pi(k)^{-1}\}.
\end{gather*}
A short calculation shows that the defining property of $C_c(G,\on{End}(V))^{\on{cj}_K}$ is preserved under convolutions
\begin{gather*}
    (\Phi_1\ast\Phi_2)(g)\coloneqq\int_G\Phi_2(h)\circ\Phi_1(gh^{-1})\intd h,
\end{gather*}
which turn the space into an algebra. Moreover, it acts on sections $F\in C(G\times_K V)$ by right convolution
\begin{gather*}
    (F\ast \Phi)(g) = \int_G\Phi(h)F(gh^{-1})\intd h
\end{gather*}
since, for each $k\in K$,
\begin{align*}
    (F\ast \Phi)(gk) &= \int_G\Phi(h)F(gkh^{-1})\intd h = \int_G\Phi(k^{-1}khk)F(gh^{-1})\intd h\\
&= \int_G\Phi(k^{-1}hk)F(gh^{-1}k)\intd h = \int_G(\pi(k^{-1})\circ\Phi(h)) F(gh^{-1})\intd h\\
&= \pi(k^{-1})(F\ast \Phi)(g)
\end{align*}
and $(F\ast\Phi_1)\ast\Phi_2=F\ast(\Phi_1\ast\Phi_2)$ follows from the unimodularity of $G$.

\begin{lemma}\label{la:kernel_op_conv}
The kernel of the action of $C_c(G,\on{End}(V))^{\on{cj}_K}$ on $C(G\times_K V)$ is given by
\begin{gather*}
    U\coloneqq\{\Phi\in C_c(G,\on{End}(V))^{\on{cj}_K}\mid \int_K\pi(k)\circ\Phi(gk)\intd k = 0\}.
\end{gather*}
\end{lemma}

\begin{proof}
If $\Phi\in U$ we have
\begin{align*}
(F\ast\Phi)(g)&=\int_K\pi(k)(F\ast\Phi)(gk)\intd k=\int_K\int_G\pi(k)\Phi(h)F(gkh^{-1})\intd h\intd k\\
&=\int_G\int_K\pi(k)\Phi(hk)\intd k F(gh^{-1})\intd h=0.
\end{align*}
On the other hand, if $\Phi$ is in the kernel of the action and $v\in V$ we consider the function $F(g)\coloneqq \mathbbm{1}_K(g)\pi(g^{-1})v$. Then $F\in C(G\times_KV)$ and
\begin{gather*}
    0=\int_G\Phi(hg)F(h^{-1})\intd h=\int_K\Phi(kg)\pi(k)v\intd k=\left(\int_K\pi(k)\circ\Phi(gk)\intd k\right)v.\qedhere
\end{gather*}
\end{proof}

Lemma \ref{la:kernel_op_conv} motivates the following definition.

\begin{definition}\label{def:OperatorValuedHeckeAlgebra}
    We call the quotient
\begin{gather*}
\mc{H}(G,K;V)\coloneqq C_c(G,\on{End}(V))^{\on{cj}_K}/U
\end{gather*}
 the \emph{operator valued Hecke algebra}. By Lemma \ref{la:kernel_op_conv}, $\mc{H}(G,K;V)$ acts faithfully on $C(G\times_KV)$.
\end{definition}

\begin{lemma}
The map
\begin{gather}\label{eq:HeckeAlgIso}
    \Upsilon\colon C_c(G,\on{End}(V))^{\on{cj}_K}\to \{\Phi\in C_c(G,\on{End}(V))\mid \Phi(k_{1}gk_{2}) = \pi(k_2^{-1})\circ \Phi(g)\circ \pi(k_1^{-1})\}
\end{gather}
which maps $\Phi$ to
\begin{gather*}
    \Upsilon(\Phi)(g)\coloneqq\int_K\pi(k)\Phi(gk)\intd k
\end{gather*}
is surjective and induces an isomorphism between $\mc{H}(G,K;V)$ and the right hand side of \eqref{eq:HeckeAlgIso}.
\end{lemma}

\begin{proof}
Note first that
\begin{align*}
    \Upsilon(\Phi)(k_1gk_2) &= \int_K\pi(k)\Phi(k_1gk_2k)\intd k = \left(\int_K\pi(kk_1)\circ \Phi(gk_2kk_1)\intd k\right)\circ \pi(k_1^{-1})\\
&= \left(\int_K\pi(k_2^{-1}k)\circ \Phi(gk)\intd k\right)\circ \pi(k_1^{-1}) = \pi(k_2^{-1})\circ \Upsilon(\Phi)(g)\circ\pi(k_1^{-1})
\end{align*}
so that $\Upsilon$ is defined. Moreover, if $\Phi$ is in the codomain of $\Upsilon$,
\begin{gather*}
    \Upsilon(\Phi)(g)=\int_K\pi(k)\Phi(gk)\intd k=\Phi(g),
\end{gather*}
so $\Upsilon$ is surjective.
\end{proof}

\subsection{Connections to the principal series}
Recall the representation $\pi_z$, $z\neq0$, on $\mathcal{D}'(\Omega)$ from Equation \eqref{eq:pi_z_def} and write $H_z\coloneqq(\pi_z,\mathcal{D}'(\Omega))$. In this section we present an analog of the vector valued Poisson transform of \cite{Ol} for the case of regular trees.

 We first consider the space $\Hom{K}{H_z}{V}$.
The following result allows us to weakly approximate the finitely additive measures $\mu\in\mathcal{D}'(\Omega)$ by locally constant functions multiplied with the Haar measure $\intd k$ on $K$.

\begin{lemma}\label{la:meas_approx}
For every $\mu\in\mathcal{D}'(\Omega)$ there exists a sequence $\{\chi_n\}_n\subseteq C^{\mathrm{lc}}(\Omega)$ such that
$$ \int_Kf(k\omega_0)\chi_n(k\omega_0)\intd k \to \int_\Omega f\intd\mu \qquad \mbox{for all }f\in C^{\mathrm{lc}}(\mathcal{O}),\omega_0\in\Omega. $$
\end{lemma}

\begin{proof}
	For every $x\in\widetilde{\mathfrak X}$ let
	$$ \Omega_{o,x} := \{\omega\in\Omega\mid x\in[o,\omega)\}. $$
	Then $\Omega=\bigsqcup_{x:d(o,x)=n}\Omega_{o,x}$ for every $n\in\N$. Moreover, for fixed $\omega\in\Omega$, the sets $\Omega_{o,x}$, $x\in[o,\omega[$, are open and closed and form a neighborhood basis of $\omega$. Write $\nu$ for the measure on $\Omega$ induced by the Haar measure $\intd k$ on $K$, then $\nu(\Omega_{o,x})>0$ for all $x$. For a given $\mu\in\mathcal{D}'(\Omega)$ we put
	$$ \chi_n \coloneqq \sum_{x:d(o,x)=n}\frac{\mu(\Omega_{o,x})}{\nu(\Omega_{o,x})}\mathbbm{1}_{\Omega_{o,x}} \qquad (n\in\N). $$
	Then $\chi_n\in C^{\mathrm{lc}}(\Omega)$ since the sets $\Omega_{o,x}$ are open and closed. Moreover, if $f\in C^{\mathrm{lc}}(\Omega)$ and $N$ is large enough such that $f$ is constant on all $\Omega_{o,x}$ for $d(o,x)\geq N$, then for every $n\geq N$ and every choice of $\omega_{o,x}\in\Omega_{o,x}$ ($d(o,x)=n$):
\begin{gather*}
	\int_\Omega f\chi_{n}\intd\nu = \sum_{x:d(o,x)=n}f(\omega_{o,x})\int_{\Omega_{o,x}}\chi_{n}\intd\nu = \sum_{x:d(o,x)=n}f(\omega_{o,x})\mu(\Omega_{o,x}) = \int_\Omega f\intd\mu(x).\qedhere
\end{gather*}
\end{proof}

\begin{lemma}
The image of the map
\begin{gather*}
    \Phi\colon \Hom{K}{V}{C^{\mathrm{lc}}(\Omega)}\to \Hom{K}{H_z}{V'},\quad \Phi(T)(\mu)(v)\coloneqq\mu(T(v))
\end{gather*}
is given by $\mathrm{Hom}_K^\mathrm{cont}(H_z,V')\coloneqq\{T\in \Hom{K}{H_z}{V'}\mid T(\mu_n)\to T(\mu)\text{ whenever }\mu_n\stackrel{w}{\to}\mu\}$.
\end{lemma}

\begin{proof}
If $\mu_n\stackrel{w}{\to}\mu$ we have
\begin{gather*}
    \Phi(T)(\mu_n)(v)=\mu_{n}(T(v))\to\mu(T(v))=\Phi(T)(\mu)(v)
\end{gather*}
for all $v\in V$. On the other hand, if $T\in\mathrm{Hom}_K^\mathrm{cont}(H_z,V')$, Lemma~\ref{la:meas_approx} implies that $T$ is completely determined by its restriction to $C^\mathrm{lc}(\Omega)$, which we embed into $\mathcal{D}'(\Omega)$ by $\iota(f)(\varphi)\coloneqq\int_Kf(k\omega_{+})\varphi(k\omega_{+})\intd k$. However, dualizing $T|_{C^{\mathrm{lc}}(\Omega)}\colon C^{\mathrm{lc}}(\Omega)\to V'$ then gives rise to a homomorphism that gets mapped to $T$ by $\Phi$.
\end{proof}

 We now construct a representation of $\mathcal{H}(G,K;V)$ on the space $\operatorname{Hom}_K^{\mathrm{cont}}(H_z,V)$. Let $dg$ denote a Haar measure on on $G$.

\begin{proposition}
    $C_c(G,\on{End}(V))^{\on{cj}_K}$ acts on $\operatorname{Hom}_K^{\mathrm{cont}}(H_z,V)$ by
\begin{gather*}
    \omega_z\colon C_c(G,\on{End}(V))^{\on{cj}_K}\to \on{End}(\operatorname{Hom}_K^{\mathrm{cont}}(H_z,V)),\quad \omega_z(\Phi)(T)\coloneqq\int_G\Phi(g)\circ T\circ\pi_z(g)\intd g
\end{gather*}
with $\omega_z(\Phi_1)\circ\omega_z(\Phi_2)=\omega_z(\Phi_2\ast\Phi_1)$. Moreover, $U\subseteq\on{ker}\omega_z$ so that the action factors through $\mathcal{H}(G,K;V)$.
\end{proposition}

\begin{proof}
We first show that $\omega_z(\Phi)(T)\in \operatorname{Hom}_K^{\mathrm{cont}}(H_z,V)$. For $k\in K$ and $\mu\in\mathcal{D}'(\Omega)$ we have
\begin{align*}
\omega_z(\Phi)(T)(\pi_z(k)\mu)&=\int_G\Phi(g)(T(\pi_z(gk)\mu))\intd g\\
&=\int_G\Phi(gk^{-1})(T(\pi_z(g)\mu))\intd g\\
&=\int_G(\pi(k)\circ \Phi(k^{-1}g)\circ \pi(k^{-1}))(T(\pi_z(g)\mu))\intd g\\
&=\pi(k)\int_G(\Phi(g)\circ \pi(k^{-1}))(T(\pi_z(kg)\mu))\intd g\\
&=\pi(k)\int_G\Phi(g)(T(\pi_z(g)\mu))\intd g\\
&=\pi(k)(\omega_z(\Phi)(T)(\mu)).
\end{align*}
Now let $\Phi_1,\,\Phi_2\in C_c(G,\on{End}(V))^{\on{cj}_K}$. Then
\begin{align*}
(\omega_z(\Phi_1)\circ \omega_z(\Phi_2))(T)&=\int_G\Phi_1(g)\circ \omega_z(\Phi_2)(T)\circ \pi_z(g)\intd g\\
&=\int_G\int_G\Phi_1(g)\circ\Phi_2(h)\circ T\circ \pi_z(hg)\intd h\intd g\\
&=\int_G\int_G\Phi_1(g)\circ\Phi_2(hg^{-1})\circ T\circ \pi_z(h)\intd h\intd g\\
&=\int_G\int_G\Phi_1(g)\circ\Phi_2(hg^{-1})\intd g\circ T\circ \pi_z(h)\intd h\\
&=\int_G(\Phi_2\ast\Phi_1)(h)\circ T\circ \pi_z(h)\intd h=\omega_z(\Phi_2\ast\Phi_1)(T).
\end{align*}
Finally, we obtain that for each $\Phi\in U$
\begin{align*}
\omega_z(\Phi)(T)\mu&=\int_G(\Phi(g)\circ T\circ\pi_z(g))\mu\intd g\\
&=\int_K\int_G(\Phi(kg)\circ T\circ \pi_z(kg))\mu\intd g\intd k\\
&=\int_K\int_G(\pi(k)\circ\Phi(gk)\circ\pi(k^{-1})\circ T\circ \pi_z(kg))\mu\intd g\intd k\\
&=\int_K\int_G(\pi(k)\circ\Phi(gk)\circ T\circ \pi_z(g))\mu\intd g\intd k\\
&=\int_G\int_K\pi(k)\circ\Phi(gk)\intd k\, T(\pi_z(g)\mu)\intd g=0.\qedhere
\end{align*}
\end{proof}

\begin{definition}\label{def:vv_PT}
For each $K$-representation $(\pi,V)$ and $z\neq0$ we define the \emph{(vector valued) Poisson transform} by
\begin{gather*}
\mc P_z^\pi\colon\operatorname{Hom}_K^{\mathrm{cont}}(H_z,V)\otimes H_z\to C(G\times_KV),\quad \mc P_z^{\pi}(T\otimes \mu)(g)\coloneqq T(\pi_z(g^{-1})\mu).
\end{gather*}
Moreover, let $E_z\subseteq C(G\times_KV)$ denote the image of the map
\begin{gather*}
\Hom{\mathcal{H}(G,K;V)}{\operatorname{Hom}_K^{\mathrm{cont}}(H_z,V)}{C(G\times_KV)}\otimes\operatorname{Hom}_K^{\mathrm{cont}}(H_z,V)\to C(G\times_KV)\\
T\otimes f\mapsto T(f),
\end{gather*}
where $\mathcal{H}(G,K;V)$ acts on $\operatorname{Hom}_K^{\mathrm{cont}}(H_z,V)$ by $\omega_z$ and on $C(G\times_KV)$ by right convolutions. Note that $E_z$ is a joint eigenspace for all $\Phi\in \mathcal{H}(G,K;V)$ if $\operatorname{Hom}_K^{\mathrm{cont}}(H_z,V)=\C f$ is one dimensional since then $T(f)*\Phi=T(\omega_z(\Phi)(f))=T(\lambda_\Phi f)=\lambda_\Phi T(f)$ for some $\lambda_\Phi\in\C$.
\end{definition}

\begin{proposition}\label{prop:PT_equivariance}
The Poisson transform $\mc P_z^\pi$ is $\mathcal{H}(G,K;V)\times G$-equivariant, i.e.\@ for each $\Phi\in \mathcal{H}(G,K;V),\; T\in\operatorname{Hom}_K^{\mathrm{cont}}(H_z,V),\; f\in H_z$ and $g\in G$ we have
\begin{gather*}
\mc P_z^\pi(\omega_z(\Phi)(T)\otimes \mu)=\mc P_z^\pi(T\otimes \mu)*\Phi\quad\text{ and }\quad \mc P_z^\pi(T\otimes\pi_z(g)\mu)=L_g(\mc P_z^\pi(T\otimes \mu)),
\end{gather*}
where $L$ denotes the left regular representation on $C(G\times_KV)$. Hence, $\mc P_z^\pi$ maps into $E_z$.
\end{proposition}

\begin{proof}
For each $x\in G$ we have
\begin{align*}
\mc P_z^\pi(\omega_z(\Phi)(T)\otimes \mu)(x)&=\omega_z(\Phi)(T)(\pi_z(x^{-1})\mu)\\
&=\int_G\Phi(g)(T(\pi_z(gx^{-1})\mu))\intd g\\
&=\int_G\Phi(g)(\mc P_z^\pi(T\otimes \mu)(xg^{-1}))\intd g\\
&=(\mc P_z^\pi(T\otimes\mu)\ast\Phi)(x).
\end{align*}
Furthermore, $\mc P_z^\pi(T\otimes\pi_z(g)\mu)(x)=T(\pi_z(x^{-1})\pi_z(g)\mu)=\mc P_z^\pi(T\otimes\mu)(g^{-1}x)$.
\end{proof}

We end this subsection by formulating the reducibility result for the (unramified) principal series $\pi_{z}$ (recall the $K$-representations from Proposition~\ref{prop:rep_of_K}):

\begin{lemma}[{\cite[Prop.\@~2.5.1]{Ch94}}]\label{la:exc_pts}
The representation $(\pi_z,\mathcal{D}'(\Omega))$ is irreducible except for $z\in \{\pm 1,\pm q\}$.
\begin{enumerate}
    \item\label{it:ex1} For $z=\pm q$, $\pi_z$ has length two, a one dimensional irreducible subrepresentation which only contains the $K$-type $V_0$ and an irreducible quotient containing the $K$-types $V_1\oplus V_2\oplus V_3\oplus\ldots$.
    \item\label{it:ex2} For $z=\pm 1$, $\pi_z$ has length two, a one dimensional irreducible quotient which only contains the $K$-type $V_0$ and an irreducible subrepresentation containing the $K$-types $V_1\oplus V_2\oplus V_3\oplus\ldots$.
\end{enumerate}
In each case the infinite dimensional irreducible representation is called a \emph{Steinberg representation} (\cite[Def.\@~2.4.1/2]{Ch94}). The one dimensional representation is either the trivial representation (for $z\in\{1,q\}$) or the non-trivial character $\chi$ of $G$, $\chi(g)=(-1)^{d(o,go)}$ (for $z\in\{-1,-q\}$).
\end{lemma}

Later (in Proposition~\ref{prop:comp_rep_PT_geom_PT}), we will see that the scalar Poisson transform arises as a special case of the general construction in Definition~\ref{def:vv_PT} when choosing the projection onto $V_0$ as homomorphism. Thus, Lemma~\ref{la:exc_pts} gives an abstract reason why the scalar Poisson transform fails to be injective at the points $z=\pm 1$; in these cases the whole Steinberg subrepresentation lies in the kernel of $\mc P_{\pm 1}$. For $z\neq\pm 1$, however, there is no invariant subspace that does not contain $V_0$.
On the other hand, the homomorphism that corresponds to the edge Poisson transform projects onto the direct sum $W_1=V_0\oplus V_1$. Since any non-trivial subrepresentation of $\pi_z$ intersects $W_1$, this explains why $\ep{z}$ is always injective.

\subsection{Poisson transforms related to edges}
Let $\vec{e}_1\coloneqq(o,x_1)$ and denote its stabilizer in $K$ by $K_1$. In this subsection we consider the left regular representation of $K$ on $V\coloneqq \on{Maps}(K/K_1,\mathbb{C})$. Note that the space $\iota^{-1}(o)$ of edges starting in $o$ is isomorphic to $K/K_1$, as $K$ acts transitively on $\iota^{-1}(o)$. This in particular implies that $V\cong W_1$, the space of functions on $\iota^{-1}(o)$. Let us now describe the Poisson transforms in this case. We first recall from Remark~\ref{rem_W1} that $V$ decomposes as the direct sum of the two irreducible subspaces
\begin{gather*}
    V_0=\{f\in V\mid f\text{ constant}\}\quad\text{and}\quad V_1=\{f\in V\mid \sum_{\iota(\vec{e})=o}f(\vec{e})=0\}.
\end{gather*}
Since the decomposition of $C^{\mathrm{lc}}(\Omega)$ is multiplicity-free, the space $\operatorname{Hom}_K^{\mathrm{cont}}(H_z,V)$ is therefore two-dimensional. We choose the basis given by
\begin{gather*}
    T_0\colon H_z\to V,\quad T_0\mu(kK_1)\coloneqq\mu(\mathbbm{1})\quad\text{and}\quad T_1\colon H_z\to V,\quad T_1\mu(kK_1)\coloneqq\mu(\chi_{kK_1}),
\end{gather*}
where $\chi_{gK_1}$, for $g\in G$, denotes the indicator function of $\partial_{+}(g\vec{e}_1)$ in $C^{\mathrm{lc}}(\Omega)$.

Consider the maps $\mc P_{z}^{\pi,T_j}\coloneqq \mc P_{z}^\pi(T_j\otimes\bigcdot)\colon H_z\to C(G\times_KV)$, which we again call Poisson transforms. In order to compare these maps to the scalar and edge Poisson transforms we first need to identify $C(G\times_KV)$ with the space of functions on the directed edges. Note that the latter space is naturally isomorphic to $C(G/K_1)$ via $\Phi_{\mathfrak{E}}\colon C(\mathfrak{E})\to C(G/K_1),\ \Phi_{\mathfrak{E}}f(gK_1)\coloneqq f(g\vec{e}_1)$. In the following, we will use this identification without explicitly stating it to avoid cluttery notation.

\begin{lemma}\label{la:edge_functions}
The map $\Theta f(gK_1)\coloneqq f(g)(eK_1)$ defines a $G$-equivariant linear isomorphism from $C(G\times_KV)$ to $C(G/K_1)$, the space of functions on all directed edges.
\end{lemma}

\begin{proof}
Since $f(gk_1)(eK_1)=f(g)(k_1K_1)=f(g)(eK_1)$ for each $k_1\in K_1$, $\Theta$ is well-defined. To show that $\Theta$ is an isomorphism, we observe that its inverse is given by
\begin{gather*}
   \Theta^{-1}f(g)(kK_1)\coloneqq f(gkK_1)\qquad (g\in G,\, k\in K).
\end{gather*}
Since for $g\in G$ and $k,k'\in K$
\begin{gather*}
\Theta^{-1}f(gk')(kK_1)=f(gk'kK_1)=\Theta^{-1}f(g)(k'kK_1)=[\pi(k')^{-1}\Theta^{-1}f(g)](kK_1),
\end{gather*}
we have $\Theta^{-1}f\in C(G\times_KV)$. Moreover, it is clear from the definitions that $\Theta$ and $\Theta^{-1}$ are $G$-equivariant and mutually inverse.
\end{proof}

\begin{remark}\label{rem:KAK_decomp}
 Note that every element $g$ of $G$ can be written as a product $g=k_1\tau^{j}k_2$ for some $j\in-\mathbb{N}_0$ and $k_i\in K$. Indeed, there is an element $k_1^{-1}\in K$ that maps $go$ to $x_j=\tau^{j}o$, where $j=-d(o,go)$. Hence, $k_2=\tau^{-j}k_1^{-1}g\in K$ and we obtain $g=k_1\tau^jk_2$. Taking inverses shows that we can alternatively choose $j\in\mathbb{N}_0$.
\end{remark}

\begin{lemma}\label{la:second_horocycle_id}
For each $g\in G$ and $\omega\in\Omega$ we have
\begin{gather*}
    \langle go,g\omega\rangle=-\langle g^{-1}o,\omega\rangle.
\end{gather*}
\end{lemma}

\begin{proof}
Let us write $g=k_1\tau^{n}k_2$ with $k_1,k_2\in K$ and $n\geq0$ as in Remark~\ref{rem:KAK_decomp}. Then
\begin{gather*}
    \langle go,g\omega\rangle=\langle k_1\tau^{n}k_2o,k_1\tau^{n}k_2\omega\rangle=\langle \tau^{n}o,\tau^{n}k_2\omega\rangle
\end{gather*}
and $-\langle g^{-1}o,\omega\rangle=-\langle k_2^{-1}\tau^{-n}o,\omega\rangle=-\langle \tau^{-n}o,k_2\omega\rangle$.
Denoting $\widetilde{\omega}\coloneqq k_2\omega$ it thus suffices to prove that $\langle x_n,\tau^{n}\widetilde{\omega}\rangle=-\langle\tau^{-n}o,\widetilde{\omega}\rangle$ for each $\widetilde{\omega}\in\Omega$. We first consider the case $\widetilde{\omega}\not\in\{\omega_\pm\}$. Let $x_j\coloneqq \mathrm{argmax}_{x_i\in[o,\widetilde{\omega}[\cap]\omega_{-},\omega_{+}[}d(o,x_i)$. Then
\begin{gather*}
 \langle \tau^{-n}o,\widetilde{\omega}\rangle=d(o,x_j)-d(x_{-n},x_j)=d(x_n,x_{n+j})-d(o,x_{n+j})=-\langle x_n,\tau^{n}\widetilde{\omega}\rangle.
\end{gather*}
For $\widetilde{\omega}=\omega_{\pm}$ we have the same equality for all large resp.\@ small enough $j$.
\end{proof}

\begin{proposition}\label{prop:comp_rep_PT_geom_PT}
      $\Theta\circ \mc P_{z}^{\pi,T_0}=\mc P_{z}\circ\iota$ and $\Theta\circ \mc P_{z}^{\pi,T_1}=\ep{z}$
\end{proposition}

\begin{proof}
For $\mu\in\mathcal{D}'(\Omega)$ we have
\begin{align*}
\Theta P_{z}^{\pi,T_0}\mu(g\vec{e}_1)&=P_{z}^{\pi,T_0}\mu(g)(eK_1)=T_0(\pi_z(g^{-1})\mu)(eK_1)=\pi_z(g^{-1})\mu(\mathbbm{1})\\
&=\mu(z^{-\langle g^{-1}o,g^{-1}\bigcdot\rangle})=\mu(z^{\langle go,\bigcdot\rangle})=\mc P_z\mu(\iota(g\vec{e}_1)),
\end{align*}
where we use Lemma \ref{la:second_horocycle_id} in the penultimate step. Similarly,
\begin{align*}
\Theta P_{z}^{\pi,T_1}\mu(g\vec{e}_1)&=P_{z}^{\pi,T_1}\mu(g)(eK_1)=T_1(\pi_z(g^{-1})\mu)(eK_1)=\pi_z(g^{-1})\mu(\chi_{eK_1})\\
&=\mu(\chi_{gK_1}z^{-\langle g^{-1}o,g^{-1}\bigcdot\rangle})=\int_{\partial_+(g\vec{e}_1)}z^{\langle go,\bigcdot\rangle}\intd\mu=\ep{z}\mu(g\vec{e}_1).\qedhere
\end{align*}
\end{proof}

\subsection{Structure and action of the Hecke algebra on functions on edges}
In this subsection we use the decomposition from Remark~\ref{rem:KAK_decomp} to investigate $\mc{H}(G,K;V)$ and its action on $C(G/K_1)\cong C(\mathfrak{E})$ in detail.
In order to obtain a basis of $\mathcal{H}(G,K;V)$, the decomposition suggests to look at maps $\Phi_{n,T}$ of the form
\begin{gather*}
    \Phi_{n,T}(g)\coloneqq
\begin{cases}
    \pi(k_2^{-1})\circ T\circ \pi(k_1^{-1}) &\colon g=k_1\tau^{n}k_2,\ k_i\in K\\
    0                                       &\colon \text{else}
\end{cases},
\end{gather*}
where $n\in-\mathbb{N}_0$ and $T\in \on{End}(V)$. However, this is only well-defined when $T$ has the property that $\pi(k_2)\circ T = T\circ\pi(k_1)$ whenever $k_1\tau^{n}=\tau^{n}k_2$. We denote the space of all such endomorphisms by $\on{End}(V)^{K\times K}_n$.

\begin{lemma}
For $k_1,k_2\in K$ the following holds:
\begin{gather*}
    k_1\tau^{n}=\tau^{n}k_2\Rightarrow k_1\in\on{Stab}_K(o,\ldots,x_n)\text{ and }k_2\in\on{Stab}_K(o,\ldots,x_{-n}).
\end{gather*}
Moreover, if $k_1\in\on{Stab}_K(o,\ldots,x_n)$ then $k_2\coloneqq \tau^{-n}k_1\tau^{n}$ is in $\on{Stab}_K(o,\ldots,x_{-n})$.
\end{lemma}

\begin{proof}
Use that $K=\on{Stab}_G(o)$.
\end{proof}

\begin{proposition}\label{prop:Hecke_perm}
 Consider the action of the symmetric group $S(\iota^{-1}(o))$ of $\iota^{-1}(o)$ on $V$ given by $\pi.\delta_{\vec{e}_j}\coloneqq \delta_{\pi(\vec{e}_j)}$.
Then, for $n<0$, an endomorphism $T\in \on{End}(V)$ is in the space $\on{End}(V)^{K\times K}_n$ if and only if for any two permutations $\pi_{+},\, \pi_{-}$ of $\iota^{-1}(o)$ such that $\pi_{\pm}$ fixes $(o,x_{\pm 1})$, we have
\begin{gather}\label{eq:T_perm}
    T\circ \pi_{-}=\pi_{+}\circ T.
\end{gather}
Moreover, $T$ lies in $\on{End}(V)^{K\times K}_0$ if and only if
\begin{gather*}
    T\circ\sigma = \sigma\circ T
\end{gather*}
for each $\sigma\in S(\iota^{-1}(o))$.
\end{proposition}

\begin{proof}
Let $\iota^{-1}(o)=\{\vec{e}_1,\ldots,\vec{e}_{\on{deg}(o)}\}$ and $n<0$. We first assume that $T$ fulfills \eqref{eq:T_perm} and take $k_1,\, k_2\in K$. Then $\pi(k_i)$ permutes the elements of the set $\{\delta_{\vec{e}_1},\ldots,\delta_{\vec{e}_{\on{deg}(o)}}\}$ so that its action on $V$ agrees with that of a permutation $\pi_i$, say. But then, for $k_1\in\on{Stab}_K(o,\ldots,x_n)$ and $k_2=\tau^{-n}k_1\tau^{n}$ we have
\begin{gather*}
    \pi(k_2)\circ T=\pi_2\circ T=T\circ \pi_1=T\circ \pi(k_1).
\end{gather*}
On the other hand, let $T\in \on{End}(V)^{K\times K}_n$ and $\pi_{\pm}\in S(\iota^{-1}(o))$ with $\pi_{\pm}(o,x_{\pm})=(o,x_{\pm})$. We claim that there exists a $k_1\in K$ such that $\pi(k_1)$ acts like $\pi_{-}$ and $k_2=\tau^{-n}k_1\tau^{n}$ acts like $\pi_{+}$ on $V\cong\on{Maps}(\iota^{-1}(o),\mathbb{C})=\on{span}_{\mathbb{C}}\{\delta_{\vec{e}_1},\ldots,\delta_{\vec{e}_{\on{deg}(o)}}\}$. In fact, any $k_1$ which permutes the edges $\iota^{-1}(o)\setminus\{(o,x_{-1})\}$ like $\pi_{-}$ and acts like $\tau^{n}\pi_{+}\tau^{-n}$ on $\iota^{-1}(x_{n})\setminus\{(x_n,x_{n+1})\}$ will do. Then we have
\begin{gather*}
   T\circ \pi_{-} = T\circ \pi(k_1) = \pi(k_2)\circ T = \pi_{+}\circ T.
\end{gather*}
For $n=0$ we have $k_1=k_2\in\on{Stab}_K(o)=K$ which acts on $V$ as a permutation. Moreover, we obtain every permutation by varying $k_1$ so that $\pi(k)\circ T = T\circ \pi(k)$ (for all $k\in K$) is equivalent to $\sigma\circ T = T\circ \sigma$ (for all $\sigma\in S(\iota^{-1}(o))$).
\end{proof}

We now describe a basis of $\on{End}(V)^{K\times K}_n$. Note first that the only proper non-trivial subsets $A\subseteq\iota^{-1}(o)$ which are invariant under all permutations fixing $(o,x_{-1})$ are given by $\{(o,x_{-1})\}$ and $\iota^{-1}(o)\setminus\{(o,x_{-1})\}$. Thus, for $n<0$, Proposition~\ref{prop:Hecke_perm} implies that for each $T\in\on{End}(V)^{K\times K}_n$ we must have that $T(\delta_{\bigcdot})$ is constant on these sets. Moreover, since the images of such $T$ have to be invariant under $\pi_{+}\in S(\iota^{-1}(o)\setminus\{\vec{e}_1\})$, a basis of $\on{End}(V)^{K\times K}_n$ is given by $T_1,\ldots, T_4$ with
\begin{alignat*}{2}
    T_1(\delta_{\vec{e}})&\coloneqq
\begin{cases}
    \delta_{\vec{e}_1}  &\colon \vec{e}\in\iota^{-1}(o)\setminus\{(o,x_{-1})\}\\
    0                   &\colon \vec{e}=(o,x_{-1})
\end{cases},\quad
    T_2(\delta_{\vec{e}})&&\coloneqq
\begin{cases}
    \mathbbm{1}   &\colon \vec{e}\in\iota^{-1}(o)\setminus\{(o,x_{-1})\}\\
    0   &\colon \vec{e}=(o,x_{-1})
\end{cases},\\
    T_3(\delta_{\vec{e}})&\coloneqq
\begin{cases}
    0                       &\colon\vec{e}\in\iota^{-1}(o)\setminus\{(o,x_{-1})\}\\
    \delta_{\vec{e}_1}      &\colon \vec{e}=(o,x_{-1})
\end{cases},\quad
    T_4(\delta_{\vec{e}})&&\coloneqq
\begin{cases}
    0                       &\colon\vec{e}\in\iota^{-1}(o)\setminus\{(o,x_{-1})\}\\
    \mathbbm{1}      &\colon \vec{e}=(o,x_{-1})
\end{cases}.
\end{alignat*}
Alternatively, we can express the values of the basis $T_j$ on an arbitrary element $f\in V$:
\begin{alignat*}{2}
    T_1(f)(\vec{e})&=
\begin{cases}
    \sum_{\vec{e}\,'\neq(o,x_{-1})}f(\vec{e}\,')  &\colon \vec{e}=\vec{e}_1\\
    0                   &\colon \vec{e}\neq\vec{e}_1
\end{cases},\qquad
    &&T_2(f)(\vec{e})=\sum_{\vec{e}\,'\neq(o,x_{-1})}f(\vec{e}\,')\\
    T_3(f)(\vec{e})&=
\begin{cases}
    f((o,x_{-1}))                       &\colon\vec{e}=\vec{e}_1\\
    0      &\colon \vec{e}\neq\vec{e}_1
\end{cases},\qquad
    &&T_4(f)(\vec{e})=f((o,x_{-1})).
\end{alignat*}

For $n=0$ the situation in Proposition~\ref{prop:Hecke_perm} is slightly different and only allows the maps
\begin{gather*}
    T^0_0(f)(\vec{e})\coloneqq f(\vec{e})\quad\text{and}\quad T^0_1(f)(\vec{e})=\sum_{\vec{e}\,'\in\iota^{-1}(o)}f(\vec{e}\,').
\end{gather*}

We now compute the actions of all $\Phi_{-n,T}$ for $n\in\mathbb{N}_0$ and $T$ as above. For this recall the operators $\Delta^{\mathrm{e}}$, $\Xi$ and $\Sigma$ from Definition~\ref{def:edge_Laplacian}. Further note that the restriction of a Haar measure on $G$ to the open subset $K\tau^nK$ is bi-$K$-invariant and hence (up to a constant) given by integrating over $K\times K$ with respect to its normalized Haar measure. For $n=0$ we get
\begin{align*}
(F\ast \Phi_{0,T^0_1})(g)(eK_1)&=\left(\int_G\Phi_{0,T^0_1}(h)F(gh^{-1})\intd h\right)(eK_1)\\
&=\left(\int_K\int_K\Phi_{0,T^0_1}(k_1k_2)F(gk_2^{-1}k_1^{-1})\intd k_1\intd k_2\right)(eK_1)\\
&=\left(\int_K\int_K\pi(k_2^{-1})\circ T^0_1\circ\pi(k_1^{-1})\circ\pi(k_1)\circ F(gk_2^{-1})\intd k_1\intd k_2\right)(eK_1)\\
&=\int_KT^0_1(F(gk_2^{-1}))(k_2K_1)\intd k_2\\
&=\int_K\sum_{\vec{e}\,'\in\iota^{-1}(o)}F(gk_2^{-1})(\vec{e}\,')\intd k_2\\
&=\int_K\sum_{\vec{e}\,'\in\iota^{-1}(o)}f(gk_2^{-1}\vec{e}\,')\intd k_2=\sum_{\vec{e}\,'\in\iota^{-1}(o)}f(g\vec{e}\,')=\Sigma f(g\vec{e}_1),
\end{align*}
where $F\in C(G\times_KV)$ and $f\coloneqq\Theta F\in C(G/K_1)$ denotes its corresponding function on the directed edges (see Lemma~\ref{la:edge_functions}). Thus, the function on $\mathfrak{E}$ corresponding to $F\ast \Phi_{0,T^0_1}$ is given by $\Sigma f$. For $T^0_0$ we similarly obtain the identity.

Let us now investigate the case of $n>0$. Then, using the same notation as above, we infer that
\begin{align*}
(F\ast \Phi_{-n,T_1})(g)(eK_1)&=\int_KT_1(F(gk_2^{-1}\tau^{n}))(k_2K_1)\intd k_2=\int_{K_1}\sum_{\vec{e}\,'\neq(o,x_{-1})}F(gk_2^{-1}\tau^{n})(\vec{e}\,')\intd k_2\\
&=\operatorname{const.}\cdot \sum_{\vec{e}\in\Omega_{g\vec{e}_1}^{n}}f(\vec{e})=\operatorname{const.}\cdot (\el)^{n}f(g\vec{e}_1).
\end{align*}
Moreover,
\begin{align*}
(F\ast \Phi_{-n,T_2})(g)(eK_1)&=\int_KT_2(F(gk_2^{-1}\tau^{n}))(k_2K_1)\intd k_2\\
&=\int_K\sum_{\vec{e}\,'\neq(o,x_{-1})}F(gk_2^{-1}\tau^{n})(\vec{e}\,')\intd k_2\\
&=\sum_{\substack{\vec{e}\in\mathfrak{E}_{\iota(g\vec{e}_1)}\\d(\iota(\vec{e}),\iota(g\vec{e}_1))=n}}f(\vec{e})=\operatorname{const.}\cdot \Sigma(\el)^nf(g\vec{e}_1),
\end{align*}
where $\mathfrak{E}_{x}=\{\vec{e}\in\mathfrak{E}\mid \vec{e}\text{ points away from }x\}$. Therefore, these convolutions correspond to averaging over edges in distance $n$ pointing away from the initial point of the edge. For $T_3$ we have
\begin{align*}
(F\ast \Phi_{-n,T_3})(g)(eK_1)&=\int_KT_3(F(gk_2^{-1}\tau^{n}))(k_2K_1)\intd k_2=\int_{K_1}F(gk_2^{-1}\tau^{n})((o,x_{-1}))\intd k_2\\
&=\operatorname{const.}\cdot  \sum_{\vec{e}\in\Omega_{g\vec{e}_1}^{n-1}}f(\eop{e})=\operatorname{const.}\cdot (\el)^{n-1}\Xi f(g\vec{e}_1).
\end{align*}
Finally,
\begin{align*}
(F\ast \Phi_{-n,T_4})(g)(eK_1)&=\int_KT_4(F(gk_2^{-1}\tau^{n}))(k_2K_1)\intd k_2=\int_KF(gk_2^{-1}\tau^{n})((o,x_{-1}))\intd k_2\\
&=\operatorname{const.}\cdot \sum_{\substack{\vec{e}\in\mathfrak{E}_{\iota(g\vec{e}_1)}\\d(\iota(\vec{e}),\iota(g\vec{e}_1))=n-1}}f(\eop{e})=\operatorname{const.}\cdot\Sigma(\el)^{n-1}\Xi.
\end{align*}
This shows that the following operators form a basis of $\mathcal{H}(G,K;V)$:
\begin{gather*}
    \Sigma^k(\el)^n\Xi^\ell\quad (k,\ell\in\{0,1\},\ n\in\mathbb{N}_0).
\end{gather*}
Now note that $\Sigma=\Xi\el+\mathrm{Id}$ since
\begin{gather*}
\Sigma f(\vec{e})=f(\vec{e})+\sum_{\substack{\vec{e}\neq\vec{e}\,'\in{\mf E}\\\iota(\vec{e}\,')=\tau(\eop{e})}}f(\vec{e}\,')=f(\vec{e})+\el(\eop{e})=f(\vec{e})+\Xi\el f(\vec{e}).
\end{gather*}
This implies that also the following operators form a basis of $\mathcal{H}(G,K;V)$:
\begin{gather}
	(\Delta^{\on{e}})^m \quad (m\geq0), \qquad\qquad \Xi(\Delta^{\on e})^m\Xi \quad (m\geq1),\notag\\
	(\Delta^{\on e})^m\Xi \quad (m\geq1) \qquad\qquad \mbox{and} \qquad\qquad \Xi(\Delta^{\on e})^m \quad (m\geq1).\label{eq:BasisOfOperatorHeckeAlgebra}
\end{gather}
The following statement describes the Hecke algebra in terms of generators and relations:

\begin{proposition}\label{prop:HeckeAlgebraGenRel}
    The operator Hecke algebra $\mathcal{H}(G,K;V)$ acting on $\mathrm{Maps}(\mathfrak{E},\mathbb{C})$ is isomorphic to the algebra generated by $\el$ and $\Xi$ subject to the relations
$$ \Xi^2 = \operatorname{Id} \qquad \mbox{and} \qquad \Delta^{\on{e}}\Xi\Delta^{\on{e}} = q\Xi+(q-1)\Delta^{\on{e}}. $$
\end{proposition}

\begin{proof}
	It remains to show the claimed relations, because using the relations one can reduce every word in $\el$ and $\Xi$ to a unique linear combination of words of the form \eqref{eq:BasisOfOperatorHeckeAlgebra} and these form a basis of $\mathcal{H}(G,K;V)$. The first relation is obvious, and for the second one we compute
\begin{align*}
\el\Sigma f(\vec{e})&=\sum_{\substack{\eop{e}\neq\vec{e}\,'\in{\mf E}\\\iota(\vec{e}\,')=\tau(\vec{e})}}\Sigma f(\vec{e}\,')=\sum_{\substack{\eop{e}\neq\vec{e}\,'\in{\mf E}\\\iota(\vec{e}\,')=\tau(\vec{e})}}\sum_{\iota(\vec{e}\,'')=\iota(\vec{e}\,')}f(\vec{e}\,'')=\sum_{\substack{\eop{e}\neq\vec{e}\,'\in{\mf E}\\\iota(\vec{e}\,')=\tau(\vec{e})}}\sum_{\iota(\vec{e}\,'')=\tau(\vec{e})}f(\vec{e}\,'')\\
&=q\sum_{\iota(\vec{e}\,'')=\tau(\vec{e})}f(\vec{e}\,'')=qf(\eop{e})+q\sum_{\substack{\eop{e}\neq\vec{e}\,''\in{\mf E}\\\iota(\vec{e}\,'')=\tau(\vec{e})}}f(\vec{e}\,'')=q(\Xi+\el)f(\vec{e}),
\end{align*}
which implies the relation since $\Sigma=\Xi\el+\mathrm{Id}$.
\end{proof}

\begin{remark}
    For more general trees of bounded degree we obtain
\begin{gather*}
    \el\Sigma f(\vec{e})=q_x(\Xi+\el)f(\vec{e})\quad\text{and}\quad\el\Xi\el f(\vec{e})=(q_{x}\Xi+(q_{x}-1)\el) f(\vec{e}),
\end{gather*}
where $q_x+1$ is the degree of the vertex $x=\tau(\vec{e})$. This might be of independent interest.
\end{remark}

\subsection{Action of the Hecke algebra on images of Poisson transforms}
In the previous subsection we computed the action of $\mathcal{H}(G,K;V)$ on $C(G\times_KV)$. In particular, we observed that for the generators $\el$ and $\Xi$ we have (up to a constant)
\begin{gather*}
    (\Theta^{-1}\circ \el\circ \Theta) F = F\ast\Phi_{-1,T_1}\quad\text{and}\quad(\Theta^{-1}\circ \Xi\circ \Theta) F = F\ast\Phi_{-1,T_3},
\end{gather*}
for each $F\in C(G\times_KV)$. Now suppose that $F$ is contained in the image of a Poisson transform, i.e.\@ $F=\mc P_z^{\pi,T_j}\mu$ for some $j\in\{0,1\}$ and $\mu\in\mathcal{D}'(\Omega)$. Then Proposition \ref{prop:PT_equivariance} implies
\begin{gather*}
   \mc P_z^{\pi,T_j}\mu\ast \Phi_{-1,T_1}=\mc P_z^{\pi}(T_j\otimes\mu)\ast\Phi_{-1,T_1}=\mc P_z^{\pi}(\omega_z(\Phi_{-1,T_1})(T_j)\otimes\mu)
\end{gather*}
and similar for $\Phi_{-1,T_3}$. But since $\{T_0,T_1\}$ is a basis of $\operatorname{Hom}_K^{\mathrm{cont}}(H_z,V)$, there exist $a_j(z)$ and $b_j(z)$ in $\mathbb{C}$ such that
\begin{gather*}
    \omega_z(\Phi_{-1,T_1})(T_j)=a_j(z)T_0+b_j(z)T_1.
\end{gather*}
Therefore,
\begin{gather*}
    (\Theta^{-1}\circ \el\circ \Theta) \mc P_z^{\pi,T_j}\mu=\mc P_z^{\pi}(a_j(z)T_0+b_j(z)T_1\otimes\mu)=a_j(z)\mc P_z^{\pi,T_0}\mu+b_j(z)\mc P_z^{\pi,T_1}\mu,
\end{gather*}
which, by Proposition \ref{prop:comp_rep_PT_geom_PT}, is equivalent to
\begin{gather*}
    \el(\mathcal{P}_z\mu\circ \iota)=a_0(z)(\mc P_z\mu\circ \iota) + b_0(z)\ep{z}\mu\quad\text{resp.}\quad\el\ep{z}\mu=a_1(z)(\mc P_z\mu\circ \iota) + b_1(z)\ep{z}\mu.
\end{gather*}
Note that the coefficients $a_j(z)$ and $b_j(z)$ are independent of $\mu$ and of the edge at which we are evaluating both sides. In particular, we have for $j=1$ (and similar for $j=0$)
\begin{align*}
    \el\ep{z}\delta_{\omega_{\pm}}(\vec{e}_1)=a_1(z)(\mc P_z\delta_{\omega_{\pm}}\circ \iota)(\vec{e}_1) + b_1(z)\ep{z}\delta_{\omega_{\pm}}(\vec{e}_1)=a_1(z)\mc P_z\delta_{\omega_{\pm}}(o) + b_1(z)\ep{z}\delta_{\omega_{\pm}}(\vec{e}_1),
\end{align*}
where $\delta_{\omega_{\pm}}\in\mathcal{D}'(\Omega)$ is given by $\delta_{\omega_{\pm}}(f)\coloneqq f(\omega_{\pm})$.

\begin{lemma}\label{la:actions_on_delta}
We have
\begin{enumerate}
    \item $\mc P_z\delta_{\omega_{\pm}}(o)=1$,
    \item $\ep{z}\delta_{\omega_{+}}(\vec{e}_1)=1,\ \ep{z}\delta_{\omega_{-}}(\vec{e}_1)=0$,
    \item $\el\ep{z}\delta_{\omega_{+}}(\vec{e}_1)=z,\ \el\ep{z}\delta_{\omega_{-}}(\vec{e}_1)=0$,\label{it:elep}
    \item $\Xi\ep{z}\delta_{\omega_{+}}(\vec{e}_1)=0,\ \Xi\ep{z}\delta_{\omega_{-}}(\vec{e}_1)=z^{-1}$,
    \item $\el(\mathcal{P}_z\delta_{\omega_{\pm}}\circ \iota)(\vec{e}_1)=qz^{\pm 1}$,\label{it:el_iota}
    \item $\Xi(\mathcal{P}_z\delta_{\omega_{\pm}}\circ\iota)(\vec{e}_1)=z^{\pm 1}$.\label{it:Xi_iota}
\end{enumerate}
\end{lemma}

\begin{proof}
Note first that $\mathcal{P}_{z}\delta_{\omega_{\pm}}(o)=z^{\langle o,\omega_{\pm}\rangle}=1$. Moreover, for $\omega_0\in\Omega$,
\begin{gather*}
\ep{z}\delta_{\omega_0}(\vec{e}_1)=\int_{\partial_{+}\vec{e}_1}z^{\langle o,\omega\rangle}\intd\delta_{\omega_0}(\omega)=
\begin{cases}
   z^{\langle o,\omega_{0}\rangle}&\colon \omega_0\in\partial_{+}\vec{e}_1\\
   0&\colon \omega_0\not\in\partial_{+}\vec{e}_1
\end{cases},
\end{gather*}
proving the second part. For \ref{it:elep} note that
\begin{gather*}
    \el\ep{z}\delta_{\omega_{0}}(\vec{e}_1)=\sum_{\substack{\vec{e}_1^{\,\mathrm{op}}\neq\vec{e}\in{\mf E}\\\iota(\vec{e})=\tau(\vec{e}_1)}}\int_{\partial_{+}\vec{e}}z^{\langle\tau(\vec{e}_1),\omega\rangle}\intd\delta_{\omega_0}(\omega)=
\begin{cases}
    z^{\langle x_1,\omega_{+}\rangle}=z&\colon \omega_0=\omega_{+}\\
    0&\colon \omega_0=\omega_{-}
\end{cases}.
\end{gather*}
Furthermore,
\begin{gather*}
    \Xi\ep{z}\delta_{\omega_{0}}(\vec{e}_1)=\ep{z}\delta_{\omega_{0}}(\vec{e}_1^{\,\mathrm{op}})=\int_{\partial_{+}\vec{e}_1^{\,\mathrm{op}}}z^{\langle x_1, \omega\rangle}\intd\delta_{\omega_0}(\omega)=
\begin{cases}
    0&\colon \omega_0=\omega_{+}\\
    z^{\langle x_1,\omega_{-}\rangle}=z^{-1}&\colon \omega_0=\omega_{-}
\end{cases}.
\end{gather*}
For \ref{it:el_iota} note that $\el=q\Xi$ on $\Theta(C(G\times_KV_0))\ni \mathcal{P}_z\delta_{\omega_{\pm}}\circ\iota$, the space of functions on $\mathfrak{E}$ that only depend on the initial point of the edge, so that it suffices to prove \ref{it:Xi_iota}. But
\begin{gather*}
    \Xi(\mathcal{P}_z\delta_{\omega_{\pm}}\circ\iota)(\vec{e}_1)=\mathcal{P}_z\delta_{\omega_{\pm}}(x_1)=z^{\langle x_1,\omega_{\pm}\rangle}=z^{\pm 1}.\qedhere
\end{gather*}
\end{proof}

We can now use Lemma \ref{la:actions_on_delta} to determine the actions of $\el$ and $\Xi$ on the images of $\ep{z}$ and $\mathcal{P}_z\circ\iota$.

\begin{proposition}\label{prop:Hecke_action_on_PT}
    For each $\mu\in\mathcal{D}'(\Omega)$ we have
\begin{align*}
\el(\ep{z}\mu)&=0\cdot(\mc P_z\mu\circ \iota) + z\cdot\ep{z}\mu\\
\Xi(\ep{z}\mu)&=z^{-1}\cdot(\mc P_z\mu\circ \iota) - z^{-1}\ep{z}\mu\\
\Xi(\mathcal{P}_z\mu\circ \iota)&=z^{-1}\cdot(\mc P_z\mu\circ \iota) + (z-z^{-1})\cdot\ep{z}\mu\\
\el(\mathcal{P}_z\mu\circ \iota)&=qz^{-1}\cdot(\mc P_z\mu\circ \iota) + q(z-z^{-1})\cdot\ep{z}\mu.
\end{align*}
\end{proposition}

\begin{proof}
By the reasoning above there exist constants $a,b\in\mathbb{C}$ independent of $\mu$ such that
\begin{gather*}
\el(\ep{z}\mu)=a\cdot(\mc P_z\mu\circ \iota) + b\cdot\ep{z}\mu.
\end{gather*}
Choosing $\mu=\delta_{\omega_{\pm}}$ and evaluating both sides at $\vec{e}_1$, Lemma \ref{la:actions_on_delta} yields:
\begin{gather*}
    z=a+b\text{ for }\mu=\delta_{\omega_{+}}\quad\text{and}\quad 0=a\text{ for }\mu=\delta_{\omega_{-}}.
\end{gather*}
Thus, $a=0$ and $b=z$. Using the same reasoning we obtain $0 = a+b$ and $z^{-1}=a$ for $\Xi(\ep{z}\mu)$ and $z = a+b$ and $z^{-1}=a$ for $\Xi(\mathcal{P}_z\mu\circ \iota)$. The equation for $\el(\mathcal{P}_z\mu\circ \iota)$ again follows from the previous ones since $\el=q\Xi$ on $\Theta(C(G\times_KV_0))\ni \mathcal{P}_z\delta_{\omega_{\pm}}\circ\iota$.
\end{proof}

\begin{remark}\label{rem:rel_Hecke_vs_geometric}
    Note that the first equation from Proposition \ref{prop:Hecke_action_on_PT} shows that the image of the edge Poisson transform is contained in the $z$-eigenspace of $\el$ (for regular trees). Moreover, the second equation motivates Lemma~\ref{la:eigenspace_characterization1} and the third one agrees with Equation~\eqref{eq:rel_PTs}. Since $\Sigma=\Xi\el+\mathrm{Id}$ Proposition~\ref{prop:Hecke_action_on_PT} also describes the action of $\Sigma$:
\begin{gather*}
    \Sigma(\ep{z}\mu)=(\Xi\el+\mathrm{Id})(\ep{z}\mu)=\Xi z\ep{z}\mu+\ep{z}\mu=\mc P_z\mu\circ \iota - \ep{z}\mu + \ep{z}\mu=\mc P_z\mu\circ \iota\\
    \Sigma(\mathcal{P}_z\mu\circ \iota)=(\Xi\el+\mathrm{Id})(\mathcal{P}_z\mu\circ \iota)=(q\Xi^2+\mathrm{Id})(\mathcal{P}_z\mu\circ \iota)=(q+1)(\mathcal{P}_z\mu\circ \iota).
\end{gather*}
This in particular shows
\begin{align*}
    \Sigma\Xi(\mathcal{P}_{z}\mu\circ\iota)&=\Sigma(z^{-1}(\mathcal{P}_{z}\mu\circ\iota)+(z-z^{-1})\ep{z}\mu)=(q+1)z^{-1}(\mathcal{P}_{z}\mu\circ\iota)+(z-z^{-1})(\mathcal{P}_{z}\mu\circ\iota)\\
&=(z+qz^{-1})(\mathcal{P}_{z}\mu\circ\iota),
\end{align*}
where $\Sigma\Xi(\mathcal{P}_{z}\mu\circ\iota)=(q+1)(\Delta\mathcal{P}_{z}\mu)\circ\iota$ acts like the scalar Laplacian.
\end{remark}

\subsection{Poisson transform attached to \texorpdfstring{$V_1$}{V1}}
In this final subsection we investigate the Poisson transform $\mc P_z^{V_1}\coloneqq \mc P_z^{\pi,\mathrm{pr}_{V_1}}$ where $\mathrm{pr}_{V_1}\in\operatorname{Hom}_K^{\mathrm{cont}}(H_z,V)$ denotes the projection onto $V_1$, i.e.
\begin{gather*}
    \mathrm{pr}_{V_1}\colon H_z\to V_1,\quad T\mu(kK_1)\coloneqq \mu(\chi_{kK_1})-\frac{1}{q+1}\mu(\mathbbm{1}).
\end{gather*}
By Proposition~\ref{prop:PT_equivariance} the Hecke algebra $\mathcal{H}(G,K;V_1)$ acts equivariantly on $\mc P_z^{V_1}$. We will now describe the structure of this Hecke algebra and its action. As before, we study the spaces $\mathrm{End}(V_1)^{K\times K}_n$ for $n\leq 0$. Note that the only linear combinations of elements in $\{\delta_{\vec{e}}\mid \iota(\vec{e})=o\}$ that are invariant under $S(\iota^{-1}(o))\setminus\{(o,x_{-1})\}$ and are contained in $V_1$ are multiples of $\delta_{(o,x_{-1})}-\frac{1}{q}\sum_{\vec{e}\in\iota^{-1}(o)\setminus\{(o,x_{-1})\}}\delta_{\vec{e}}$. Using Proposition~\ref{prop:Hecke_perm} we infer that for $n<0$, $\mathrm{End}(V_1)^{K\times K}_n$ is spanned by
\begin{gather*}
    T(f)(\vec{e})\coloneqq
\begin{cases}
    -qf((o,x_{-1}))&\colon \vec{e}=\vec{e}_1\\
    f((o,x_{-1}))&\colon \vec{e}\neq\vec{e}_1
\end{cases}.
\end{gather*}
Moreover, the space $\mathrm{End}(V_1)^{K\times K}_0$ consists of multiples of the identity. Note that we may write
\begin{gather*}
    T=\left(T_1-qT_3-\frac{1}{q+1}(T_2-qT_4)\right)\big|_{V_1}.
\end{gather*}
For each $F\in C(G\times_KV_1)$, $F\ast\Phi_{-n,T}$ is thus acting as
\begin{align*}
    &\phantom{{}={}}\Theta^{-1}\circ\left((\el)^n-q(\el)^{n-1}\Xi-\frac{1}{q+1}(\Sigma(\el)^n-q\Sigma(\el)^{n-1}\Xi)\right)\circ\Theta F\\
    &=\Theta^{-1}\circ\left((\el)^{n-1}(\el-q\Xi)-\frac{1}{q+1}\Sigma(\el)^{n-1}(\el-q\Xi)\right)\circ\Theta F\\
    &=\Theta^{-1}\circ \left(\mathrm{Id}-\frac{1}{q+1}\Sigma\right)\circ(\el)^{n-1}\circ (\el-q\Xi)\circ\Theta F.
\end{align*}
Note that $\Theta^{-1}\circ (\mathrm{Id}-\frac{1}{q+1}\Sigma)\circ \Theta$ acts by $0$ on $C(G\times_KV_0)$ and as the identity on $C(G\times_KV_1)$. Thus, it is given by the projection $C(G\times_KV)\to C(G\times_KV_1)$. Moreover, $\el-q\Xi=-(q+1)\Xi(\mathrm{Id}-\frac{1}{q+1}\Sigma)$.

\begin{proposition}
    $X_n\coloneqq(\mathrm{Id}-\frac{1}{q+1}\Sigma)\circ(\el)^{n}\circ (\el-q\Xi)$ is a polynomial in $X_0$.
\end{proposition}

\begin{proof}
    Note first that $X_0=(\mathrm{Id}-\frac{1}{q+1}\Sigma)(\el-q\Xi)$ and, since $(\el-q\Xi)\Sigma=0$, we have $X_0^n=(\mathrm{Id}-\frac{1}{q+1}\Sigma)(\el-q\Xi)^n$. We have
\begin{align}
\nonumber\el(\el-q\Xi)&=(\el-q\Xi)^2+q\Xi(\el-q\Xi)=(\el-q\Xi)^2+q(\Sigma-(q+1)\mathrm{Id})\\
&=(\el-q\Xi)^2-q(q+1)\left(\mathrm{Id}-\frac{1}{q+1}\Sigma\right)\label{eq:tmp_V1}.
\end{align}
Furthermore,
\begin{gather*}
    \el\left(\mathrm{Id}-\frac{1}{q+1}\Sigma\right)=\el-\frac{q}{q+1}(\el+\Xi)=\frac{1}{q+1}(\el-q\Xi).
\end{gather*}
Combining these two equations, we infer for $n\geq 2$
\begin{align*}
X_n&=\left(\mathrm{Id}-\frac{1}{q+1}\Sigma\right)\circ(\el)^{n-1}\circ \left((\el-q\Xi)^2-q(q+1)\left(\mathrm{Id}-\frac{1}{q+1}\Sigma\right)\right)\\
&=X_{n-1}(\el-q\Xi)-q(q+1)\left(\mathrm{Id}-\frac{1}{q+1}\Sigma\right)(\el)^{n-2}\frac{1}{q+1}(\el-q\Xi)\\
&=X_{n-1}(\el-q\Xi)-qX_{n-2}.
\end{align*}
Iterating this equation shows that each $X_n$ can be written as a linear combination of $X_1(\el-q\Xi)^{k}$ and $X_0(\el-q\Xi)^\ell$. But $X_0(\el-q\Xi)^{\ell}=X_0^{\ell+1}$ and, using \eqref{eq:tmp_V1},
\begin{align*}
X_1(\el-q\Xi)^{k}&=\left(\mathrm{Id}-\frac{1}{q+1}\Sigma\right)\el(\el-q\Xi)^{k+1}\\
&=\left(\mathrm{Id}-\frac{1}{q+1}\Sigma\right)(\el-q\Xi)^{k+2}-q(q+1)\left(\mathrm{Id}-\frac{1}{q+1}\Sigma\right)(\el-q\Xi)^{k}\\
&=X_0^{k+2}-q(q+1)X_0^{k}.\qedhere
\end{align*}
\end{proof}

\begin{proposition}
The action of $X_0$ on the image of $\mc P_z^{V_1}$ is given by
\begin{gather*}
    (\Theta^{-1}\circ X_0\circ \Theta) \mc P_z^{V_1}\mu = (z+qz^{-1})\mc P_z^{V_1}\mu\qquad (\mu\in\mathcal{D}'(\Omega)).
\end{gather*}
\end{proposition}

\begin{proof}
Note that $\Theta\circ \mc P_z^{V_1}=\ep{z}-\frac{1}{q+1}\mathcal{P}_{z}\circ\iota$. Moreover, using Proposition \ref{prop:Hecke_action_on_PT},
\begin{align*}
(\el-q\Xi)\left(\ep{z}\mu-\frac{1}{q+1}\mathcal{P}_{z}\mu\circ\iota\right)&=z\ep{z}\mu-\frac{1}{q+1}\el(\mathcal{P}_{z}\mu\circ\iota)-q\Xi\ep{z}\mu+\frac{q}{q+1}\Xi(\mathcal{P}_{z}\mu\circ\iota)\\
&=z\ep{z}\mu-q\Xi\ep{z}\mu\\
&=(z+qz^{-1})\ep{z}\mu-qz^{-1}\mathcal{P}_{z}\mu\circ\iota.
\end{align*}
But since $\mathrm{Id}-\frac{1}{q+1}\Sigma$ fixes $\Theta\circ P_{z}^{V_1}\mu$ and annihilates $\mathcal{P}_{z}\mu\circ\iota$,
\begin{align*}
 X_0\left(\ep{z}\mu-\frac{1}{q+1}(\mathcal{P}_{z}\mu\circ\iota)\right)&=\left(\mathrm{Id}-\frac{1}{q+1}\Sigma\right)((z+qz^{-1})\ep{z}\mu-qz^{-1}(\mathcal{P}_{z}\mu\circ\iota))\\
&=(z+qz^{-1})\left(\ep{z}\mu-\frac{1}{q+1}(\mathcal{P}_{z}\mu\circ\iota)\right).\qedhere
\end{align*}
\end{proof}

\bibliographystyle{amsalpha}
\bibliography{Literatur}

\end{document}